\documentclass[11pt,a4paper]{amsart}
\usepackage[a4paper,top=2.5cm,bottom=2.5cm,left=3cm,right=3cm,heightrounded]{geometry}
\usepackage{microtype}
\usepackage[T1]{fontenc}
\usepackage{lmodern}
\usepackage{xcolor}

\usepackage{mathtools,amssymb,amsthm}
\numberwithin{equation}{section}

\usepackage{graphicx}
\graphicspath{{./}{./figures/}}
\usepackage{caption}

\usepackage{booktabs}
\usepackage{array}
\usepackage{enumitem}

\usepackage{tikz}
\usetikzlibrary{arrows.meta}

\usepackage{float}
\usepackage{placeins}
\usepackage{needspace}

\floatstyle{ruled}
\newfloat{algorithm}{tbp}{loa}
\floatstyle{plain}
\floatname{algorithm}{Algorithm}
\newlength{\algorithmicindent}
\makeatletter
\newcounter{algorithmicdepth}
\newenvironment{algorithmic}[1][]{
  \setcounter{algorithmicdepth}{0}
  \begin{list}{}{
    \setlength{\leftmargin}{0pt}
    \setlength{\itemindent}{0pt}
    \setlength{\itemsep}{1pt}
    \setlength{\parsep}{0pt}
    \setlength{\topsep}{2pt}}
}{\end{list}}
\newcommand{\ALG@item}{
  \item[]\hspace*{\dimexpr\value{algorithmicdepth}\algorithmicindent\relax}}
\newcommand{\State}{\ALG@item}
\newcommand{\Statex}{\ALG@item}
\newcommand{\algorithmicrequire}{\textbf{Require:}}
\newcommand{\algorithmicensure}{\textbf{Ensure:}}
\newcommand{\Require}{\ALG@item\algorithmicrequire\ }
\newcommand{\Ensure}{\ALG@item\algorithmicensure\ }
\newcommand{\If}[1]{
  \ALG@item\textbf{if}\ #1\ \textbf{then}\stepcounter{algorithmicdepth}}
\newcommand{\EndIf}{\addtocounter{algorithmicdepth}{-1}}
\newcommand{\Else}{\addtocounter{algorithmicdepth}{-1}\ALG@item\textbf{else}\stepcounter{algorithmicdepth}}
\newcommand{\For}[1]{
  \ALG@item\textbf{for}\ #1\ \textbf{do}\stepcounter{algorithmicdepth}}
\newcommand{\ForAll}[1]{
  \ALG@item\textbf{for each}\ #1\ \textbf{do}\stepcounter{algorithmicdepth}}
\newcommand{\EndFor}{\addtocounter{algorithmicdepth}{-1}}
\newcommand{\Return}{\ALG@item\textbf{return}\ }
\newcommand{\Comment}[1]{\hfill$\triangleright$~\textit{#1}}
\makeatother

\usepackage{aliascnt}

\newtheorem{theorem}{Theorem}[section]
\newtheorem{maintheorem}{Theorem}

\newaliascnt{proposition}{theorem}
\newtheorem{proposition}[proposition]{Proposition}
\aliascntresetthe{proposition}

\newaliascnt{lemma}{theorem}
\newtheorem{lemma}[lemma]{Lemma}
\aliascntresetthe{lemma}

\newaliascnt{corollary}{theorem}
\newtheorem{corollary}[corollary]{Corollary}
\aliascntresetthe{corollary}

\newaliascnt{fact}{theorem}

\aliascntresetthe{fact}

\theoremstyle{definition}
\newaliascnt{definition}{theorem}
\newtheorem{definition}[definition]{Definition}
\aliascntresetthe{definition}
\newaliascnt{example}{theorem}
\newtheorem{example}[example]{Example}
\aliascntresetthe{example}

\theoremstyle{remark}
\newaliascnt{remark}{theorem}

\aliascntresetthe{remark}



\makeatletter
\renewcommand{\subsection}{
  \@startsection{subsection}{2}{\z@}
    {1.5ex \@plus 1ex \@minus .2ex}
    {.6ex \@plus .2ex}
    {\normalfont\bfseries}
}
\makeatother

\usepackage{cite}
\usepackage[hidelinks]{hyperref}
\usepackage{bookmark}
\DeclareRobustCommand{\doi}[1]{\href{https://doi.org/#1}{\nolinkurl{doi:#1}}}
\DeclareRobustCommand{\arxiv}[1]{\href{https://arxiv.org/abs/#1}{\nolinkurl{arXiv:#1}}}
\DeclareRobustCommand{\weblink}[2]{\href{#2}{\nolinkurl{#1}}}
\usepackage{orcidlink}
\usepackage[capitalise,noabbrev]{cleveref}

\crefname{theorem}{Theorem}{Theorems}
\Crefname{theorem}{Theorem}{Theorems}
\crefname{maintheorem}{Theorem}{Theorems}
\Crefname{maintheorem}{Theorem}{Theorems}
\crefname{proposition}{Proposition}{Propositions}
\Crefname{proposition}{Proposition}{Propositions}
\crefname{lemma}{Lemma}{Lemmas}
\Crefname{lemma}{Lemma}{Lemmas}
\crefname{fact}{Fact}{Facts}
\Crefname{fact}{Fact}{Facts}
\crefname{corollary}{Corollary}{Corollaries}
\Crefname{corollary}{Corollary}{Corollaries}
\crefname{definition}{Definition}{Definitions}
\Crefname{definition}{Definition}{Definitions}
\crefname{example}{Example}{Examples}
\Crefname{example}{Example}{Examples}
\crefname{remark}{Remark}{Remarks}
\Crefname{remark}{Remark}{Remarks}
\crefname{section}{Section}{Sections}
\Crefname{section}{Section}{Sections}
\crefname{figure}{Figure}{Figures}
\Crefname{figure}{Figure}{Figures}
\crefname{table}{Table}{Tables}
\Crefname{table}{Table}{Tables}
\renewcommand{\algorithmicrequire}{\textbf{Input:}}
\renewcommand{\algorithmicensure}{\textbf{Output:}}
\numberwithin{algorithm}{section}
\crefname{algorithm}{Algorithm}{Algorithms}
\Crefname{algorithm}{Algorithm}{Algorithms}

\hypersetup{
  pdftitle={Semialgebraic interior layers in root closures of bounded-height integer polynomials},
  pdfauthor={Bernat Espigul\'e, David Juher},
  pdfsubject={Mathematics; Topological Dynamics, Fractal Geometry, Algebraic Number Theory},
  pdfkeywords={roots of integer polynomials, zeros of power series, connectedness loci, iterated function systems, semialgebraic sets}
}

\newcommand{\C}{\mathbb{C}}
\newcommand{\R}{\mathbb{R}}
\newcommand{\Z}{\mathbb{Z}}

\newcommand{\D}{\mathbb{D}}
\newcommand{\Dbar}{\overline{\D}}

\newcommand{\Mn}{\mathcal{M}_n}
\newcommand{\Xn}{\mathcal{X}_n}

\newcommand{\PTrap}{\mathcal{C}}
\newcommand{\PEnc}{\mathcal{E}}

\newcommand{\Dn}{D_n}
\newcommand{\Rn}{\mathcal{R}_n}

\newcommand{\Real}{\operatorname{Re}}
\newcommand{\Imag}{\operatorname{Im}}
\newcommand{\interior}{\operatorname{int}}
\newcommand{\closure}[1]{\overline{#1}}

\DeclareMathOperator{\dist}{dist}

\newcommand{\lv}{\ell_v}
\newcommand{\ls}{\ell_s}
\newcommand{\VE}{\mathcal{V}_{\mathcal{E}}}
\newcommand{\SE}{\mathcal{S}_{\mathcal{E}}}
\newcommand{\Av}{\mathcal A_v}

\newcommand{\Hutch}{\mathcal{H}}

\newcommand{\Interior}{\textup{\textsc{Interior}}}
\newcommand{\Exterior}{\textup{\textsc{Exterior}}}
\newcommand{\Undetermined}{\textup{\textsc{Undetermined}}}
\newcommand{\emptyword}{\varepsilon}

\newcommand{\DrawTrapP}[4][]{
  \begingroup
  \pgfmathsetmacro{\cx}{#2}
  \pgfmathsetmacro{\cy}{#3}
  \pgfmathsetmacro{\nnum}{#4}
  \pgfmathsetmacro{\NNum}{2*\nnum - 1}
  \pgfmathsetmacro{\rho}{sqrt(\cx*\cx + \cy*\cy)}
  \pgfmathsetmacro{\xabs}{abs(\cx)}
  \pgfmathsetmacro{\yabs}{abs(\cy)}
  \pgfmathsetmacro{\Vpar}{(\NNum - 2*\xabs)*\yabs/(\rho*\rho)}
  \pgfmathsetmacro{\Spar}{\NNum * \yabs/\rho}
  \pgfmathsetmacro{\invY}{1/(\cy)}
  \coordinate (C1) at ({(\rho*\invY)*(-\Spar) - (\cx*\invY)*(-\Vpar)}, {-\Vpar});
  \coordinate (C2) at ({(\rho*\invY)*(\Spar) - (\cx*\invY)*(-\Vpar)}, {-\Vpar});
  \coordinate (C3) at ({(\rho*\invY)*(\Spar) - (\cx*\invY)*(\Vpar)}, {\Vpar});
  \coordinate (C4) at ({(\rho*\invY)*(-\Spar) - (\cx*\invY)*(\Vpar)}, {\Vpar});
  \filldraw[#1] (C1) -- (C2) -- (C3) -- (C4) -- cycle;
  \endgroup
}

\newcommand{\DrawPone}[4][]{
  \begingroup
  \pgfmathsetmacro{\cx}{#2}
  \pgfmathsetmacro{\cy}{#3}
  \pgfmathsetmacro{\nnum}{#4}
  \pgfmathsetmacro{\NNum}{2*\nnum - 1}
  \pgfmathsetmacro{\rhosq}{(\cx*\cx + \cy*\cy)}
  \pgfmathsetmacro{\aoneoverc}{\cx/(\rhosq)}
  \pgfmathsetmacro{\boneoverc}{-\cy/(\rhosq)}
  \pgfmathtruncatemacro{\M}{(\NNum - 1)/2}
  \foreach \s in {-\M,...,\M}{
    \pgfmathtruncatemacro{\t}{2*\s}
    \begin{scope}[shift={(\t,0)}, cm={\aoneoverc,\boneoverc,-\boneoverc,\aoneoverc,(0,0)}]
      \DrawTrapP[#1]{#2}{#3}{#4}
    \end{scope}
  }
  \endgroup
}

\makeatletter
\newcommand{\DrawPlevel}[5][]{
  \begingroup
  \pgfmathtruncatemacro{\kk}{#2}
  \expandafter\DrawPlevel@\expandafter{\kk}{#1}{#3}{#4}{#5}
  \endgroup
}
\newcommand{\DrawPlevel@}[5]{
  \ifnum#1=0\relax
    \DrawTrapP[#2]{#3}{#4}{#5}
  \else
    \pgfmathsetmacro{\NNum}{2*#5 - 1}
    \pgfmathsetmacro{\rhosq}{(#3*#3 + #4*#4)}
    \pgfmathsetmacro{\aoneoverc}{#3/(\rhosq)}
    \pgfmathsetmacro{\boneoverc}{-(#4)/(\rhosq)}
    \pgfmathtruncatemacro{\M}{(\NNum - 1)/2}
    \foreach \s in {-\M,...,\M}{
      \pgfmathtruncatemacro{\t}{2*\s}
      \begin{scope}[shift={(\t,0)}, cm={\aoneoverc,\boneoverc,-\boneoverc,\aoneoverc,(0,0)}]
        \expandafter\DrawPlevel@\expandafter{\numexpr#1-1\relax}{#2}{#3}{#4}{#5}
      \end{scope}
    }
  \fi
}
\makeatother

\title[Semialgebraic interior layers in polynomial root closures]{Semialgebraic interior layers in root closures of bounded-height integer polynomials}

\author[B. Espigul\'e]{Bernat Espigul\'e\,\orcidlink{0000-0002-8036-7851}}
\thanks{Corresponding author: \href{mailto:bernat@espigule.com}{bernat@espigule.com}.}

\author[D. Juher]{David Juher\,\orcidlink{0000-0001-5440-1705}}
\address[B. Espigul\'e and D. Juher]{Department of Computer Science, Applied Mathematics and Statistics, Universitat de Girona, 17003 Girona, Catalonia, Spain}
\email[B. Espigul\'e]{bernat@espigule.com}
\email[D. Juher]{david.juher@udg.edu}


\date{}

\ifdefined\pdfinfoomitdate\pdfinfoomitdate=1\fi
\ifdefined\pdfsuppressptexinfo\pdfsuppressptexinfo=15\fi
\ifdefined\pdftrailerid\pdftrailerid{}\fi
\hypersetup{pdfcreator={},pdfproducer={},pdfsubject={}}
\begin{document}

\begin{abstract}
Finite-degree root plots of bounded-height monic integer
polynomials can exhibit gaps around interior points of their
root closures. Outside the closed unit disk, these closures
coincide with Mandelbrot sets for affine iterated function
systems with collinear digits. We construct explicit semialgebraic
interior layers, certified when finite inverse orbits enter
canonical self-covering parallelograms. Such parallelograms exist
exactly in the non-real lens, which contains the non-real locus precisely
from twenty maps onward.
Within the non-real lens, limits of each capture
level are captured within two additional steps, optimally,
and the layers recover the locus by closure. The layers expose
persistent interior deficits: explicit quadratic parameters
evade every finite level; sufficiently close distinct
parameters have logarithmically growing capture depths;
the underlying lattice obstruction yields degree-dependent root-free
radii, explaining such gaps. For non-real parameters of modulus at least the square
root of the number of maps, exact enclosure trees have sharp,
alphabet-independent branching bounds.
\end{abstract}

\maketitle
\vspace{-1.5\baselineskip}

\section{Introduction}
Fix a bound on the absolute values of the coefficients of monic integer polynomials, but allow
arbitrary degrees. We study their roots and limit points outside the closed unit disk. A gap in
a plot up to a fixed degree may disappear at higher degrees. How can one certify interior points
of this root closure?

The same set describes connectedness in a family of self-similar sets built from smaller copies
placed at equally spaced points on a line. Each set is either connected or totally disconnected.
The parameters giving connected sets form the \emph{connectedness locus}, also called the
\emph{Mandelbrot set} of the family~\cite{Barnsley1985AMaps,Bandt2002OnMaps,SolomyakXu2003}.
We use this correspondence to construct nested open semialgebraic capture levels and to bound
the inverse searches yielding their finite certificates. These levels also expose persistent
interior deficits: for every $n\ge3$, explicit quadratic interior parameters evade all finite
levels, although all sufficiently close distinct parameters are captured at logarithmically
increasing depths (\Cref{prop:quadratic-noncapture}). The underlying lattice obstruction gives an
explicit root-free radius for each degree bound~\eqref{eq:explicit-degree-radius}. Thus a
finite-degree root plot may show a gap entirely within the connectedness locus.

The binary locus was introduced by Barnsley and Harrington~\cite{Barnsley1985AMaps}.
Bousch proved connectedness and local connectedness~\cite{Bousch1992Thesis,Bousch1993ConnexiteFonctions}.
Bandt proved the existence of a hole and conjectured that the interior is dense away from the real
axis~\cite{Bandt2002OnMaps}. Solomyak and Xu proved this in a region near the imaginary
axis~\cite{SolomyakXu2003}; Calegari, Koch and Walker proved the full conjecture using
topological traps~\cite{Calegari2017RootsConjecture}. Our earlier work extended the root-closure
description and the covering region to arbitrary numbers of maps~\cite{EJSn24}.
Finite capture was subsequently used to prove infinitely many holes for every $n\ge3$~\cite{E26holes},
extending the binary result of~\cite{Calegari2017RootsConjecture}.

The dichotomy between connected and totally disconnected attractors has a random counterpart in
Mandelbrot's fractal percolation, whose limit set is almost surely totally disconnected below a
critical probability and crosses the unit square with positive probability from that value
on~\cite{ChayesChayesDurrett1988}. Connectedness loci, such as the Mandelbrot set of quadratic
polynomials and the loci studied here, describe deterministic versions of this transition.

In his posthumous paper~\cite{Thurston2014Entropy}, Thurston characterized the entropies of
postcritically finite interval maps and studied the fractal set formed by the Galois conjugates of
the growth rates of postcritically finite real quadratic maps. For superattracting real quadratic
maps, Tiozzo~\cite{Tiozzo2020Galois} proved that, inside the unit disk, the closure of these
conjugates coincides with the closure of the zeros $\lambda$ of polynomials with coefficients $\pm1$.
Written in $c=1/\lambda$, this \emph{Thurston set} lies in the binary
locus~\cite{SilvestriPerez2023Accessibility}.
For the binary family, Silvestri and P\'erez~\cite{SilvestriPerez2023Accessibility} give sufficient
conditions for boundary accessibility via asymptotic similarity. Calegari and
Walker~\cite{CalegariWalker2026Laminations} study filled-attractor boundary dynamics using
external angles and laminations.

Collinear digits also arise in numeration: the usual alphabet
$0,\ldots,n-1$ extends positional notation to a complex base $c$, linking these attractors to
beta-expansions and self-affine tiles~\cite{Renyi1957,Parry1960,KiratLau2000,Akiyama2021CollinearTiles}.
We use an affinely equivalent centered alphabet to keep the difference sets and covering inequalities
in the same notation.

\Needspace{10\baselineskip}
\subsection{Roots and connectedness}
Fix $n\ge2$, put $N=2n-1$, and take $|c|>1$. The contraction multiplier is $\lambda=1/c$
and the similarity ratio is $|c|^{-1}$. In the variable $\lambda$
of~\cite{Barnsley1985AMaps,Bandt2002OnMaps,SolomyakXu2003}, the connectedness locus lies inside
the unit disk.
Let $A_n=\{-n+1,-n+3,\ldots,n-1\}$ and consider the iterated function system (IFS)
$\{f_t\}_{t\in A_n}$, where $f_t(z)=t+z/c$. Its unique nonempty compact attractor
$E(c,n)$~\cite{Hutchinson1981FractalsMeasures} admits the series representation
\begin{equation}\label{alphabet}
 E(c,n)=\left\{\sum_{j=0}^{\infty}a_jc^{-j}:a_j\in A_n\right\}.
\end{equation}
The sets $f_t(E(c,n))$ are its \emph{first-level pieces}; a representing digit sequence is an
\emph{address}. Every equally spaced collinear alphabet of cardinality $n$ gives rise to the
same connectedness locus (\Cref{prop:affine-universality}). We adopt $A_n$ for computational
convenience, since
\[
 E(c,n)-E(c,n)=E(c,N).
\]
Write $\Mn=\{c:E(c,n)\text{ is connected}\}$ and
\begin{equation}\label{defRn}
 \Rn=\left\{c\in\C:c^k+\sum_{j=1}^k d_jc^{k-j}=0,
 \quad d_j\in D_n=\{-n+1,\ldots,n-1\},\ k\ge1\right\}.
\end{equation}
Thus $\Rn$ is the root set of monic integer polynomials of \emph{height} (maximum absolute
coefficient) at most $n-1$. Related sets arise from restricted-coefficient
polynomials~\cite{OdlyzkoPoonen1993,BorweinErdelyiLittmann2008} and
power series~\cite{Beaucoup1998}. The algebraic starscapes of Harriss, Stange and
Trettel~\cite{HarrissStangeTrettel2022} illustrate the complementary regime, in which the degree
is bounded and the height is allowed to grow. Our two basic characterizations are
\begin{equation}\label{eq:intro-root-closure}
 \Mn=\overline{\Rn}\cap\{|c|>1\},\qquad
 c\in\Mn\quad\Longleftrightarrow\quad 2c\in E(c,N).
\end{equation}
We call $2c$ the \emph{marked point}. Section~\ref{sec:prelim} proves these identities and shows
that neighboring first-level overlap implies connectedness; otherwise the attractor is totally
disconnected. For the attractor $K(c,n)$ with digits $0,\ldots,n-1$, the criterion is
$c\in K(c,n)-K(c,n)$, as in Solomyak and Xu's binary
formulation~\cite{SolomyakXu2003}; Appendix~\ref{supp:foundations} gives the affine conversion.

We write $\D=\{|z|<1\}$ and $\Dbar=\{|z|\le1\}$. Parameter closures, interiors, and boundaries
are relative to $\C\setminus\Dbar$, except that $\overline{\Rn}$ denotes Euclidean closure.
Dynamical-plane closures are Euclidean.

The next proposition describes the root closure on both sides of the
unit circle and allows non-monic polynomials of the same height. We use the
projective line to include $0$ and $\infty$ in the same statement.
\begin{proposition}[Projective root closure]\label{prop:projective-root-closure}
For $h\ge1$, let $Z_h\subset\mathbb P^1(\C)$ be the union of the zero sets $V(F)$ of the nonzero
binary forms $F\in\Z[X,Y]$ of height at most~$h$. Identify $z\in\C$ with $[z:1]$, put
$\infty=[1:0]$, and let $\iota[X:Y]=[Y:X]$, so that $\iota(z)=1/z$. Then, with closures taken
in $\mathbb P^1(\C)$,
\[
 \overline{Z_{n-1}}=\Mn\cup\iota(\Mn)\cup\partial\D\cup\{0,\infty\},
 \qquad\text{and}\qquad
 \overline{Z_{n-1}}\cap\C=\overline{\Rn}.
\]
\end{proposition}
Thus the roots of all nonzero integer polynomials of height at most $n-1$, monic or not and inside or
outside the unit disk, have closure $\overline{\Rn}$. The symmetry $\iota$ carries the
open semialgebraic capture levels constructed below to interior layers inside the unit disk,
where the locus is written in the variable~$\lambda$. The proof is in
Section~\ref{sec:prelim}.

Figure~\ref{fig:roots-n4} gives finite-address proofs of connectedness outside the canonical
lens: the first-quadrant root $c_4$ of $z^3-3z^2+z+3$ satisfies
$2c_4=6-2/c_4-6/c_4^2$. More generally, a root of a degree-$k$ polynomial as in \eqref{defRn}
gives a finite digit expansion of $2c$ using at most $k$ digits from $A_N$.

\begin{figure}[p]
\centering
\includegraphics[width=.97\linewidth]{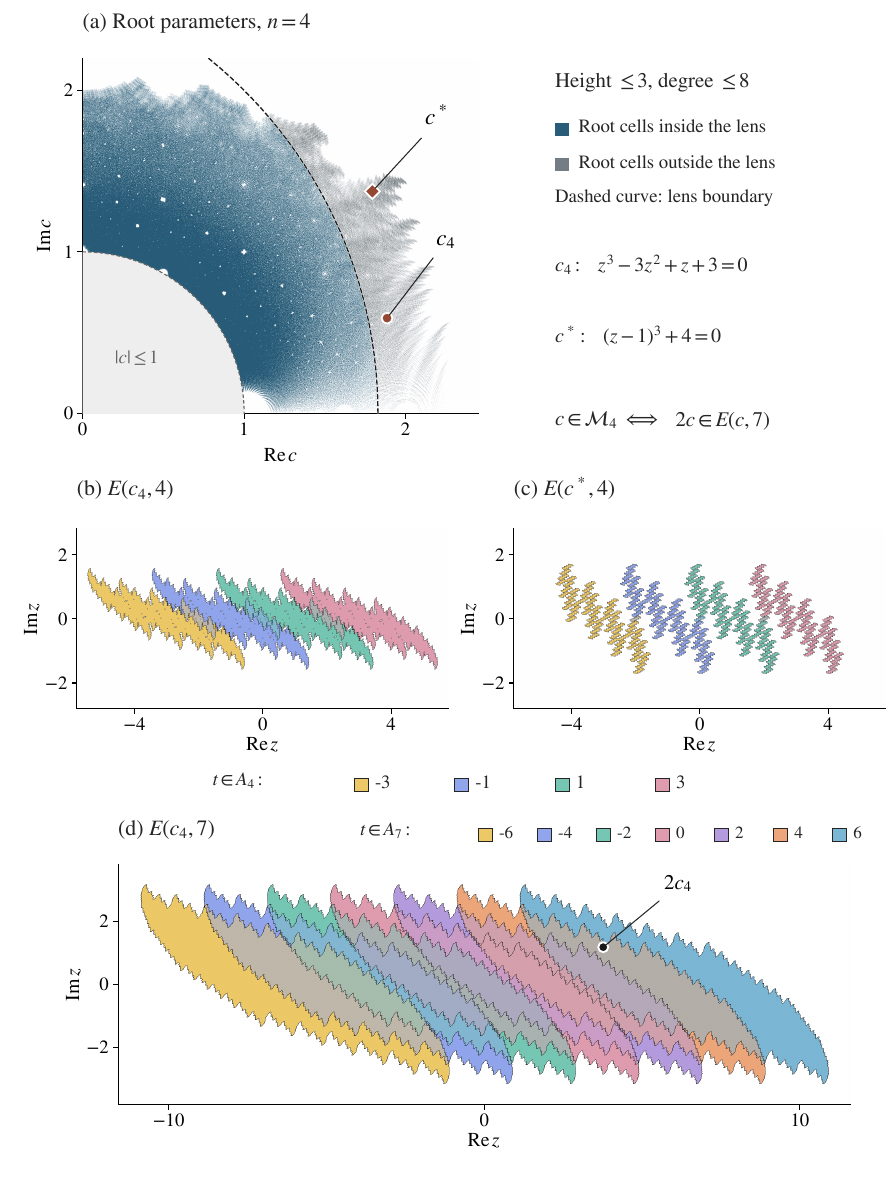}
\caption{Roots and first-level pieces for $n=4$. (a) Pixels containing numerically computed roots of monic integer polynomials of height at most three and degrees $1$--$8$: blue inside the lens, gray outside. (b,c) Attractors at the first-quadrant roots $c_4,c^*$ of $z^3-3z^2+z+3$ and $(z-1)^3+4$. (d) Difference attractor with $2c_4$ marked. Piece colors blend on overlaps; numerical contours use escape depth $24$ (Appendix~\ref{app:computations}). The finite addresses $2c_4=6-2/c_4-6/c_4^2$ and $2c^*=6-6/c^*-6/(c^*)^2$ prove connectedness, although both parameters lie outside the lens $\mathcal X_4$ of \eqref{eq:intro-lens}.}
\label{fig:roots-n4}
\end{figure}

\subsection{Canonical capture levels}
To find open subsets of the root closure, we seek finite certificates for the condition
$2c\in E(c,N)$. We undo the contractions one digit at a time, using their inverses
$g_t(z)=c(z-t)$, $t\in A_N$. The aim is to reach an open region from which a bounded
inverse orbit can be continued indefinitely. A bounded open set $P$ is \emph{self-covering} if
\[
 P\subset\bigcup_{t\in A_N}\left(t+\frac Pc\right).
\]
Once an inverse orbit enters $P$, self-covering supplies at every later step a digit that
returns it to $P$; the bounded orbit gives an address of the starting point in $E(c,N)$.
For the canonical parallelogram below, entry persists under small changes of $c$, so finite
entry of $2c$ proves that $c$ is an interior point of $\Mn$.
This covering method appears in Indlekofer, J\'arai and K\'atai~\cite{Indlekofer1995},
Solomyak and Xu~\cite[Section~3]{SolomyakXu2003}, and
Hung~\cite[Section~3.3]{Hung2007PolygonFractals}; see also~\cite{ShmerkinSolomyak2006,EJSn24}.

Our canonical parallelogram has sides parallel to $1$ and $\overline c$.
These directions reflect the dynamics: the digits translate horizontally, while multiplication
by $c$ sends a line parallel to $\overline c$ to a horizontal line. For non-real parameters,
its covering domain is the lens
\begin{equation}\label{eq:intro-lens}
 \Xn=\{c:|c|>1,\ |c-1|<\sqrt{2n},\ |c+1|<\sqrt{2n}\}.
\end{equation}
\begin{maintheorem}[Parallelogram covering criterion]
\label{thm:sharp-parallelogram-lens}\label{prop:NE-feasible-lens}
Let $n\ge2$, $N=2n-1$, and $c\notin\R$ with $|c|>1$.
A nonempty bounded open self-covering parallelogram, not necessarily centered at zero, with sides
parallel to $1$ and $\overline c$
exists if and only if
\[
 |c|^2+2|\operatorname{Re}c|<N,
\]
or equivalently $c\in\Xn$.
\end{maintheorem}
For $c=x+iy\in\Xn\setminus\R$, we use
\begin{equation}\label{eq:intro-trap}
 \PTrap(c,N)=\{z:
 |\operatorname{Im}(cz)|<N|y|,\quad
 |c|^2|\operatorname{Im}z|<(N-2|x|)|y|\}.
\end{equation}
This \emph{canonical trap} contains the covering rectangle of~\cite{EJSn24}. It differs from the
topological traps of Calegari, Koch and Walker: their signed path intersections force first-level
pieces to overlap, without necessarily providing an inverse-entry region.

For a word $u=t_0\cdots t_{k-1}\in A_N^k$, let
$f_u=f_{t_0}\circ\cdots\circ f_{t_{k-1}}$ and $g_u=f_u^{-1}$; inverse digits act in their written
order. The empty word acts as the identity. For $c\in\Xn\setminus\R$, we say that $u$ \emph{captures} $c$ if
$g_u(2c)\in\PTrap(c,N)$. The \emph{word region} associated with $u$ is the set of parameters it captures:
\begin{equation}\label{wcharts}
 \mathcal U_u=\{c\in\Xn\setminus\R:g_u(2c)\in\PTrap(c,N)\}.
\end{equation}
A capture word together with rigorous bounds verifying the lens and entry inequalities
forms an \emph{interior certificate}.
The \emph{capture level} $\Theta_k(n)$ consists of the parameters captured within $k$ steps.
With $\min\varnothing=\infty$, define the levels, their union, and the \emph{first capture depth} by
\begin{equation}\label{def:levels}
 \Theta_k(n)=\bigcup_{|u|\le k}\mathcal U_u,
 \qquad \Theta_{\mathrm{fin}}(n)=\bigcup_{k\ge0}\Theta_k(n),
 \qquad \kappa_n(c)=\min\{k:c\in\Theta_k(n)\}.
\end{equation}
Self-covering and persistence of strict entry give
\[
 \Theta_k(n)\subset\Theta_{k+1}(n)\subset\operatorname{int}\Mn.
\]
Every root of a monic integer polynomial of degree at most $k$ with non-leading coefficients
in $\Dn$ that lies in the non-real lens belongs to $\Theta_k(n)$: the digits of the expansion of $2c$
form an inverse word sending $2c$ to $0\in\PTrap(c,N)$ (\Cref{prop:dyn-poly}).
The \emph{capture layers} consist of parameters of a fixed first capture depth:
\begin{equation}\label{def:strata}
 \Omega_0(n)=\Theta_0(n),\qquad
 \Omega_k(n)=\Theta_k(n)\setminus\Theta_{k-1}(n)\quad(k\ge1).
\end{equation}
Thus $\Theta_k(n)$ contains all parameters captured by depth $k$, whereas
$\Omega_k(n)$ consists of those first captured at depth $k$.

\begin{maintheorem}[Semialgebraic capture levels]
\label{thm:main-finite-level-geometry}\label{thm:semialgebraic-capture-charts}
Each word region $\mathcal U_u$ and capture level $\Theta_k(n)$ is open and semialgebraic.
Each capture layer $\Omega_k(n)$ is semialgebraic. Their relative boundaries are finite unions
of real-analytic semialgebraic arcs and finitely many points.
\end{maintheorem}
Within the non-real lens, $\Theta_k(n)$ is a union of at most $N^k$ word regions. For words of
length $k$, the boundary candidates in each quadrant have degrees $k+2$ and $k+1$.
The base level is an intersection of two disks with the closed unit disk and real axis removed
(\Cref{prop:Theta0}); one- and two-letter word regions have cubic/hyperbolic and
quartic/cubic boundaries in the non-real lens, respectively (Section~\ref{sec:finite-level-geometry}). A word capturing $c$ may be longer than $\kappa_n(c)$, so a word region can
meet several capture layers.

\begin{maintheorem}[Two-step closure inclusion]\label{thm:bounded-delay-closure}
For every $n\ge2$ and $k\ge0$,
\[
 \overline{\Theta_k(n)}\cap(\Xn\setminus\R)\subset\Theta_{k+2}(n).
\]
\end{maintheorem}
For a limit $c\in\Xn\setminus\R$ of parameters captured by a word $u$, $g_u(2c)$ may lie on the parallelogram boundary; two further steps restore strict entry.
The delay is optimal, already at depth zero (\Cref{ex:sharp-two-step}).
Since the levels are open, it follows that boundaries of finite capture levels are transient: every
such boundary point inside the non-real lens belongs to one of the next two capture layers
(\Cref{cor:boundary-depth}). A non-real parameter of the lens that is never captured lies outside
the closure of every fixed level.

\begin{maintheorem}[Recovery inside the lens]\label{thm:finite-capture-non-real}
For every $n\ge2$,
\[
 (\Mn\cap\Xn)\setminus\R
 =\overline{\Theta_{\mathrm{fin}}(n)}\cap(\Xn\setminus\R).
\]
\end{maintheorem}
\paragraph{\emph{Persistent interior deficits.}}
For every $n\ge3$, the quadratic parameter $\gamma_n=(1+i\sqrt{4n-1})/2$ lies in
$\operatorname{int}\Mn$ but in no finite capture level. All sufficiently close distinct parameters are captured, with
\[
 \kappa_n(c)=\frac{\log(1/|c-\gamma_n|)}{\log\sqrt n}+O(1).
\]
For all large $k$, the component of $\C\setminus\Theta_k(n)$ containing $\gamma_n$ therefore has
diameter comparable to $n^{-k/2}$ and lies entirely inside the true locus.
Theorem~\ref{prop:quadratic-noncapture} proves these assertions.

\subsection{The global threshold and finite search}
The lens contains the entire non-real locus exactly when $n\ge20$ (\Cref{thm:lens-regime}),
extending the $n\ge21$ result of~\cite{EJSn24}. For each $2\le n\le19$, a certified non-real
root in $\Rn$ lies outside the lens (Appendix~\ref{supp:off-lens-certificates}). Thus the same explicit
capture levels describe the whole non-real locus by closure for $n\ge20$, while for smaller
$n$ their recovery statement is restricted to the lens.

Exterior tests use the smallest closed parallelogram $\PEnc(c,N)$ in the canonical directions
containing $E(c,N)$. The \emph{enclosure tree} of $z$ consists of the inverse words whose prefix iterates of $z$ lie there;
distinct words remain distinct vertices even when endpoints coincide.
\begin{maintheorem}[Uniform branching bounds]
\label{thm:main-bounded-branching}\label{thm:uniform-branching-cap}
Let $n\ge2$ and $N=2n-1$, and let $c$ be non-real with $|c|\ge\sqrt n$. For any starting point, every vertex of the exact
enclosure tree has at most four children, and at most three if $n=2$ or $n\ge25$.
Its depth-$k$ level therefore contains at most $4^k$ words, or $3^k$ in the latter cases.
The uniform bound four and the eventual three-child threshold $25$ are sharp, even for trees rooted
at $2c$.
\end{maintheorem}
A finite search returns interior, exterior, or undetermined.
The same method tests arbitrary points in the dynamical plane
(Section~\ref{sec:dynamical-certification}).
Our experiments compare entry depths for the earlier rectangle and canonical parallelogram,
and illustrate parameter and dynamical capture layers.

Section~\ref{sec:prelim} establishes affine normalization and the correspondence between
polynomial root closures, connectedness, and inverse dynamics of the marked point.
Sections~\ref{sec:trap}--\ref{sec:global} develop capture and recovery by closure;
Sections~\ref{sec:certification}--\ref{sec:experiments} give the search and experiments.
Supporting proofs, certificates and numerical methods are in the appendices.
Table~\ref{tab:notation} collects the principal notation.

\begin{table}[htbp]
\centering\small
\caption{Principal notation, with $n\ge2$ and $N=2n-1$. Throughout, $c=x+iy$ and $\rho=|c|$.}
\label{tab:notation}
\begin{tabular}{@{}p{3.6cm}p{10.7cm}@{}}
\toprule
Symbol & Meaning\\
\midrule
$A_n$, $D_n$ & IFS digits of spacing two; consecutive integer polynomial coefficients\\
$E(c,n)$, $E(c,N)$ & Attractor and difference attractor\\
$\Mn$, $\Rn$ & Connectedness locus; monic integer polynomial roots, height $\le n-1$\\
$\Xn$, $\PTrap(c,N)$ & Canonical covering domain and open parallelogram\\
$\mathcal U_u$ & Parameters captured by the inverse word $u$\\
$\Theta_k(n)$, $\Omega_k(n)$ & Capture by depth $k$; first capture at depth $k$\\
$\Theta_{\mathrm{fin}}(n)$ & Union of all finite capture levels\\
$\kappa_n(c)$ & Minimum length of a canonical capture word\\
$\PEnc(c,N)$, $\Xi_k(n)$ & Closed outer enclosure; survival through depth $k$\\
\bottomrule
\end{tabular}
\end{table}
\FloatBarrier

\section{Connectedness, roots, and inverse words}\label{sec:prelim}
We use the centered alphabets $A_m=\{-m+1,-m+3,\ldots,m-1\}$ throughout.

\begin{proposition}[Affine normalization]\label{prop:affine-universality}
Let $B=\{b+j\omega:0\le j<m\}$, where $m\ge2$ and $\omega\ne0$, and set $\mu=b+(m-1)\omega/2$,
$s=\omega/2$. For $|c|>1$, the map
\[
 H_{B,c}(z)=sz+\mu\frac{c}{c-1}
\]
conjugates the $A_m$-IFS to the IFS with digits in $B$. Its attractor is therefore $K_B(c)=H_{B,c}(E(c,m))$, with
$K_B(c)-K_B(c)=sE(c,2m-1)$. In particular, the connectedness locus depends only on $m$.
\end{proposition}
The conjugacy calculation and the conversion to digits $0,\ldots,m-1$ are given in Appendix~\ref{supp:foundations}.

\subsection{The marked-point criterion}
The neighboring-overlap criterion for collinear digits is due to
Bousch~\cite[Chapter~1, p.~41]{Bousch1992Thesis}. For completeness, we prove that any first-level
overlap forces neighboring overlap, using the horizontal-chord property of plane continua
(Appendix~\ref{app:horizontal-chords}).

\begin{proposition}\label{prop:Mn_characterization}\label{rem:basic-symmetries}
Let $n\ge2$, $|c|>1$, and $N=2n-1$. The following are equivalent:
\begin{enumerate}[label=(\roman*),leftmargin=*]
\item $E(c,n)$ is connected;
\item two distinct first-level pieces intersect;
\item $E(c,n)/c\cap(2+E(c,n)/c)\ne\varnothing$;
\item $2c\in E(c,N)$.
\end{enumerate}
Otherwise the first-level pieces are pairwise disjoint and $E(c,n)$ is totally disconnected. The
locus $\Mn$ is invariant under $c\mapsto\overline c$ and $c\mapsto-c$.
\end{proposition}
\begin{proof}
Hopf's horizontal-chord theorem~\cite{Hopf1936Chords} implies that, for a plane continuum $K$, the
positive $a$ with $K\cap(K+a)=\varnothing$ form an additive set
(Appendix~\ref{app:horizontal-chords}). Thus overlap of any two translates $K+t$, $t\in A_n$,
forces overlap of all neighboring translates.

Write $E=E(c,n)$. Clearly (i) implies (ii). If two first-level pieces intersect, iterate
$H_{j+1}=\bigcup_{t\in A_n}(t+H_j/c)$ from the disk
$H_0=\{z:|z|\le(n-1)|c|/(|c|-1)\}$. These compact sets decrease to $E$ and are connected
by induction: the assumed overlap persists, so the chord property connects all neighboring pieces.
Hence $E$ is connected. The same property applied to $E/c$ gives (iii), which implies (ii), while
\[
 E/c\cap(2+E/c)\ne\varnothing
 \quad\Longleftrightarrow\quad 2c\in E-E=E(c,N).
\]
If (ii) fails, the sets $f_u(E)$, $u\in A_n^k$, partition $E$ into relatively clopen sets of diameter at most
$|c|^{-k}\operatorname{diam}E$, so every connected subset is a singleton.

The series representation and $A_N=2\Dn=-A_N$ give
\begin{equation}\label{eq:series-critical}
 c\in\Mn\quad\Longleftrightarrow\quad
 1+\sum_{j=1}^{\infty}d_jc^{-j}=0
 \quad\text{for some }d_j\in\Dn.
\end{equation}
Conjugation preserves the identity, and $c\mapsto-c$ is compensated by $d_j\mapsto(-1)^jd_j$.
\end{proof}

\subsection{Root closure and finite inverse words}
\begin{proposition}[Root-closure description]\label{prop:Mn-closure-Rn}\label{lem:Mn-closed-prelim}
The locus $\Mn$ is closed relative to $\C\setminus\Dbar$, and
\[
 \Mn=\overline{\Rn}\setminus\Dbar.
\]
\end{proposition}
\begin{proof}
Hausdorff continuity of contractive IFS attractors~\cite{Hutchinson1981FractalsMeasures} makes
$\Mn$ the zero set of the continuous function $c\mapsto\dist(2c,E(c,N))$ on $|c|>1$.
Every root in this domain gives a finite instance of \eqref{eq:series-critical}, so
$\Rn\setminus\Dbar\subset\Mn$; relative closedness then gives
$\overline{\Rn}\setminus\Dbar\subset\Mn$.

Conversely, for $c\in\Mn$ choose digits from \eqref{eq:series-critical} and set
\[
 F(z)=1+\sum_{j=1}^{\infty}d_jz^j,\qquad
 F_k(z)=1+\sum_{j=1}^{k}d_jz^j.
\]
Local uniform convergence in $\D$ and Hurwitz's theorem give zeros $\lambda_k\to1/c$ of $F_k$
for all large $k$, since $F(1/c)=0$ and $F(0)=1$. The reciprocal polynomial $z^kF_k(1/z)$ is monic
with coefficients in $\Dn$, so $1/\lambda_k\in\Rn$ and $1/\lambda_k\to c$.
\end{proof}

\begin{proof}[Proof of \Cref{prop:projective-root-closure}]
Put $h=n-1$. Since $F(X,Y)\mapsto F(Y,X)$ preserves height, $\iota(Z_h)=Z_h$. Let $F(c,1)=0$
with $|c|>1$, and write $F(z,1)=\sum_{j=0}^{k}a_jz^{k-j}$ with $a_0\ne0$. As $A_n-A_n=2\Dn$, there
are $t_j,s_j\in A_n$ with $t_j-s_j=2a_j$. For $w\in E(c,n)$, the points
$f_{t_0}\circ\cdots\circ f_{t_k}(w)$ and $f_{s_0}\circ\cdots\circ f_{s_k}(w)$ differ by
$2c^{-k}F(c,1)=0$, and they lie in the first-level pieces $f_{t_0}(E(c,n))$ and
$f_{s_0}(E(c,n))$, which are distinct because $t_0\ne s_0$. Hence $c\in\Mn$ by
\Cref{prop:Mn_characterization}. Since $\Rn\subset Z_h$, \Cref{prop:Mn-closure-Rn} now gives
$\overline{Z_h}\cap(\C\setminus\Dbar)=\Mn$, and $\iota(Z_h)=Z_h$ gives
$\overline{Z_h}\cap(\D\setminus\{0\})=\iota(\Mn)$. The forms $X^m-Y^m$, $X$ and $Y$ place
$\partial\D$, $0$ and $\infty$ in $\overline{Z_h}$. Finally, by Hurwitz's theorem, a zero
$\lambda\in\D$ of $f(z)=F(z,1)$ is a limit of zeros of the monic polynomials $z^m+f(z)$ as
$m\to\infty$; hence $\overline{Z_h}\cap\C=\overline{\Rn}$.
\end{proof}

For the maps $f_t(z)=t+z/c$, $g_t(z)=c(z-t)$ and a word $u=t_0\cdots t_{k-1}\in A_N^k$, set
\begin{equation}\label{def:gu}
 f_u=f_{t_0}\circ\cdots\circ f_{t_{k-1}},\qquad
 g_u=g_{t_{k-1}}\circ\cdots\circ g_{t_0}=f_u^{-1}.
\end{equation}
We write $|u|=k$ and take
$f_{\emptyword}=g_{\emptyword}=\mathrm{id}$. Concatenation satisfies $f_{uv}=f_u\circ f_v$ and
$g_{uv}=g_v\circ g_u$.

\begin{proposition}[Inverse words and polynomials]\label{prop:dyn-poly}
For $u=t_0\cdots t_{k-1}$, write $t_{j-1}=2d_j$ with $d_j\in\Dn$. Then
\[
 g_u(2c)=2cP_u(c),\qquad
 P_u(z)=z^k-d_1z^{k-1}-\cdots-d_k,
\]
with $P_{\emptyword}=1$. Consequently, $\Rn\setminus\Dbar$ consists exactly of the parameters at
which a finite inverse orbit of $2c$ reaches zero.
\end{proposition}
\begin{proof}
For prefixes $u_j=t_0\cdots t_{j-1}$, the quantities $R_j=g_{u_j}(2c)/(2c)$ satisfy
$R_0=1$ and $R_j=cR_{j-1}-d_j$. This is Horner's recursion for $P_u(c)$. Since $-\Dn=\Dn$, every polynomial defining $\Rn$ arises in this way.
\end{proof}
Equivalently, the finite expansion and its remainder are
\begin{equation}\label{eq:finite-expansion-remainder}
 1=\sum_{j=1}^k d_jc^{-j}+\frac{g_u(2c)}{2c^{k+1}}.
\end{equation}

\begin{proposition}[Connected annulus, \cite{EJSn24}]\label{prop:inner-annulus}
For every $n\ge2$, $\{c:1<|c|<\sqrt n\}\subset\operatorname{int}\Mn$ and
$\{c:1<|c|\le\sqrt n\}\subset\Mn$.
\end{proposition}
The open-annulus inclusion is \cite[Proposition~5(i)]{EJSn24}; the closed outer circle follows from
relative closedness. Appendix~\ref{supp:annulus-proof} gives an elementary area proof.

\section{The inner parallelogram}\label{sec:trap}
Entry into the canonical parallelogram certifies interior membership. We prove that the lens is the exact
self-covering domain for parallelograms with sides parallel to $1$ and $\overline c$.
Throughout this section, $c=x+iy$, $\rho=|c|>1$, $y\ne0$, and $N=2n-1$ with $n\ge2$.

\label{sec:canonical-coordinates}
Write $\Hutch_{c,m}(F)=\bigcup_{t\in A_m}(t+F/c)$ and $g_t(z)=c(z-t)$.
A nonempty set $U$ is \emph{self-covering} if
\begin{equation}\label{eq:self-covering-forward}
 U\subset\Hutch_{c,N}(U),
 \quad\text{equivalently}\quad
 (\forall z\in U)(\exists t\in A_N)\ g_t(z)\in U.
\end{equation}
If $U$ is bounded and open, then $U\subset\operatorname{int}E(c,N)$: successive returns give
\[
 z=\sum_{j=0}^{k-1}t_jc^{-j}+z_kc^{-k},\qquad z_k\in U.
\]
The remainder tends to zero, placing $z$ in $E(c,N)$; openness gives the interior inclusion.
This is the open-set counterpart of compact sub-self-similarity~\cite{Falconer1995SubSelfSimilar}.

To construct such a set, write $z=u+iv$ and use the \emph{canonical coordinates}
\[
 s=\ls(z;c)=\frac{\operatorname{Im}(cz)}{\rho}=\frac{yu+xv}{\rho},\qquad
 v=\lv(z)=\operatorname{Im}z.
\]
This real-linear isomorphism sends horizontal lines to $v=\mathrm{constant}$ and lines parallel
to $\overline c$ to $s=\mathrm{constant}$. The centered parallelogram
\[
 \mathcal P_c(\mathcal S,\mathcal V)
 =\{z:|\ls(z;c)|<\mathcal S,\ |\lv(z)|<\mathcal V\}
\]
becomes the open rectangle $|s|<\mathcal S$, $|v|<\mathcal V$ when both widths are positive,
and is empty otherwise. Since both functionals have Euclidean norm one, $\mathcal S$ and
$\mathcal V$ are the strip half-widths (\Cref{fig:canonical-parallelogram-to-rectangle}).

\begin{figure}[!htbp]
\centering
\begin{tikzpicture}[x=.88cm,y=.88cm,>=Stealth,font=\small,
  line cap=round,line join=round]
  \definecolor{teachblue}{HTML}{285C78}
  \path[use as bounding box] (-3.8,-2.25) rectangle (12.5,2.6);
  \node[anchor=west,font=\small\bfseries] at (-3.6,2.3) {(a) Dynamical coordinates};
  \node[anchor=west,font=\small\bfseries] at (6.6,2.3) {(b) Canonical coordinates};
  \filldraw[fill=teachblue!22,draw=teachblue,line width=.9pt]
    ({-sqrt(5)+.575},-1.15)--({sqrt(5)+.575},-1.15)--({sqrt(5)-.575},1.15)--({-sqrt(5)-.575},1.15)--cycle;
  \draw[->,black!70] (-3.5,0)--(3.55,0) node[below] {$u$};
  \draw[->,black!70] (0,-1.6)--(0,1.75) node[right] {$v$};
  \node[fill=teachblue!22,inner sep=1pt] at (-.9,-.72) {$\mathcal P_c(\mathcal S,\mathcal V)$};
  \node[anchor=north] at (0,-1.65) {$z=u+iv$};
  \draw[teachblue,line width=1.1pt] ({-sqrt(5)-.575},1.15)--({-sqrt(5)+.575},-1.15);
  \draw[->,teachblue,line width=.9pt] (0,0)--({2/sqrt(5)},{1/sqrt(5)});
  \node[text=teachblue,inner sep=1pt] at (1.35,.76) {$e_s(c)$};
  \draw[->,black!75,densely dashed,line width=.7pt] (0,0)--(.70,-1.40)
    node[right,inner sep=1.5pt] {$\overline c$};
  \draw[teachblue,line width=.55pt] (.20,.10)--(.30,-.10)--(.10,-.20);
  \node[anchor=east,align=right] at (-2.65,.1) {parallel to\\$\overline c$};
  \draw[->,black!75,line width=.85pt] (3.5,.35)--(6.35,.35);
  \node[align=center] at (4.95,1.05) {$s=\ell_s(z;c)$\\$v=\ell_v(z)$};
  \begin{scope}[xshift=9.2cm]
    \filldraw[fill=teachblue!22,draw=teachblue,line width=.9pt] (-2,-1.15) rectangle (2,1.15);
    \draw[->,black!70] (-2.4,0)--(2.6,0) node[below] {$s$};
    \draw[->,black!70] (0,-1.6)--(0,1.75) node[right] {$v$};
    \node[above] at (1,1.15) {$v=\mathcal V$};
    \node[anchor=west,fill=white,inner sep=1pt] at (2,.75) {$s=\mathcal S$};
    \node[anchor=north] at (0,-1.65) {$|s|<\mathcal S,\quad |v|<\mathcal V$};
  \end{scope}
\end{tikzpicture}
\caption{The same open set in two coordinate systems. Projection onto the unit normal $e_s(c)=i\overline c/|c|$ gives the slanted coordinate. Horizontal sides remain horizontal, while the sides parallel to $\overline c$ become vertical. The two strip inequalities therefore become the independent bounds of a rectangle. Boundaries are drawn to locate the open sets.}
\label{fig:canonical-parallelogram-to-rectangle}
\end{figure}

\begin{lemma}[Inverse dynamics in canonical coordinates]\label{lem:dynamics}
If $z$ has coordinates $(s,v)$, then $g_t(z)$ has coordinates
\[
 v'=\rho s-yt,\qquad s'=\frac{2x}{\rho}v'-\rho v.
\]
\end{lemma}
\begin{proof}
Linearity gives $\lv(cz)=\rho s$ and $\ls(cz;c)=2xs-\rho v$. Subtracting $ct$ gives
$v'=\rho s-yt$ and $s'=2xs-\rho v-2xyt/\rho$, as required.
\end{proof}

Choose $t$ to control $v'$, then use the second equation to bound $s'$. Since $A_N$ has spacing
$2$ and endpoints $\pm(N-1)$, every number in $[-N,N]$ is within distance $1$ of a digit.

\begin{proposition}[Nearest-digit covering criterion]\label{prop:C-NE}
If $\mathcal S,\mathcal V>0$ satisfy
\begin{equation}\label{eq:nearest-digit-criterion}
 |y|<\mathcal V,\qquad
 \rho\mathcal S\le N|y|,\qquad
 \rho\mathcal V+\frac{2|x||y|}{\rho}\le\mathcal S,
\end{equation}
then $\mathcal P_c(\mathcal S,\mathcal V)$ is self-covering. For each of its points, any digit
$t\in A_N$ nearest to $\alpha=\rho s/y$ gives
\begin{equation}\label{eq:nearest-digit-estimates}
 |v'|\le|y|<\mathcal V,\qquad
 |s'|\le\rho|v|+\frac{2|x||y|}{\rho}<\mathcal S.
\end{equation}
\end{proposition}
\begin{proof}
The second inequality in \eqref{eq:nearest-digit-criterion} gives $|\alpha|<N$. Hence
$|\alpha-t|\le1$ and $|v'|=|y|\,|\alpha-t|\le|y|$. The formula for $s'$ in
\Cref{lem:dynamics}, together with $|v|<\mathcal V$, proves the remaining estimate.
\end{proof}

Equality in the last two inequalities gives the canonical half-widths.

\begin{definition}[Canonical parallelogram]\label{def:functionals-canonical}\label{def:canonical-trap}
The \emph{canonical parallelogram} is
\[
 \PTrap(c,N)=\mathcal P_c\bigl(\mathcal S(c,N),\mathcal V(c,N)\bigr),\qquad
 \mathcal S(c,N)=\frac{N|y|}{\rho},\qquad
 \mathcal V(c,N)=\frac{(N-2|x|)|y|}{\rho^2}.
\]
When self-covering, we call it the \emph{canonical trap}.
\end{definition}

The remaining inequality $|y|<\mathcal V(c,N)$ is precisely the strict lens condition
$\rho^2+2|x|<N$. Within the lens, \eqref{eq:nearest-digit-criterion} implies
$\mathcal S\le\mathcal S(c,N)$ and $\mathcal V\le\mathcal V(c,N)$, so the canonical choice is
largest among parallelograms satisfying this criterion.

\label{sec:exact-lens}
For $n=2$, the substitution $\lambda=1/c$ turns the strict lens inequality into
$1+2|\operatorname{Re}\lambda|<3|\lambda|^2$, the inequality underlying the Solomyak--Xu
covering region~\cite[Section~3]{SolomyakXu2003}. The non-strict inequality for general $n$
and a covering rectangle appeared in~\cite{EJSn24}. The strict inequality is also necessary,
regardless of the parallelogram's center, widths or choice of inverse branches.

\begin{proof}[Proof of \Cref{thm:sharp-parallelogram-lens}]
For necessity, let $P$ be a nonempty bounded open self-covering parallelogram in the stated
directions. Set $q=|c|^2$, $p=2|x|$, and $h=N-1$. Rescale the coordinates by
\[
 \alpha(z)=\frac{\operatorname{Im}(cz)}{y}=\frac{\rho s}{y},\qquad
 \beta(z)=\frac{\operatorname{Im}z}{y}=\frac{v}{y}.
\]
In these coordinates, $P$ is a rectangle
\[
 |\alpha-A|<a,\qquad |\beta-B|<b,\qquad a,b>0,
\]
and the inverse branches satisfy
\begin{equation}\label{eq:rectangle-inverse}
 \beta(g_tz)=\alpha(z)-t,\qquad
 \alpha(g_tz)=2x\bigl(\alpha(z)-t\bigr)-q\beta(z).
\end{equation}
First, $b>1$. Indeed, the first-level images $t+P/c$ lie in the open strips
\[
 t+B-b<\alpha<t+B+b,\qquad t\in A_N.
\]
If $b\le1$, these strips are pairwise disjoint. Since $P$ is connected, its covering by the
first-level images would force it to lie in a single image. This is impossible: each image has
area $\operatorname{area}(P)/q<\operatorname{area}(P)$.

Next, each $z\in P$ has a digit $t$ with $g_t(z)\in P$. By \eqref{eq:rectangle-inverse},
$\alpha(z)=t+\beta(g_tz)$ then lies within $b$ of $B+t$, and
$2x\alpha(z)-q\beta(z)=\alpha(g_tz)+2xt$ lies within $a$ of $A+2xt$.
On $P$, the functions $\alpha$ and $2x\alpha-q\beta$ take all values in open intervals
of half-widths $a$ and $pa+qb$, while $|t|\le h$. Hence
\begin{equation}\label{eq:sharp-lens-necessary-corners}
 a\le h+b,\qquad qb+(p-1)a\le ph.
\end{equation}

Suppose first that $0\le p\le1$. Combining the two inequalities in
\eqref{eq:sharp-lens-necessary-corners} gives
\[
 (q+p-1)b
 \le qb+(p-1)a-(p-1)h
 \le h.
\]
Since $b>1$ and $q+p-1>0$, it follows that $q+p<h+1=N$.

Now suppose that $p>1$, and fix $\beta_0$ with $|\beta_0-B|<b$.
By \eqref{eq:rectangle-inverse}, a point of $P$ with $\beta=\beta_0$
returns under $g_t$ exactly when $\alpha-t$ lies in
\[
 J=\{r:|r-B|<b,\ |2xr-q\beta_0-A|<a\},
\]
an open interval of length at most $2a/p$. The $\alpha$-values of these points fill
$(A-a,A+a)$, so the translates $t+J$, $t\in A_N$, cover this interval.
If $|J|\le2$, the translates are pairwise disjoint, and the connected interval lies
in one of them; then $2a\le|J|\le2a/p$, which is impossible. Hence $|J|>2$.
Since the length of a nonempty intersection of two intervals is at most the sum
of their half-widths minus the distance between their centers,
\[
 2<b+\frac ap-\left|\frac{q\beta_0+A}{2x}-B\right|.
\]
As $\beta_0$ ranges over $(B-b,B+b)$, the affine expression inside the absolute value
ranges over an open interval of half-width $qb/p$. Its absolute value therefore has
supremum at least $qb/p$, giving
\[
 a\ge(q-p)b+2p.
\]
Substituting into the second inequality of \eqref{eq:sharp-lens-necessary-corners}
and dividing by $p$ yields
\[
 (q-p+1)b+2(p-1)\le h.
\]
Since $q-p+1=(|x|-1)^2+y^2>0$ and $b>1$, we again obtain $q+p-1<h$, as required.

Conversely, if $q+p<N$, the canonical half-widths satisfy
\eqref{eq:nearest-digit-criterion}. By \Cref{prop:C-NE}, the canonical parallelogram is
self-covering.
\end{proof}

\begin{corollary}[Canonical trap and interior inclusion]\label{cor:canonical-validity}\label{cor:self-covering}
The canonical parallelogram is nonempty exactly when $2|x|<N$. It is a nonempty open
self-covering set exactly when $c\in\Xn$; in that case,
\[
 |y|<\mathcal V(c,N),\qquad
 \PTrap(c,N)\subset\operatorname{int}E(c,N).
\]
On the boundary $\rho^2+2|x|=N$, it is not self-covering.
\end{corollary}
\begin{proof}
Nonemptiness is equivalent to $\mathcal V(c,N)>0$. The other assertions follow from
\Cref{thm:sharp-parallelogram-lens} and the bounded open self-covering argument above.
Directly, whenever $0<\mathcal V(c,N)\le|y|$, the point $1$ lies in the parallelogram,
but no $g_t(1)$ returns: $1-t$ is odd, so $|\lv(g_t(1))|\ge|y|\ge\mathcal V(c,N)$.
\end{proof}

Solomyak and Xu~\cite[Section~3, final remarks]{SolomyakXu2003}
also suggest other shapes and coverings by longer words.

The following inclusion shows that replacing the covering rectangle of~\cite{EJSn24} by the
canonical trap can only improve the first entry depth along the same inverse search.

\begin{proposition}[Comparison with the covering rectangle]\label{prop:rectangle-in-trap}
Let $c=x+iy\in\Xn\setminus\R$, set $q=|c|^2$ and $L=q+2|x|+1$, and define
\[
\begin{gathered}
 h_R=\frac{(N+1)(q+|x|)}{L},\qquad
 v_R=\frac{(N+1)|y|}{L},\\
 R(c,N)=\{z:|\Real z|\le h_R,\ |\Imag z|\le v_R\}.
\end{gathered}
\]
Then the closed rectangle $R(c,N)$ is contained in $\PTrap(c,N)$. In particular, if an inverse
word of length $k$ takes $2c$ into $R(c,N)$, then $c\in\Theta_k(n)$.
\end{proposition}
\begin{proof}
The coordinate bounds on the rectangle are
\[
 \sup_{z\in R(c,N)}|\lv(z)|=v_R,\qquad
 \sup_{z\in R(c,N)}|\ls(z;c)|=\frac{|y|h_R+|x|v_R}{\sqrt q}.
\]
Subtracting these from the canonical half-widths gives
\[
 \begin{aligned}
  \mathcal V(c,N)-v_R
  &=\frac{|y|(2|x|+1)(N-q-2|x|)}{qL}>0,\\
  \mathcal S(c,N)-\frac{|y|h_R+|x|v_R}{\sqrt q}
  &=\frac{|y|(N-q-2|x|)}{\sqrt q\,L}>0.
 \end{aligned}
\]
Both gaps are positive in the strict lens. The capture assertion follows from the definition
of $\Theta_k(n)$.
\end{proof}
\Cref{fig:parallelogram-iterates} illustrates the self-covering property and this inclusion
for $n=3$.

\begin{figure}[!htbp]
\centering
\begin{tikzpicture}[x=.46cm,y=.46cm,>=Stealth,font=\small,
  line cap=round,line join=round]
  \definecolor{teachblue}{HTML}{285C78}
  \path[use as bounding box] (-7.6,-5.4) rectangle (24,6.0);
  \node[anchor=west,font=\small\bfseries] at (-7.3,5.5)
    {(a) The rectangle lies in the trap};
  \node[anchor=west,font=\small\bfseries] at (8.9,5.5)
    {(b) Five copies cover the trap};
  \DrawTrapP[fill=teachblue!18,draw=teachblue,line width=.85pt]{0.7}{1.4}{3}
  \pgfmathsetmacro{\hR}{6*(2.45+.7)/4.85}
  \pgfmathsetmacro{\vR}{6*1.4/4.85}
  \filldraw[fill=teachblue!35,draw=black!75,line width=.75pt]
    (-\hR,-\vR) rectangle (\hR,\vR);
  \draw[->,black!55,line width=.45pt] (-7,0)--(7,0) node[right] {$\Real z$};
  \draw[->,black!55,line width=.45pt] (0,-4.0)--(0,4.05) node[above] {$\Imag z$};
  \node[fill=teachblue!35,inner sep=1pt] at (0,-.78) {$R(c,N)$};
  \node[anchor=west] (traplabel) at (2.5,3.35) {$\mathcal C(c,N)$};
  \draw[->,teachblue,line width=.65pt] (traplabel.south west)--(2.6,2.05);
  \node at (0,-4.8) {$R(c,N)\subset\mathcal C(c,N)$};
  \begin{scope}[xshift=7.452cm]
    \DrawPone[fill=teachblue!43,draw=none]{0.7}{1.4}{3}
    \DrawTrapP[fill=teachblue!18,draw=none]{0.7}{1.4}{3}
    \DrawPone[fill=none,draw=teachblue!55,line width=.35pt]{0.7}{1.4}{3}
    \DrawTrapP[fill=none,draw=teachblue,line width=.95pt]{0.7}{1.4}{3}
    \draw[->,black!55,line width=.45pt] (-7,0)--(7,0) node[right] {$\Real z$};
    \draw[->,black!55,line width=.45pt] (0,-4.0)--(0,4.05) node[above] {$\Imag z$};
    \foreach \t in {-4,-2,0,2,4}{
      \fill[black!80] (\t,0) circle (1.3pt);
      \node[below,fill=teachblue!18,inner sep=.4pt,font=\footnotesize] at (\t,-.14) {$\t$};
    }
    \filldraw[fill=black,draw=white,line width=.6pt] (1.4,2.8) circle (1.9pt);
    \node[above right,fill=white,inner sep=1.1pt] at (1.4,2.8) {$2c$};
    \node at (0,-4.8) {$\mathcal C(c,N)\subset\displaystyle\bigcup_{t\in A_5}
      \bigl(t+\mathcal C(c,N)/c\bigr)$};
  \end{scope}
\end{tikzpicture}
\caption{Two containments for $n=3$, $N=5$, and $c=0.7+1.4i$. (a) The closed covering rectangle lies strictly inside the open canonical trap. (b) The five copies $t+\mathcal C(c,N)/c$ cover the trap; thin outlines distinguish the copies and real-axis dots mark their centers $t\in A_5$. The upper dot marks $2c$. Outlines locate the boundaries of the open sets.}
\label{fig:parallelogram-iterates}
\end{figure}

\FloatBarrier

\section{Canonical capture levels and their boundaries}\label{sec:stratification}
We first describe capture by polynomial inequalities, then use the geometry of inverse returns
to show that every boundary point of a level within the non-real lens is captured within two
further steps.

\begin{figure}[!t]
\centering
\includegraphics[width=.8\linewidth]{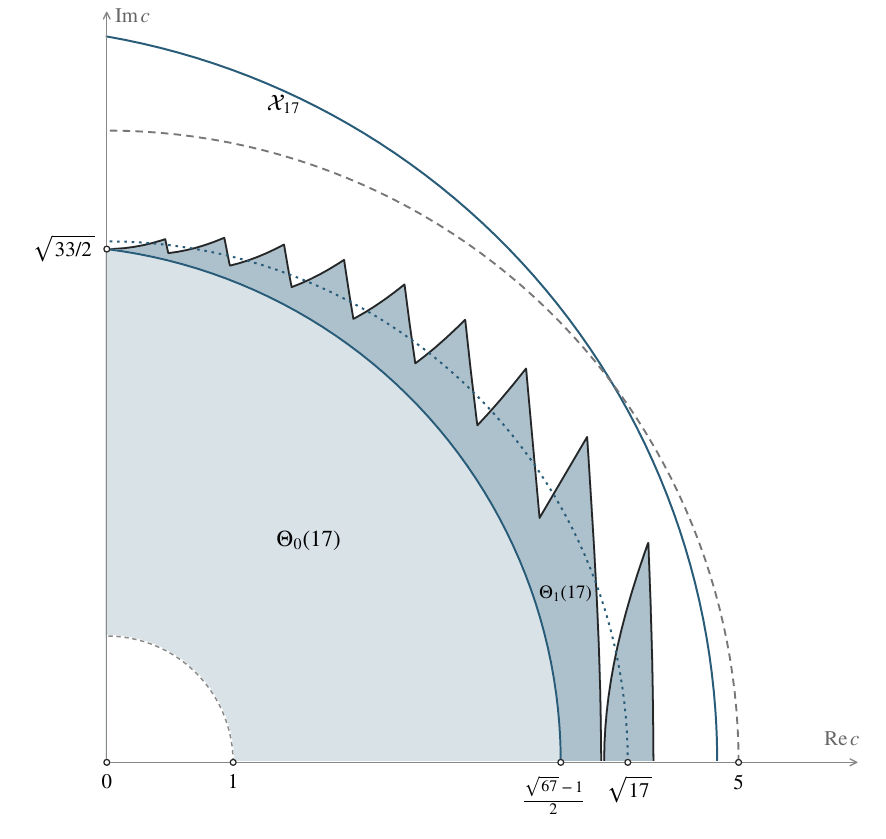}
\caption{The first two capture levels for $n=17$ in the first quadrant. The light region is the exact base level $\Theta_0(17)$ (\Cref{prop:Theta0}); the dark band is $\Theta_1(17)\setminus\Theta_0(17)$, computed exactly at sampled values of $\operatorname{Re}c$. The outer solid arc bounds the lens $\mathcal X_{17}$; dotted and dashed arcs mark $|c|=\sqrt{17}$ and the radial bound $|c|=5$ of \eqref{eq:known-radial-bound}.}
\label{fig:capture-envelope-n17}
\end{figure}

\subsection{Open levels and polynomial word regions}
Self-covering extends a successful word by one digit, so
$\Theta_k(n)\subset\Theta_{k+1}(n)$. Strict entry also persists under small parameter changes.

\begin{proposition}[Finite entry gives an interior parameter]\label{prop:finite-capture-interior}
For every $n\ge2$ and $k\ge0$, the set $\Theta_k(n)$ is open and contained in $\operatorname{int}\Mn$.
\end{proposition}
\begin{proof}
Let $c_0\in\mathcal U_u$ with $|u|\le k$. The strict entry inequalities persist on a neighborhood of $c_0$ in $\Xn\setminus\R$.
At each of its points $c$, \Cref{cor:self-covering} places $g_u(2c)$ in $E(c,N)$,
and applying $f_u$ places $2c$ there. The marked-point criterion gives connectedness
throughout this neighborhood.
\end{proof}

\Needspace{9\baselineskip}
Increasing the height bound also enlarges each capture level:
\begin{equation}\label{eq:arity-monotonicity}
 \overline{\Theta_k(n)}\cap\bigl(\C\setminus(\Dbar\cup\R)\bigr)
 \subset\Theta_k(n+1),\qquad n\ge2,\quad k\ge0.
\end{equation}
In particular, $\kappa_{n+1}(c)\le\kappa_n(c)$ at every fixed non-real parameter with $|c|>1$.
Indeed, let $c$ be a non-real limit of parameters $c_j\in\Theta_k(n)$ with $|c|>1$, and pass to a
subsequence with a common successful word. At $c$, the lens and entry inequalities hold
non-strictly. Replacing $N$ by $N+2$ retains the word's digits and makes all three inequalities
strict: the right-hand side of the lens inequality increases by $2$, and each right-hand side in
\eqref{eq:intro-trap}, $N|y|$ and $(N-2|x|)|y|$, by $2|y|>0$. The same word gives capture with $n+1$ maps.

\begin{proposition}[Base capture level]\label{prop:Theta0}
For $n\ge2$, writing $N=2n-1$ and $c=x+iy$, one has
\[
 \Theta_0(n)=\{c\in\Xn\setminus\R:|c|^2+|x|<N/2\}.
\]
Equivalently, it is the intersection of the disks $|c\pm\tfrac12|<\sqrt{n-\tfrac14}$, with the
closed unit disk and real axis removed.
\end{proposition}
\begin{proof}
At $z=2c$, the two trap inequalities are $|c|^2+|x|<N/2$ and $4|x|<N$. The first implies the second:
for $N=3$, use $|c|^2>1$; for $N\ge5$, the contrary inequality $|x|\ge N/4$ would give
$|c|^2+|x|\ge N^2/16+N/4\ge N/2$.
Conversely, the first inequality and $4|x|<N$ imply $|c|^2+2|x|<N$, hence lens membership.
Completing the square gives the two disks.
\end{proof}

Since $|c|^2<N/2<n$ on $\Theta_0(n)$, the base level lies in the connected annulus of
\Cref{prop:inner-annulus}. One further inverse digit already captures parameters beyond it, as
\Cref{fig:capture-envelope-n17} shows for $n=17$.

\subsection{Word regions and forward images}\label{sec:finite-level-geometry}
A word region has two equivalent interpretations:
\begin{equation}\label{eq:word-duality}
 c\in\mathcal U_u
 \quad\Longleftrightarrow\quad 2c\in f_u(\PTrap(c,N))
 \quad\Longleftrightarrow\quad g_u(2c)\in\PTrap(c,N),
 \qquad c\in\Xn\setminus\R.
\end{equation}
Thus a forward copy of the trap contains the marked point precisely when the corresponding
inverse orbit ends in the trap. If $z_j=g_{t_0\cdots t_{j-1}}(2c)$, the identity
\[
 2c=p_j+c^{-j}z_j,\qquad p_j=\sum_{r<j}t_rc^{-r},
\]
relates the two pictures at each step. Figure~\ref{fig:duality} illustrates a length-four word;
Figure~\ref{fig:n21-word-charts} later shows how successful words define regions in parameter space.

\begin{figure}[!htbp]
\centering
\begingroup
\definecolor{dualityblue}{HTML}{285C78}
\newcommand{\DualityTrap}[1][]{%
  \path[#1] (-4.28822847340248,-2.03797830265102)
   -- (5.71177152659752,-2.03797830265102)
   -- (4.28822847340248,2.03797830265102)
   -- (-5.71177152659752,2.03797830265102) -- cycle;}
\begin{minipage}[t]{.48\linewidth}
\centering
\textbf{(a) Forward image}\par\smallskip
\begin{tikzpicture}[x=1cm,y=1cm,>=Stealth,font=\footnotesize,
 line cap=round,line join=round]
 \path[use as bounding box] (-.6,-.8) rectangle (6.05,4.3);
 \draw[->,black!55] (-.35,0)--(2.65,0) node[right,inner sep=2pt] {$\Real z$};
 \draw[->,black!55] (0,-.3)--(0,3.8) node[above] {$\Imag z$};
 \coordinate (p0) at (0,0);
 \coordinate (p1) at (2,0);
 \coordinate (p2) at (1.2566354813603,2.1284368696094);
 \coordinate (p3) at (1.7538420689028,2.5239879818675);
 \coordinate (p4) at (1.1480864155238,2.9061050671560);
 \coordinate (marked) at (1.17,3.35);
 \begin{scope}[shift={(p4)},
  cm={.02268842707178,.09833530936224,-.09833530936224,.02268842707178,(0,0)}]
  \DualityTrap[fill=dualityblue!25,draw=dualityblue,line width=.7pt]
 \end{scope}
 \draw[dualityblue,line width=.9pt] (p0)--(p1)--(p2)--(p3)--(p4);
 \foreach \j in {0,1,2,3,4}{
  \filldraw[fill=white,draw=dualityblue,line width=.6pt] (p\j) circle (1.5pt);}
 \node[below left,inner sep=2pt] at (p0) {$p_0$};
 \node[below left,inner sep=2pt] at (p1) {$p_1$};
 \node[left,inner sep=3pt] at (p2) {$p_2$};
 \node[right,inner sep=3pt] at (p3) {$p_3$};
 \node[left,inner sep=3pt] at (p4) {$p_4$};
 \filldraw[fill=black,draw=white,line width=.5pt] (marked) circle (1.9pt);
 \node[above right,inner sep=2pt] at (marked) {$2c$};
 \draw[fill=white,draw=black!45,line width=.6pt] (3.60,-.2) rectangle (5.9,3.45);
 \begin{scope}[shift={(4.75,1.5)},
  cm={.06806528121534,.29500592808672,-.29500592808672,.06806528121534,(0,0)}]
  \DualityTrap[fill=dualityblue!25,draw=dualityblue,line width=.7pt]
 \end{scope}
 \filldraw[fill=white,draw=dualityblue,line width=.6pt] (4.75,1.5) circle (1.5pt);
 \node[below left,inner sep=2pt] at (4.75,1.5) {$p_4$};
 \filldraw[fill=black,draw=white,line width=.5pt]
  (4.8157407534287,2.8316847985319) circle (1.9pt);
 \node[above right,inner sep=2pt] at (4.8157407534287,2.6) {$2c$};
 \node[above,align=center] at (4.75,3.48) {$f_u(\PTrap(c,5))$\\magnified $3\times$};
 \draw[->,black!55,line width=.6pt] (1.47,3.18)
  .. controls (2.55,3.85) and (3.2,3.85) .. (3.6,3.05);
 \node at (2.65,-.58) {$p_j=\textstyle\sum_{r<j}t_rc^{-r}$};
\end{tikzpicture}
\end{minipage}\hfill
\begin{minipage}[t]{.50\linewidth}
\centering
\textbf{(b) Inverse entry}\par\smallskip
\begin{tikzpicture}[x=.485cm,y=.485cm,>=Stealth,font=\footnotesize,
 line cap=round,line join=round]
 \path[use as bounding box] (-7.5,-4.86) rectangle (7.7,5.66);
 \DualityTrap[fill=dualityblue!14,draw=dualityblue,line width=.75pt]
 \draw[->,black!55] (-7,0)--(7,0) node[below] {$\Real z$};
 \draw[->,black!55] (0,-3.6)--(0,4.5) node[above] {$\Imag z$};
 \node[fill=dualityblue!14,inner sep=1pt] at (-1,.85) {$\PTrap(c,5)$};
 \coordinate (z0) at (1.17,3.35);
 \coordinate (z1) at (-6.0968,.5695);
 \coordinate (z2) at (-2.1805405,-3.1789825);
 \coordinate (z3) at (5.219179495,-2.1621101);
 \coordinate (z4) at (4.334754422075,.777291245625);
 \draw[->,dualityblue,line width=.85pt] (z0)--(z1)
  node[pos=.48,above,fill=white,inner sep=1pt] {$2$};
 \draw[->,dualityblue,line width=.85pt] (z1)--(z2)
  node[pos=.55,below left,fill=white,inner sep=1pt] {$-4$};
 \draw[->,dualityblue,line width=.85pt] (z2)--(z3)
  node[pos=.48,below,fill=white,inner sep=1pt] {$-2$};
 \draw[->,dualityblue,line width=.85pt] (z3)--(z4)
  node[pos=.5,right,fill=white,inner sep=1pt] {$4$};
 \foreach \j in {0,1,2,3}{\fill[dualityblue] (z\j) circle (1.5pt);}
 \fill[black] (z4) circle (1.9pt);
 \node[above right,inner sep=2pt] at (z0) {$z_0=2c$};
 \node[above left,inner sep=2pt] at (z1) {$z_1$};
 \node[below left,inner sep=2pt] at (z2) {$z_2$};
 \node[below right,inner sep=2pt] at (z3) {$z_3$};
 \node[above right,inner sep=2pt] at (z4) {$z_4$};
 \node at (.1,-4.45) {$z_{j+1}=c(z_j-t_j)$};
\end{tikzpicture}
\end{minipage}
\endgroup
\caption{The two descriptions of finite capture, for $n=3$,
 $c=117/200+(67/40)i$ and $u=(2,-4,-2,4)$.
 (a) The prefix sums locate the copy
 $f_u(\PTrap(c,5))=p_4+c^{-4}\PTrap(c,5)$ containing $2c$.
 (b) The same digits, read along the arrows, send $2c$ into the canonical trap.
 The identity $2c=p_j+c^{-j}z_j$ gives
 $c\in\mathcal U_u\iff2c\in f_u(\PTrap(c,5))\iff g_u(2c)\in\PTrap(c,5)$.
 No word of length at most three captures $c$, so $\kappa_3(c)=4$.}
\label{fig:duality}
\end{figure}

We now prove \Cref{thm:main-finite-level-geometry}: word regions and capture levels are open and
semialgebraic, capture layers are semialgebraic, and their relative boundaries consist of
finitely many real-analytic semialgebraic arcs and points.

\begin{proof}[Proof of \Cref{thm:main-finite-level-geometry}]
Write $c=x+iy$. By \Cref{prop:dyn-poly}, $z_u(c)=g_u(2c)$ is polynomial in $c$, including
$z_{\emptyword}(c)=2c$. Put $q=x^2+y^2$. On $\Xn\setminus\R$, the trap half-widths are positive, so
its two strict inequalities are equivalent to
\begin{equation}\label{eq:word-chart-polynomials}
 q^2\bigl(\operatorname{Im}z_u(c)\bigr)^2<(N-2|x|)^2y^2,
 \qquad
 \bigl(\operatorname{Im}(cz_u(c))\bigr)^2<N^2y^2.
\end{equation}
Splitting into $x\ge0$ and $x\le0$ gives polynomial inequalities, hence semialgebraicity of
$\mathcal U_u$; the original strict inequalities give openness. Finite unions and differences
then give semialgebraicity of $\Theta_k(n)=\bigcup_{|u|\le k}\mathcal U_u$ and
$\Omega_k(n)=\Theta_k(n)\setminus\Theta_{k-1}(n)$ for $k\ge1$, with $\Omega_0(n)=\Theta_0(n)$.

The relative boundary of a semialgebraic subset of the ambient parameter plane is semialgebraic of
dimension at most one. Real-analytic semialgebraic stratification gives a finite partition into
one-dimensional manifolds and points~\cite[Sections~2.8 and~9.1]{BochnakCosteRoy1998}. Cutting the
one-dimensional strata at finitely many points gives the asserted arcs.
\end{proof}

For $|u|=k$, the polynomial $z_u$ has degree $k+1$, so the polynomial differences in
\eqref{eq:word-chart-polynomials} have degrees $2k+6$ and $2k+4$. Dividing their common
factor $y^2$ and factoring gives boundary candidates
$qA_u=\pm(N-2|x|)$ and $B_u=\pm N$, where
$A_u=\operatorname{Im}z_u(c)/y$ and $B_u=\operatorname{Im}(cz_u(c))/y$
are real polynomials of degrees $k$ and $k+1$. Within each quadrant these curves have
degrees $k+2$ and $k+1$. Self-covering extends shorter successful words, so
$\Theta_k(n)=\bigcup_{|u|=k}\mathcal U_u$, a union of at most $N^k$ word regions.

For a one-letter word $u=t$, we have $z_u(c)=2c^2-tc$.
Writing $c=x+iy$ and $q=x^2+y^2$, \eqref{eq:word-chart-polynomials} becomes
$q|4x-t|<N-2|x|$ and $|3x^2-y^2-tx|<N/2$.
Thus, relative to the non-real lens and within each quadrant, the boundary of each
one-letter word region lies on the cubics $q(4x-t)=\pm(N-2|x|)$ and the hyperbolas
$y^2=3x^2-tx\pm N/2$. These equations give explicit candidate curves for the depth-one
capture boundary illustrated in \Cref{fig:capture-envelope-n17}.

For a two-letter word $u=st$, $z_u(c)=2c^3-sc^2-tc$. In the coordinates $(x,q)$ its region is
specified exactly by
\[
 q|8x^2-2q-2sx-t|<N-2|x|,\qquad
 |(s-8x)q+16x^3-4sx^2-2tx|<N,
\]
together with $q>1$, $q>x^2$ and $q+2|x|<N$. Thus its boundary candidates in $(x,y)$ are quartics
and cubics. At fixed $x$, the quadratic and linear inequalities in $q$ give at most two
intervals per word; taking their finite union describes $\Theta_2(n)$ explicitly.

For $n=13$, the component of $\C\setminus\Theta_k(13)$ containing
$p=2.06+3.05i$ is unbounded at $k=2$ and bounded at $k=3$;
$\partial\Theta_3(13)$ is disconnected.
A closed square about $p$ lies in $\Omega_4(13)$.
Appendix~\ref{supp:low-depth-topology} proves these statements by rational inequalities;
Figure~\ref{fig:M13context}(e) shows the polygon $\partial Q$ of that appendix.

\subsection{Forward images and the two-step inclusion}
We first work in the dynamical plane at a fixed parameter. The sets below
consist of starting points captured within $k$ steps; testing whether
they contain $2c$ connects them to the parameter levels $\Theta_k(n)$.
For a fixed $c\in\Xn\setminus\R$, put
\begin{equation}\label{eq:forward-word-union}
 \PTrap_0(c)=\PTrap(c,N),\qquad
 \PTrap_k(c)=\bigcup_{u\in A_N^k}f_u(\PTrap_0(c)).
\end{equation}
Self-covering makes these open sets nested. Each $f_t$ maps $\operatorname{int}E(c,N)$ into itself,
so $\PTrap_k(c)\subset\operatorname{int}E(c,N)$ and
\[
 c\in\Theta_k(n)\quad\Longleftrightarrow\quad 2c\in\PTrap_k(c).
\]

A nearest-digit inverse step makes the vertical inequality (on $\operatorname{Im}z$) strict;
a second guarantees strictness of the slanted inequality (on $\operatorname{Im}(cz)$) as well.
The following proposition identifies exactly which points require
the second step.

\begin{proposition}[Two-step return and the one-step obstruction]
\label{prop:two-step-buffer}
Fix $c=x+iy\in\Xn\setminus\R$, put $q=|c|^2$ and
$b=(N-2|x|)/q>1$, and use the coordinates
\[
 \Phi_c(z)=(\alpha(z),\beta(z))
 =\left(\frac{\Imag(cz)}y,\frac{\Imag z}y\right).
\]
These are rescaled canonical coordinates:
$\alpha=\rho s/y$ and $\beta=v/y$.
Let $D_c=\overline{\PTrap_0(c)}\setminus\PTrap_1(c)$. Then
\begin{equation}\label{eq:exact-one-step-obstruction}
 \Phi_c(D_c)=
 \begin{cases}
 \{(N,-\operatorname{sgn}(x)b),
   (-N,\operatorname{sgn}(x)b)\},&x\ne0,\\[1mm]
 [-N,N]\times\{-b,b\},&x=0.
 \end{cases}
\end{equation}
Thus $D_c$ consists of the two acute vertices when $x\ne0$, and of
the two closed horizontal sides when $x=0$. Moreover, for every $k\ge0$,
\begin{equation}\label{eq:forward-one-step-exceptions}
 \overline{\PTrap_k(c)}\setminus\PTrap_{k+1}(c)
 \subset\bigcup_{u\in A_N^k}f_u(D_c),
 \qquad
 \overline{\PTrap_k(c)}\subset\PTrap_{k+2}(c).
\end{equation}
In particular, if $x\ne0$, the first exceptional set in
\eqref{eq:forward-one-step-exceptions} has at most $2N^k$ points.
\end{proposition}

\begin{proof}
In these coordinates
$\PTrap_0(c)=\{|\alpha|<N,\ |\beta|<b\}$, and an inverse branch
with $r=\alpha-t$ satisfies
\[
 \beta(g_tz)=r,\qquad \alpha(g_tz)=2xr-q\beta.
\]
If $|\beta|<b$, a nearest digit has $|r|\le1<b$ and
\[
 |2xr-q\beta|\le2|x|+q|\beta|<2|x|+qb=N,
\]
so every such point of the closed base region enters in one step.

Suppose $x>0$ and $\beta=b$. For $\alpha>-N$, a nearest digit can be
chosen with $-1<r\le1$: at a tie choose $r=1$. Then
\[
 -N<2xr-qb=-N+2x(r+1)\le-N+4x<N,
\]
since $2x<N$. At $\alpha=-N$, every available digit instead has
$r\le-1$, so $2xr-qb\le-N$ and no branch enters.
Central symmetry treats $\beta=-b$. For $x<0$, replacing
$(x,\beta)$ by $(-x,-\beta)$ in these inequalities gives the first
case of \eqref{eq:exact-one-step-obstruction}. These are the acute
vertices; in complex coordinates they are
\[
 \pm\bigl(N+|x|b-i\operatorname{sgn}(x)yb\bigr).
\]
If $x=0$ and $|\beta|=b$, every branch has
$|\alpha(g_tz)|=qb=N$, proving the second case.

For any point of the closed base region, a nearest-digit step gives
$|\beta'|\le1<b$ and $|\alpha'|\le N$. The first part of the proof
therefore gives strict entry after at most one further step. Hence
$\overline{\PTrap_0(c)}\subset\PTrap_2(c)$.

Finally, closure commutes with the finite union of the affine images
$f_u(\PTrap_0(c))$. If a point in
$f_u(\overline{\PTrap_0(c)})$ is outside $\PTrap_{k+1}(c)$, its
preimage is outside $\PTrap_1(c)$ and hence lies in $D_c$.
Both conclusions in \eqref{eq:forward-one-step-exceptions} follow.
\end{proof}

Thus $\PTrap_2(c)$ contains a neighborhood of $\overline{\PTrap_0(c)}$ (\Cref{fig:nested}).
\begin{figure}[!htbp]
\centering
\begin{tikzpicture}[x=.70cm,y=.70cm,>=Stealth,font=\small,
  line cap=round,line join=round]
  \definecolor{teachblue}{HTML}{285C78}
  \path (-7.7,-4.8) rectangle (7.7,5.2);

  \DrawPlevel[fill=teachblue!75,draw=none]{2}{0.7}{1.4}{3}
  \DrawPlevel[fill=teachblue!38,draw=none]{1}{0.7}{1.4}{3}
  \DrawPlevel[fill=teachblue!18,draw=teachblue,line width=1.1pt]
    {0}{0.7}{1.4}{3}

  \draw[->,black!70] (-7.2,0)--(7.3,0) node[right] {$\Real z$};
  \draw[->,black!70] (0,-4.25)--(0,4.55) node[above] {$\Imag z$};
  \node[fill=teachblue!18,inner sep=2pt] at (0,-.65)
    {$\mathcal C_0(c)$};

  \coordinate (acute) at ({211/35},{-72/35});
  \fill[teachblue] (acute) circle[radius=1.6pt];
  \fill[teachblue] ({-211/35},{72/35}) circle[radius=1.6pt];
  \node[anchor=west,align=left,inner sep=2pt] (buffer)
    at (7.7,-3.1)
    {Acute vertex $p$\\
     $p\notin\mathcal C_1(c)$\\
     $p\in\mathcal C_2(c)$};
  \draw[->,teachblue,line width=.7pt,shorten >=2.5pt]
    (buffer.west)--(acute);

  \fill[teachblue!18] (-5.4,-4.65) rectangle (-4.9,-4.35);
  \node[anchor=west] at (-4.8,-4.5) {$\mathcal C_0(c)$};
  \fill[teachblue!38] (-.8,-4.65) rectangle (-.3,-4.35);
  \node[anchor=west] at (-.2,-4.5) {$\mathcal C_1(c)$};
  \fill[teachblue!75] (3.8,-4.65) rectangle (4.3,-4.35);
  \node[anchor=west] at (4.4,-4.5) {$\mathcal C_2(c)$};
\end{tikzpicture}
\caption{The nested open trap levels for $N=5$ and $c=0.7+1.4i$;
internal seams between individual copies are omitted. The second level
contains a neighborhood of the closed base trap:
$\overline{\PTrap_0(c)}\subset\PTrap_2(c)$.
The two marked acute vertices $\pm p$, where
$p=211/35-(72/35)i$, are exactly the points of
$\overline{\PTrap_0(c)}$ outside $\PTrap_1(c)$
(\Cref{prop:two-step-buffer}). Both belong to $\PTrap_2(c)$.}
\label{fig:nested}
\end{figure}

In the parameter plane, this gives \Cref{thm:bounded-delay-closure}: every limit in $\Xn\setminus\R$ of
parameters captured by depth $k$ is captured by depth $k+2$.
\begin{proof}[Proof of \Cref{thm:bounded-delay-closure}]
Let $c_j\in\Theta_k(n)$ converge to $c\in\Xn\setminus\R$. Pass to a subsequence with a common
entry word $u$, of length at most $k$. Continuity of $c\mapsto g_u(2c)$, together with the positive widths at $c$, gives
\[
 g_u(2c)\in\overline{\PTrap_0(c)}\subset\PTrap_2(c),
\]
by \Cref{prop:two-step-buffer}. Appending the resulting two-digit return gives
$c\in\Theta_{k+2}(n)$.
\end{proof}

\begin{example}[The depth-zero exceptions]\label{ex:sharp-two-step}
For every $n\ge2$, with $N=2n-1$,
\begin{equation}\label{eq:depth-zero-exceptions}
 \left(\overline{\Theta_0(n)}\cap(\Xn\setminus\R)\right)
 \setminus\Theta_1(n)
 =\left\{\pm i\sqrt{N/2}\right\}.
\end{equation}
Both exceptional parameters have capture depth two.

Indeed, if $c=x+iy$ belongs to the closure on the left, continuity
gives $2c\in\overline{\PTrap_0(c)}$, with
$\alpha(2c)=4x$ and $\beta(2c)=2$.
When $x\ne0$, this cannot be an acute vertex: the product
$x\alpha\beta$ is negative at either acute vertex, but equals $8x^2$
at $2c$. By \Cref{prop:two-step-buffer}, $c\in\Theta_1(n)$.
When $x=0$, failure of one-step entry requires $b=2$, hence
$|c|^2=N/2$.

Conversely, put $c_\sharp=ir$, where $r=\sqrt{N/2}>1$.
By \Cref{prop:Theta0}, every $ir'$ with $1<r'<r$ lies in
$\Theta_0(n)$, so $c_\sharp\in\overline{\Theta_0(n)}$.
The trap is $\{|\Real z|<N,\ |\Imag z|<2r\}$, and
\[
 g_t(2c_\sharp)=-N-irt\qquad(t\in A_N),
\]
so no first inverse image enters. The word $(0,-(N-1))$ gives
\[
 2c_\sharp\longmapsto-N\longmapsto-c_\sharp
 \in\PTrap(c_\sharp,N).
\]
Conjugation gives the negative-imaginary exception. Thus the uniform
parameter-space delay is sharp already at depth zero; at this depth,
these two parameters are the only exceptions to a one-step delay.
\end{example}

\begin{corollary}[Boundary layers and capture depth]\label{cor:boundary-depth}
For every $n\ge2$ and $k\ge0$,
\[
 \partial\Theta_k(n)\cap(\Xn\setminus\R)
 \subset\Omega_{k+1}(n)\cup\Omega_{k+2}(n).
\]
If $c_j\in\Theta_{\mathrm{fin}}(n)$ and $c_j\to
c_\infty\in(\Xn\setminus\R)\setminus\Theta_{\mathrm{fin}}(n)$, then $\kappa_n(c_j)\to\infty$.
\end{corollary}

\begin{proof}
Openness and \Cref{thm:bounded-delay-closure} place $\partial\Theta_k(n)\cap(\Xn\setminus\R)$ in
$\Theta_{k+2}(n)\setminus\Theta_k(n)$. If a subsequence had depths bounded by $K$, the same
theorem would put $c_\infty$ in $\Theta_{K+2}(n)$, a contradiction.
\end{proof}

\subsection{Word length and capture depth}
A polynomial root can be captured before the end of its defining word, as the overlapping word regions in
\Cref{fig:n21-word-charts} illustrate.

\begin{example}[A word region meeting two capture layers]
\label{ex:n21-word-chart-strata}
Let $n=21$. For the word $u=(10,-40,-40)$ the polynomial of \Cref{prop:dyn-poly} is
$P_u(z)=z^3-5z^2+20z+20$. Its first-quadrant root $c_*$ belongs to
$\mathcal U_u\cap\Omega_2(21)$, whereas $c_\dagger=2.85+4.057i$ belongs to
$\mathcal U_u\cap\Omega_3(21)$. Appendix~\ref{supp:chart-verification} verifies both depths.
\end{example}

\begin{figure}[p]
 \centering
 \includegraphics[width=\linewidth]{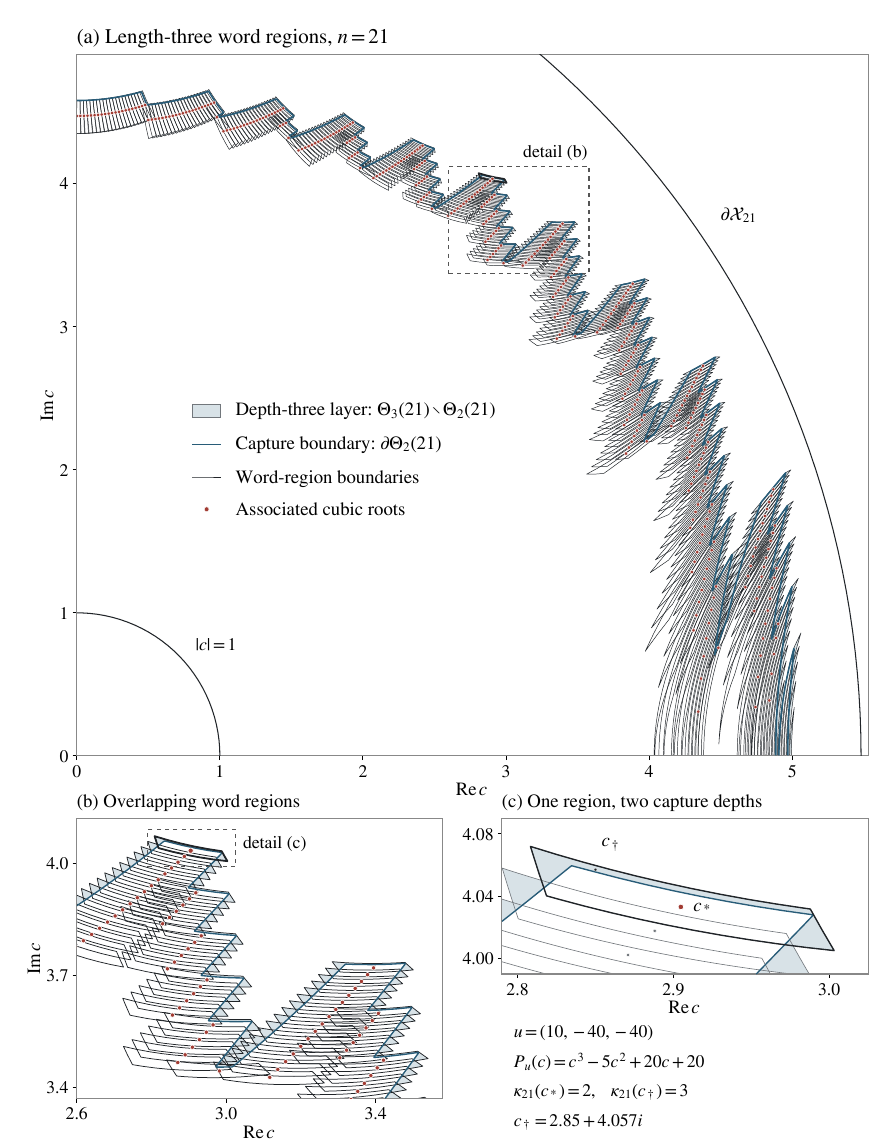}
 \caption{Length-three word regions for $n=21$, with successive enlargements in (b,c). Light blue marks sampled parts of $\Theta_3(21)\setminus\Theta_2(21)$; the blue curve is $\partial\Theta_2(21)$. Thin contours bound selected components of the word regions; red points mark associated cubic roots. The bold contour bounds a component of $\mathcal U_u$, $u=(10,-40,-40)$, containing the root $c_*$ and the black point $c_\dagger$. Their first capture depths are exactly two and three, respectively (\Cref{ex:n21-word-chart-strata}). The guides in (a) are $\partial\mathcal X_{21}$ and $|c|=1$. Numerical contours solve four equations for word-region boundaries, at most quadratic in $|c|^2$ at fixed $\operatorname{Re}c$; retained components contain sampled depth-three parameters.}
 \label{fig:n21-word-charts}
\end{figure}

\section{Recovery and the sharp global threshold}\label{sec:global}
Polynomial-root density and finite landing at zero give recovery by closure.
\subsection{Recovery from polynomial roots}
\Cref{thm:finite-capture-non-real} states that, within $\Xn\setminus\R$, the locus coincides with the
closure of $\Theta_{\mathrm{fin}}(n)$, the union of all capture levels.
\begin{proof}[Proof of \Cref{thm:finite-capture-non-real}]
One inclusion follows from \Cref{prop:finite-capture-interior} and relative closedness of $\Mn$.
Conversely, let roots $c_j\in\Rn$ tend to $c\in(\Mn\cap\Xn)\setminus\R$.
Eventually $c_j\in\Xn\setminus\R$. If $P(z)=z^k+d_1z^{k-1}+\cdots+d_k$ vanishes at $c_j$,
the word $u=t_0\cdots t_{k-1}$ with $t_{r-1}=-2d_r$ satisfies, by \Cref{prop:dyn-poly},
\[
 g_u(2c_j)=2c_jP(c_j)=0\in\PTrap(c_j,N).
\]
Thus $c_j\in\Theta_{\mathrm{fin}}(n)$, proving the reverse inclusion.
\end{proof}

\begin{corollary}\label{cor:lens-real-trace}
For every $n\ge2$,
\[
 \Mn\cap\Xn=(\overline{\Theta_{\mathrm{fin}}(n)}\cup(\partial\Mn\cap\R))\cap\Xn.
\]
\end{corollary}
\begin{proof}
The non-real identity is Theorem~\ref{thm:finite-capture-non-real}. A real interior point of
$\Mn\cap\Xn$ is a limit of non-real interior points in the same open set, so it too belongs to
$\overline{\Theta_{\mathrm{fin}}(n)}$. The remaining real points lie in $\partial\Mn$. The reverse
inclusion follows from relative closedness of $\Mn$.
\end{proof}

\Needspace{28\baselineskip}
\subsection{Interior parameters without canonical capture}
\label{sec:quadratic-depth}
We now identify quadratic interior parameters missed by every finite capture level,
with logarithmic capture depth in punctured neighborhoods.
Figure~\ref{fig:quadratic-noncapture} illustrates the parameter and dynamical geometry for $n=3$.

\begin{theorem}[Canonical noncapture and logarithmic depth]\label{prop:quadratic-noncapture}
For every $n\ge3$, put $\gamma_n=(1+i\sqrt{4n-1})/2$.
There are an open neighborhood $U_n\subset\Xn\setminus\R$ of $\gamma_n$
and a constant $C_n>0$ such that
\[
 U_n\subset\interior\Mn,\qquad
 U_n\setminus\Theta_{\mathrm{fin}}(n)=\{\gamma_n\},
\]
and, for every $c\in U_n\setminus\{\gamma_n\}$,
\begin{equation}\label{eq:quadratic-depth-law}
 \left|\kappa_n(c)-\frac{\log(1/|c-\gamma_n|)}{\log\sqrt n}\right|\le C_n.
\end{equation}
\end{theorem}
The proof is in Appendix~\ref{app:quadratic-depth}. At $\gamma_n$, every finite inverse state
is a nonzero lattice point separated from $\overline{\PTrap(\gamma_n,N)}$; this separation gives
the lower bound. Nearby, repeating a four-digit word magnifies $c^2-c+n$ until one of finitely
many suffixes gives entry, yielding the upper bound.

Rearranging \eqref{eq:quadratic-depth-law} shows that, if $U_n$ is a sufficiently small disk, there
are $\varrho_n^-,\varrho_n^+>0$ such that for every sufficiently large $k$
\begin{equation}\label{eq:quadratic-resolution}
 B(\gamma_n,\varrho_n^- n^{-k/2})\subset U_n\setminus\Theta_k(n)
 \subset B(\gamma_n,\varrho_n^+ n^{-k/2}).
\end{equation}
For large $k$, the component $O_k$ of $\C\setminus\Theta_k(n)$ containing $\gamma_n$ obeys the same bounds:
a surrounding circle in $U_n\cap\Theta_k(n)$ confines it to $U_n$, while the inner disk lies in $O_k$
by connectedness. Thus $O_k\subset\interior\Mn$ has diameter comparable to $n^{-k/2}$ and area
comparable to $n^{-k}$. A hole in a finite capture level need not be a hole in the locus.

The depth law also bounds the degree of polynomial roots near $\gamma_n$. If $c\in U_n$ is a
root of a monic integer polynomial of degree $d$ and height at most $n-1$, the corresponding
inverse word sends $2c$ to zero. Hence $c\ne\gamma_n$ and $\kappa_n(c)\le d$, so
\begin{equation}\label{eq:quadratic-degree-barrier}
 d\ge\frac{\log(1/|c-\gamma_n|)}{\log\sqrt n}-C_n.
\end{equation}
Thus $|c-\gamma_n|\ge n^{-(d+C_n)/2}$: finite-degree root plots leave a gap around this interior
parameter. The lattice obstruction also gives an explicit bound without the unknown constant
$C_n$. For any degree-$d$ monic integer polynomial of height at most $n-1$, every root $c$
satisfies
\begin{equation}\label{eq:explicit-degree-radius}
 |c-\gamma_n|\ge R_{n,d}:=
 \frac{\sqrt n-1}{(\sqrt n+1)n^{d/2}-\sqrt n}
 >\frac{\sqrt n-1}{\sqrt n+1}\,n^{-d/2}.
\end{equation}
The proof is given in Appendix~\ref{app:quadratic-depth}. Since $R_{n,d}$ decreases with $d$,
the open disk of this radius contains no such root of degree at most $d$.
The bound~\eqref{eq:explicit-degree-radius} is of Liouville type: as in Liouville's theorem, the
proof rests on the principle that a nonzero integer has absolute value at least~$1$, applied here
to the squared modulus of the polynomial at~$\gamma_n$. A related repulsion between algebraic
numbers of small height is visible in the algebraic starscapes of~\cite{HarrissStangeTrettel2022}.

\begin{figure}[p]
\centering
\includegraphics[width=.87\linewidth]{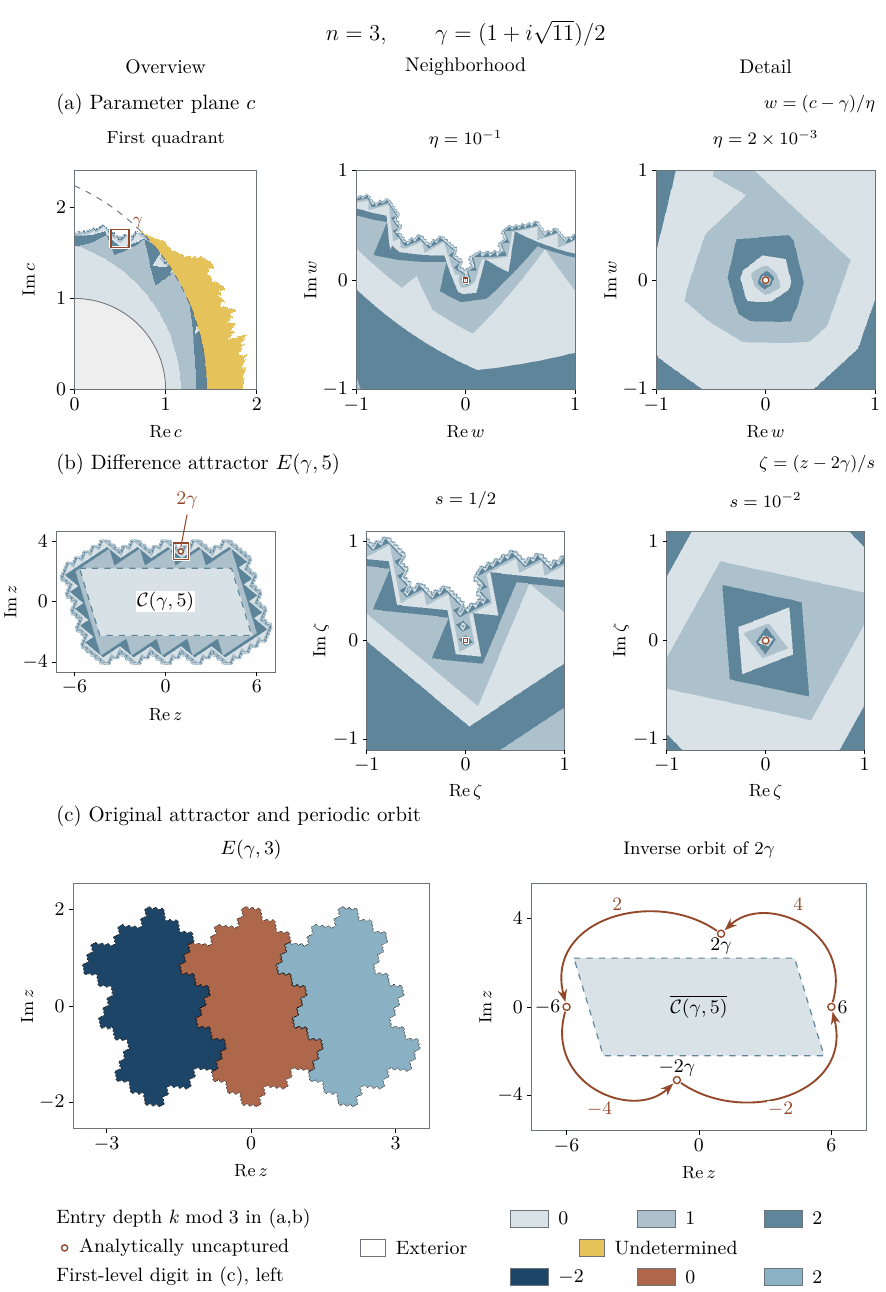}
\caption{Canonical noncapture at $\gamma=(1+i\sqrt{11})/2$.
(a) First-quadrant parameter overview and two close-ups at $\gamma$.
(b) Difference attractor $E(\gamma,5)$ and two close-ups at $2\gamma$,
with the canonical entry region outlined.
(c) Original attractor $E(\gamma,3)$ colored by first digit (left), and exact periodic inverse
orbit of $2\gamma$ (right); arrows carry inverse digits and the closed entry region is shaded.
Boxes locate zooms. Circled centers in (a,b) are analytically uncaptured;
all other detail samples enter by depth $18$ in (a) and $19$ in (b).
Depth colors repeat modulo three; yellow denotes undetermined outcomes.
The dashed curve in (a) bounds the lens.
Plots in (a) and (c, left) are numerical; entry depths in (b) are exact at grid centers,
not throughout pixels. See Appendix~\ref{supp:quadratic-illustration}.}
\label{fig:quadratic-noncapture}
\end{figure}

\FloatBarrier

\subsection{When the lens contains the non-real locus}
The radial bound~\cite[Theorem~2]{BeaucoupBorweinBoydPinner1998Ray}, in the form of~\cite[Proposition~4]{EJSn24}, is
\begin{equation}\label{eq:known-radial-bound}
 \Mn\setminus\R\subset\{c:|c|<1+\sqrt{n-1}\}.
\end{equation}
For $n\ge21$, $1+\sqrt{n-1}<\sqrt{2n}-1$ gives lens containment, as proved in
\cite[Theorem~2]{EJSn24}. Appendix~\ref{supp:radial-bound} gives an enclosure proof of the radial
bound. Using both enclosure inequalities settles $n=20$.

\begin{theorem}[Sharp threshold for lens containment]\label{thm:lens-regime}
For every $n\ge20$, $\Mn\setminus\R\subset\Xn$. For each $2\le n\le19$ there is a non-real point of $\Mn\setminus\Xn$.
The latter assertion is certified by the disk tests in Appendix~\ref{supp:off-lens-certificates}.
\end{theorem}
\begin{proof}
The range $n\ge21$ follows from \eqref{eq:known-radial-bound}.
Proposition~\ref{prop:M20-lens} in Appendix~\ref{app:twenty}
proves the case $n=20$.
Proposition~\ref{prop:off-lens-nlt20} supplies the counterexamples
for $2\le n\le19$.
\end{proof}

\begin{corollary}\label{cor:density-trap-captured}\label{cor:global-finite-capture-description}
For every $n\ge2$, the set $\Theta_{\mathrm{fin}}(n)$ is open and dense in
$(\interior\Mn\cap\Xn)\setminus\R$. If $n\ge20$, it is dense in $\interior\Mn\setminus\R$, and
\begin{equation}\label{eq:global-closure}
 \closure{\Theta_{\mathrm{fin}}(n)}=\closure{\interior\Mn},\qquad
 \Mn=\closure{\Theta_{\mathrm{fin}}(n)}\cup(\partial\Mn\cap\R).
\end{equation}
\end{corollary}
\begin{proof}
Openness and lens density follow from \Cref{prop:finite-capture-interior,thm:finite-capture-non-real}.
For $n\ge20$, lens containment extends density to every non-real point of $\Mn$.
Each real interior point is a limit of non-real interior points, so it also lies in
$\closure{\Theta_{\mathrm{fin}}(n)}$. The remaining real points lie in $\partial\Mn$.
The reverse inclusions follow from $\Theta_{\mathrm{fin}}(n)\subset\interior\Mn$ and relative closedness.
\end{proof}

For $n=2$, density of the interior at every non-real point is the theorem of
Calegari, Koch and Walker~\cite{Calegari2017RootsConjecture}. For $n>2$, our earlier work
established this density within the lens, and hence throughout the non-real locus when
$n\ge21$~\cite{EJSn24}. \Cref{cor:density-trap-captured} adds $n=20$ and identifies an explicit
dense set of certified interior parameters. The real boundary points in \eqref{eq:global-closure}
remain separate.

\section{Finite certification and branching bounds}\label{sec:certification}
Finite entry and exterior exclusion use the same inverse tree. Exterior exclusion is the standard
inverse-iteration method used to draw attractors of iterated function systems;
\Cref{prop:escape-depth} quantifies it. Entry into an inner region certifies
interior membership. An outer enclosure permits \emph{pruning}: a word and all its extensions
are discarded as soon as one of its prefix iterates is proved to lie outside the enclosure.
\Needspace{14\baselineskip}
\subsection{The exact canonical enclosure}
The enclosure half-widths are the largest absolute vertical and slanted
coordinates attained by the attractor. Each is obtained by maximizing
the corresponding projection of the digit series, choosing the digits
independently.
\begin{definition}[Canonical enclosure]\label{def:Enc}
	Let $c=x+iy\in\C\setminus(\Dbar\cup\R)$, and write $c=\rho e^{i\theta}$. Define
	\begin{equation}\label{eq:VE-series-Y}
		\VE(c,N)\coloneqq(N-1)\sum_{k=1}^{\infty}\bigl|\lv(c^{-k})\bigr|
		=(N-1)\sum_{k=1}^{\infty}\rho^{-k}\,|\sin(k\theta)|.
	\end{equation}
	Set
	\begin{equation}\label{eq:SE-from-VE-def}
		\SE(c,N)\coloneqq(N-1)\sum_{k=0}^{\infty}\bigl|\ls(c^{-k};c)\bigr|
		=(N-1)\frac{|y|}{\rho}+\frac{1}{\rho}\VE(c,N).
	\end{equation}
	The \emph{canonical enclosure} is
	\[
		\PEnc(c,N)\coloneqq\overline{\mathcal P_c\!\bigl(\SE(c,N),\,\VE(c,N)\bigr)}.
	\] 
\end{definition}

\begin{proposition}[Exact canonical enclosure]\label{prop:enc-exact}
	Let $c=x+iy$ with $|c|>1$ and $y\ne0$. Then
	\[
		\sup_{z\in E(c,N)} |\lv(z)| = \VE(c,N),
		\qquad
		\sup_{z\in E(c,N)} |\ls(z;c)| = \SE(c,N).
	\]
	Thus $\PEnc(c,N)$ is the smallest enclosing parallelogram in the canonical directions, even among translated parallelograms.
\end{proposition}

\begin{proof}
Let $\varphi\colon\C\to\R$ be real-linear. For $z=\sum_{k\ge0}a_kc^{-k}\in E(c,N)$,
\[
 |\varphi(z)|\le(N-1)\sum_{k\ge0}|\varphi(c^{-k})|.
\]
Equality holds by choosing $a_k=N-1$ when $\varphi(c^{-k})\ge0$ and $a_k=-(N-1)$ otherwise.
Symmetry gives the negative extreme. For $\varphi=\lv$, the sum is $\VE$. For $\varphi=\ls(\,\cdot\,;c)$, the identities
$\ls(1;c)=y/\rho$ and $\ls(c^{-k};c)=\lv(c^{1-k})/\rho$ for $k\ge1$ give \eqref{eq:SE-from-VE-def}.
Attainment of both extremes forces every enclosing strip in either direction to contain the
corresponding strip of $\PEnc$.
\end{proof}

Together, \Cref{cor:self-covering,prop:enc-exact} give
\[
 \PTrap(c,N)\subset\operatorname{int}E(c,N)\subset\PEnc(c,N)
 \qquad(c\in\Xn\setminus\R).
\]
Entering the open inner region certifies interior membership; leaving the closed outer region prunes a branch.
\subsection{Interior and exterior certification}
\begin{definition}[Admissible inverse words]\label{def:enclosure-admissible}
Fix $c\in\C\setminus(\Dbar\cup\R)$. A word $u\in A_N^{\ell}$, $\ell\ge0$,
is \emph{admissible} if
\[
 g_{u'}(2c)\in\PEnc(c,N)\qquad\text{for every prefix }u'\preceq u,
\]
including the empty prefix. These words form the \emph{enclosure tree}:
each vertex is a word, and its children are its admissible one-digit
extensions. The tree is empty if $2c\notin\PEnc(c,N)$ and otherwise
includes continuations after entry into the inner region. Replacing $2c$
by $z$ defines the tree rooted at $z$. Distinct words remain distinct
vertices even when their endpoints coincide.
\end{definition}

Upper bounds $\widehat{\SE}\ge\SE$ and $\widehat{\VE}\ge\VE$, established
by exact arithmetic or directed rounding, define a verified outer enclosure
\[
 \widehat{\PEnc}
 =\overline{\mathcal P_c(\widehat{\SE},\widehat{\VE})}
 \supseteq E(c,N).
\]
If $1<r\le|c|$, the tail after $M$ terms of \eqref{eq:VE-series-Y} is at most
$(N-1)r^{-M}/(r-1)$. Exact or outward-rounded evaluation, followed by
\eqref{eq:SE-from-VE-def}, therefore bounds both half-widths.

\begin{proposition}[Exterior escape depth]\label{prop:escape-depth}
Let $Y$ be any bounded enclosure containing $E(c,N)$, and set $D_Y=\operatorname{diam}Y$. If an
inverse orbit starting at $z$ stays in $Y$ through depth $k$, then
\[
 \operatorname{dist}(z,E(c,N))\le D_Y|c|^{-k}.
\]
Consequently, a point at distance $\delta>0$ from $E(c,N)$ has no surviving branch once $|c|^k\delta>D_Y$.
\end{proposition}
\begin{proof}
For each digit $t$,
\[
 \operatorname{dist}(g_t(w),E(c,N))
 =|c|\operatorname{dist}(w,f_t(E(c,N)))
 \ge |c|\operatorname{dist}(w,E(c,N)).
\]
At depth $k$, the distance is therefore at least $|c|^k\operatorname{dist}(z,E(c,N))$ and at most
$D_Y$, since the endpoint and attractor lie in $Y$.
\end{proof}

\subsection{One search in the parameter and dynamical planes}
\label{sec:three-outcome}\label{sec:dynamical-certification}
\begin{algorithm}[!tbp]
\caption{Finite inverse certification in the parameter and dynamical planes}\label{alg:geom}
\begin{algorithmic}
\Require $c\notin\R$, $|c|>1$; integer $m\ge2$; starting point $z_0\in\C$
\Statex \hspace{\algorithmicindent} depth limit $k_{\max}\ge0$, per-level word limit $L_{\max}\ge1$
\Statex \hspace{\algorithmicindent} a verified bounded enclosure $\widehat{\PEnc}\supset E(c,m)$
\Ensure \Interior, \Exterior, or \Undetermined
\If{$z_0\notin\widehat{\PEnc}$} \Return \Exterior \EndIf
\If{$\mathsf{Entry}(c,m,z_0)$} \Return \Interior\Comment{Record the empty word}\EndIf
\State $W\gets\{(z_0,\emptyword)\}$
\For{$k=1$ \textbf{to} $k_{\max}$}
  \State $W'\gets\emptyset$; $\mathit{overflow}\gets\mathrm{false}$
  \ForAll{$(z,u)\in W$}
    \ForAll{$t\in A_m$ with $g_t(z)\in\widehat{\PEnc}$}
      \If{$\mathsf{Entry}(c,m,g_t(z))$} \Return \Interior\Comment{Record $ut$}\EndIf
      \If{$|W'|<L_{\max}$}
        \State $W'\gets W'\cup\{(g_t(z),ut)\}$
      \Else
        \State $\mathit{overflow}\gets\mathrm{true}$
      \EndIf
    \EndFor
  \EndFor
  \If{$\mathit{overflow}$} \Return \Undetermined \EndIf
  \If{$W'=\emptyset$} \Return \Exterior \EndIf
  \State $W\gets W'$
\EndFor
\Return \Undetermined
\end{algorithmic}
\end{algorithm}
The starting point need not be $2c$. Fix a non-real $c$ with $|c|>1$, an integer $m\ge2$,
and a point $z_0\in\C$. We test $z_0$ against $E(c,m)$ using digits in $A_m$.
Define $\PTrap(c,m)$ and $\PEnc(c,m)$ by replacing $N$ with $m$ in the canonical inner
and outer widths. The nearest-digit covering argument and the support-width calculation apply
to every $m$, whether even or odd: they use spacing two and extreme digits $\pm(m-1)$.
In particular,
\begin{equation}\label{eq:dynamical-lens}
 |c|^2+2|\operatorname{Re}c|<m
 \quad\Longrightarrow\quad
 \PTrap(c,m)\subset\operatorname{int}E(c,m).
\end{equation}
Set
\[
 \mathsf{Entry}(c,m,z)\quad\Longleftrightarrow\quad
 |c|^2+2|\operatorname{Re}c|<m\ \text{and }z\in\PTrap(c,m).
\]
The three applications are summarized below.
\begin{center}
\begin{tabular}{@{}llll@{}}
\toprule
Test & Digits & Starting point & Domain for canonical entry\\
\midrule
$c\in\Mn$ & $A_{2n-1}$ & $2c$ & $c\in\Xn\setminus\R$\\
$z\in E(c,2n-1)$ & $A_{2n-1}$ & $z$ & $c\in\Xn\setminus\R$\\
$z\in E(c,n)$ & $A_n$ & $z$ & $|c|^2+2|\Real c|<n$\\
\bottomrule
\end{tabular}
\end{center}
Exterior exclusion remains available outside the corresponding entry domain.

In the dynamical plane, entry certifies $z_0\in\operatorname{int}E(c,m)$; in the parameter
plane, entry from $2c$ with $m=2n-1$ certifies $c\in\operatorname{int}\Mn$.
Only a verified outer-enclosure violation permits pruning.
The breadth-first search stores at most $L_{\max}$ word--endpoint pairs at each depth.
After a storage overflow it still tests all remaining children at that depth for entry,
returning \Undetermined{} if no child endpoint enters. The pseudocode uses exact comparisons;
an interval implementation retains uncertain branches or increases precision.

For fixed $c$ with $|c|^2+2|\operatorname{Re}c|<m$, the forward images of $\PTrap(c,m)$ form nested
open entry levels, as in \eqref{eq:forward-word-union}, and their union is dense in $E(c,m)$.
Indeed, if $a\in\PTrap(c,m)$ and $z=f_u(z_k)$ with $z_k\in E(c,m)$ and $|u|=k$, then
$|z-f_u(a)|\le |c|^{-k}\operatorname{diam}E(c,m)$.
Density does not assert finite capture of every interior point.

\begin{proposition}[Correctness and minimum capture depth]\leavevmode\label{prop:algorithm-correctness}%
An \Interior\ verdict from \Cref{alg:geom} implies $z_0\in\operatorname{int}E(c,m)$;
an \Exterior\ verdict implies $z_0\notin E(c,m)$.
If a capture word exists, exact execution finds one of least length whenever the depth and
per-level word limits permit it. For $m=2n-1$ and $z_0=2c$, the corresponding verdicts are
$c\in\operatorname{int}\Mn$ and $c\notin\Mn$, and the shortest length is $\kappa_n(c)$.
\end{proposition}
\begin{proof}
Entry gives $g_u(z_0)\in\PTrap(c,m)\subset\operatorname{int}E(c,m)$. Since
$f_u(\operatorname{int}E(c,m))\subset\operatorname{int}E(c,m)$, the interior verdict is sound.
If $z_0\in E(c,m)$, every prefix of an address of $z_0$ sends $z_0$ into $E(c,m)$ and survives pruning,
so the set of retained words cannot become empty. Thus exterior verdicts are sound.
For a successful word $u=vw$,
\[
 g_v(z_0)=f_w(g_u(z_0))\in E(c,m),
\]
so every prefix survives. Breadth-first traversal then finds the least entry depth when
budgets permit. The parameter verdicts follow from \Cref{prop:finite-capture-interior}
and \eqref{eq:intro-root-closure}.
\end{proof}

With interval arithmetic, a verified entry word of length $k$ certifies capture by depth $k$.
It certifies the first entry depth only when every shorter word has been ruled out; an unresolved
earlier entry test prevents that conclusion.

Entry certifies an open neighborhood, not merely a point. Suppose
$|c|^2+2|\Real c|<m$ and let $w=g_u(z_0)\in P=\PTrap(c,m)$, with $|u|=k$.
Writing $c=x+iy$ and $q=|c|^2$, set
\begin{equation}\label{eq:certified-ball-radius}
 r_u=|c|^{-k}\min\left\{
 \frac{m|y|-|\operatorname{Im}(cw)|}{|c|},
 \frac{(m-2|x|)|y|-q|\operatorname{Im}w|}{q}\right\}>0.
\end{equation}
These are the distances from $w$ to the supporting lines of $P$, scaled by $|c|^{-k}$.
Applying $f_u$ to the corresponding open ball gives
\[
 B(z_0,r_u)\subset f_u(P)\subset\operatorname{int}E(c,m).
\]
This is a dynamical-plane radius at fixed $c$; a parameter neighborhood instead follows
from persistence of the strict word-region inequalities.

\subsection{Capture layers and virtual magnification}\label{sec:scenery}
Bandt's magnification algorithm~\cite{Bandt2025Magnification} follows a reference piece and
updates the relative positions of its neighbors. It applies to equal-ratio planar similarities
without a separation assumption for visualization. In our family, these updates have a simple
relation to capture. Write $E=E(c,m)$, $P=\PTrap(c,m)$, and
$\mathcal H(K)=\bigcup_{a\in A_m}f_a(K)$, with $c$ in the covering domain. Then
$C_k=\mathcal H^k(P)$ are the cumulative dynamical entry regions.

\Needspace{6\baselineskip}
For a reference word $u$ of length $\ell$, applying $f_u^{-1}$ restores the reference piece
$f_u(E)$ to $E$.
Since all length-$\ell$ maps have the same linear part $c^{-\ell}$,
every other piece becomes a translate of $E$ in these coordinates.
For words $u,v$ of length $\ell$, the relative map $h_{u,v}=f_u^{-1}f_v$ is a translation $z\mapsto z+\tau_{u,v}$, and
\begin{equation}\label{eq:scenery-capture}
 \tau_{ua,vb}=c(\tau_{u,v}+b-a),\qquad
 f_u^{-1}(C_{\ell+k})=\bigcup_{|v|=\ell}(C_k+\tau_{u,v}).
\end{equation}
The recurrence is Bandt's neighbor-map rule~\cite[Section~4, equation~(6)]{Bandt2025Magnification}
specialized to translations. The union identity follows by applying $f_u^{-1}$ to
$C_{\ell+k}=\bigcup_{|v|=\ell}f_v(C_k)$. The same identity holds with $E$ in place of both
entry regions.
For $Q\supset E$ a bounded enclosure and $W$ a display window, the same calculation gives
\begin{equation}\label{eq:scenery-inner-outer}
 W\cap\bigcup_{|v|=\ell}(P+\tau_{u,v})
 \ \subset\ W\cap f_u^{-1}(E)
 \ \subset\ W\cap\bigcup_{|v|=\ell}(Q+\tau_{u,v}).
\end{equation}
These inclusions transport certified inner and outer bounds in Bandt's neighbor coordinates.
Capture layers are obtained by subtracting successive unions in \eqref{eq:scenery-capture},
so overlapping neighbors must be treated together.
Bandt's integer-data hypotheses make the neighbor-map updates exact;
entry into $P$ supplies the interior certificates.

Figure~\ref{fig:scenery-capture} gives an exact geometric example with $c=1+i$, $m=5$.
Here $|c|^2+2|\operatorname{Re}c|=4<5$ and the attractor is the octagon
\[
 E=I+c^{-1}I+c^{-2}I+c^{-3}I,\qquad I=[-16/3,16/3].
\]
Indeed, $c^4=-4$ and the five intervals $a-I/4$, $a\in A_5$, cover $I$; grouping an address
by its index modulo four gives the formula. For the reference word $u=(0)$,
$\tau_{u,(a)}=ca$.

\begin{figure}[!htbp]
\centering
\includegraphics[width=.94\linewidth]{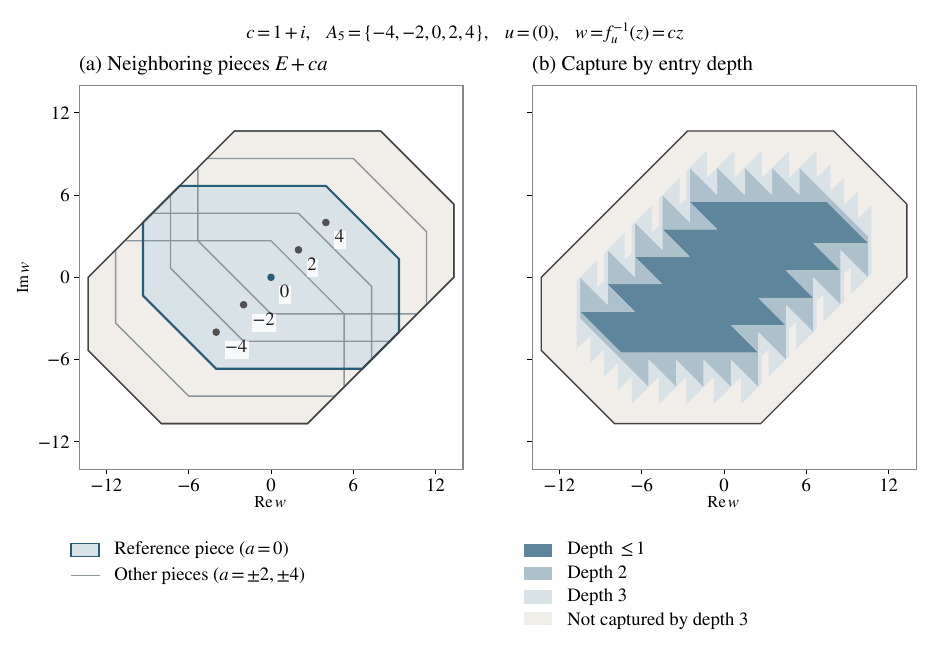}
\caption{Neighbor coordinates and capture for $E(1+i,5)$, with reference word $u=(0)$ and
coordinate $w=f_0^{-1}(z)=cz$. (a) The five exact octagons $E+ca$, $a\in A_5$, whose union
is $cE$; the central piece is highlighted. (b) The cumulative region $cC_1$ and the successive
layers $c(C_2\setminus C_1)$ and $c(C_3\setminus C_2)$. Gray marks the part of $cE$ not
captured by the displayed depth; white is outside $cE$. Both panels use the same window.}
\label{fig:scenery-capture}
\end{figure}

\FloatBarrier
\subsection{Two worked finite certificates}\label{sec:worked-certificates}
Take $n=3$ and $c_0=0.585+1.675i$, in the lens but beyond the connected annulus:
$|c_0|^2=3.14785>3$ and $|c_0|^2+2\operatorname{Re}c_0=4.31785<5$.
The word $(2,-4,-2,4)$ sends the marked point to
\[
 g_{(2,-4,-2,4)}(2c_0)=4.334754422075+0.777291245625i.
\]
All decimals are exact. The vertical and slanted margins in \eqref{eq:certified-ball-radius}
exceed $3.96$ and $0.65$, respectively; division by $|c_0|^2$ and $|c_0|$ gives the corresponding
distances to the supporting lines in \Cref{fig:interior-exterior}(b).
Thus $c_0\in\operatorname{int}\mathcal M_3$. Exhaustive rational checks exclude shorter words,
so $\kappa_3(c_0)=4$. Moreover, \eqref{eq:certified-ball-radius} gives the explicit open set
\begin{equation}\label{eq:c0-certified-ball}
 B(2c_0,3/80)\subset\operatorname{int}E(c_0,5).
\end{equation}
Appendix~\ref{supp:worked-certificates} verifies this radius by rational inequalities.

The monic minimal polynomial of $c_0$ is
$z^2-\tfrac{117}{100}z+\tfrac{62957}{20000}$ over $\mathbb Q$.
Hence $c_0$ is not an algebraic integer: finite entry also detects interior parameters outside the
monic integer polynomial root set.

For the nearby parameter $c_1=0.585+1.705i$, the verified outer bounds are
\[
 |\operatorname{Im}z|\le3.813,\qquad |\operatorname{Im}(c_1z)|\le10.633.
\]
The numbers of retained words through depths zero to four are $1,1,1,1,0$, so $c_1\notin\mathcal
M_3$. The last candidate, along the same word, has
$|\operatorname{Im}(c_1z_4)|=10.64642170337075>10.633$.
Appendix~\ref{supp:worked-certificates} gives orbit data, support-tail bounds and shorter-word exclusions.
\begin{figure}[!htbp]
\centering
\includegraphics[width=\linewidth]{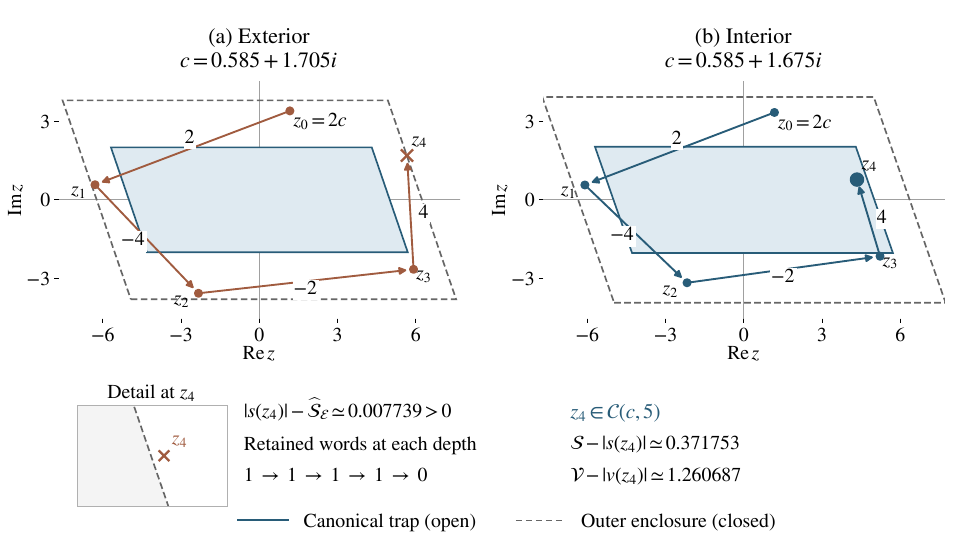}
\caption{Exact certificates for $n=3$, starting from $z_0=2c$ along $(2,-4,-2,4)$. (a) At $c=0.585+1.705i$, the fourth iterate $z_4$ violates the slanted enclosure bound; the inset resolves the gap. (b) At $c=0.585+1.675i$, $z_4$ enters the canonical trap. Arrows carry inverse digits. Here $s=\ell_s$, $v=\ell_v$, and $\widehat{\mathcal S}_{\mathcal E}$ is the computed upper bound for the slanted half-width.}
\label{fig:interior-exterior}
\end{figure}

\FloatBarrier
\subsection{Inner and outer approximations}
We now complement the capture levels with outer approximations to the
non-real locus. A parameter survives to depth $k$ if at least one inverse
branch from $2c$ remains inside the enclosure through that depth.
\begin{definition}[Survival sets]\label{def:survival}
For $k\ge0$, let
\[
 \Xi_k(n)=\{c\in\C\setminus(\Dbar\cup\R):
 \text{an admissible word of length }k\text{ exists}\}.
\]
\end{definition}

\begin{proposition}[Inner and outer approximations]
\label{prop:depth-k-exterior-certificate}\label{prop:Theta-Xi-bounds}
The survival sets decrease with $k$, and
\[
 \Mn\setminus\R=\bigcap_{k\ge0}\Xi_k(n).
\]
For every $k\ge0$,
\[
 \Theta_k(n)\subset\operatorname{int}\Mn,
 \qquad \Mn\setminus\R\subset\Xi_k(n).
\]
If no admissible word remains at some finite depth, the parameter lies outside $\Mn$.
\end{proposition}
\begin{proof}
Taking prefixes gives monotonicity. If $c\in\Mn\setminus\R$, every prefix of an address of $2c$
sends $2c$ into the attractor and passes the enclosure test. Conversely, survival at every depth and
\Cref{prop:escape-depth} give $2c\in E(c,N)$, hence $c\in\Mn$. The inner inclusion follows from
\Cref{prop:finite-capture-interior}.
\end{proof}

Survival without entry is inconclusive: $(\Xi_k(n)\cap\Xn)\setminus\Theta_k(n)$ contains all
non-real boundary points inside the lens and may also contain interior and exterior parameters.
A per-level word limit can leave further parameters undetermined.

\subsection{Digit restrictions and branching bounds}\label{sec:branching}
Each enclosure test restricts the next digit to an interval; their intersection gives a consecutive
block of the alphabet. We use these blocks to prove \Cref{thm:main-bounded-branching}: for non-real
$c$ with $|c|\ge\sqrt n$, every vertex of the exact enclosure tree has at most four children,
and at most three if $n=2$ or $n\ge25$; these bounds and the threshold $25$ are sharp.
Throughout, $c=x+iy$ is non-real, $\rho=|c|>1$, $N=2n-1$, $a=N-1$, and
$H=\VE(c,N)/|y|$.

\begin{proposition}[Bounds on admissible digits]\label{prop:vertical-pruning}
For a parent with canonical coordinates $(s,v)$, the vertically admissible digits are
\[
 \Av(s)=\{t\in A_N:|\rho s-yt|\le\VE(c,N)\},
 \qquad \#\Av(s)\le\lfloor H\rfloor+1.
\]
The number of digits whose children pass both enclosure tests is at most
\begin{equation}\label{eq:joint-digit-bound}
 \min\left\{N,\ 1+\left\lfloor\min\left(H,\frac{a+H}{2|x|}\right)\right\rfloor\right\},
\end{equation}
where $(a+H)/(2|x|)=+\infty$ when $x=0$.
\end{proposition}
\begin{proof}
In the normalized coordinates
\[
 \beta=\frac{v}{y}=\frac{\Imag z}{y},\qquad
 \alpha=\frac{\rho s}{y}=\frac{\Imag(cz)}{y},
\]
the child $g_t(z)$ has coordinates
\begin{equation}\label{eq:normalized-branching}
 \beta'=\alpha-t,\qquad \alpha'=2x(\alpha-t)-\rho^2\beta,
\end{equation}
by \eqref{eq:rectangle-inverse}. The support identity \eqref{eq:SE-from-VE-def} makes the enclosure
tests $|\beta'|\le H$ and $|\alpha'|\le a+H$. These restrict $t$ to intervals of lengths $2H$ and
$(a+H)/|x|$, respectively. The step-two digit spacing proves both bounds. If $x=0$, the second test is independent of $t$.
\end{proof}

\Needspace{13\baselineskip}
\begin{proof}[Proof of \Cref{thm:main-bounded-branching}]
Fix $n\ge2$, put $N=2n-1$, and let $c\notin\R$ satisfy $|c|\ge\sqrt n$.
We prove the stronger pointwise assertion that, for every parent point
$z\in\C$, at most $\nu_n$ digits $t\in A_N$ satisfy
$g_t(z)\in\PEnc(c,N)$, where
\[
 \nu_n=\begin{cases}
 3,&n=2\text{ or }n\ge25,\\
 4,&3\le n\le24.
 \end{cases}
\]
Write $c=\rho e^{i\theta}$. Since $\sin\theta\ne0$, the support series gives
\begin{equation}\label{eq:branch-vertical-estimate}
 H=\frac{a}{\rho|\sin\theta|}\sum_{j\ge1}\rho^{-j}|\sin(j\theta)|
 <\frac{a}{(\rho-1)^2}
 \le\frac{2(\sqrt n+1)}{\sqrt n-1}.
\end{equation}
Here $|\sin(j\theta)|\le j|\sin\theta|$, with strict inequality at $j=2$.
Thus $H<4$ for $n\ge9$ and $H<3$ for $n\ge25$. The case $n=2$ has only three digits. By
\eqref{eq:joint-digit-bound}, it remains to treat $3\le n\le8$.

Fix $\theta$ and decrease $\rho$ to $r=\sqrt n$. Every summand defining $H$ increases, as does
$(a+H)/(2\rho|\cos\theta|)$ when $\cos\theta\ne0$. It suffices to bound the inner minimum in
\eqref{eq:joint-digit-bound} at $\rho=r$. In the remainder of this estimate, $H$ is evaluated at
$c=re^{i\theta}$.
For $0<\chi<1$, geometric summation gives
\[
 \sum_{j\ge1}\chi^j=\frac \chi{1-\chi},\qquad
 \sum_{j\ge1}\chi^j\sin^2(j\theta)
 =\frac{\chi(1+\chi)\sin^2\theta}{(1-\chi)(1-2\chi\cos(2\theta)+\chi^2)}.
\]
The weighted Cauchy--Schwarz inequality gives
\[
 \left(\sum_{j\ge1}\chi^j|\sin(j\theta)|\right)^2
 \le\left(\sum_{j\ge1}\chi^j\right)
       \left(\sum_{j\ge1}\chi^j\sin^2(j\theta)\right).
\]
Substitution of $\chi=1/r$ and $a=2(r^2-1)$ therefore yields
\begin{equation}\label{eq:branch-cauchy-schwarz}
 H\le C_r(\xi):=\frac{2(r+1)^{3/2}}{\sqrt{r(r+1)^2-4\xi^2}},
 \qquad \xi=r|\cos\theta|\in[0,r).
\end{equation}
The function $C_r$ is increasing and strictly convex on $[0,r]$. Put
$\xi_0=(r^2+1)/4$, so that $8\xi_0-a=4$ and $0<\xi_0<r$ for $\sqrt3\le r\le\sqrt8$.
We claim
\begin{equation}\label{eq:branch-cs-endpoints}
 C_r(\xi_0)<4,\qquad C_r(r)<8r-a.
\end{equation}
After squaring positive quantities, the first inequality is equivalent to
\[
 -r^4+3r^3+3r^2+r-2
 =r^3(3-r)+3r^2+r-2>0.
\]
For $r=\sqrt3$, the second reduces to $1192\sqrt3-2056>0$, which follows from $298^2\cdot3>514^2$.
For $2\le r\le\sqrt8$, the function
\[
 C_r(r)^2=\frac{4(r+1)^3}{r(r-1)^2}
\]
has logarithmic derivative $(1-5r)/(r(r-1)(r+1))<0$. Hence $C_r(r)\le3\sqrt6<8$, whereas
$8r-a=2(5-(r-2)^2)>8$. This proves \eqref{eq:branch-cs-endpoints} for all six integer values $3\le
n\le8$.

If $\xi\le\xi_0$, monotonicity gives $H\le C_r(\xi)<4$. On $[\xi_0,r]$, convexity and the endpoint
inequalities put $C_r(\xi)$ strictly below the line $8\xi-a$. Thus $(a+H)/(2\xi)<4$ in that range.
In either case the inner minimum in \eqref{eq:joint-digit-bound} is less than four, so at most four
digits survive. Induction proves the depth bound. The sharpness statements follow from
\Cref{prop:branch-sharpness} below.
\end{proof}

For $|c|\ge\sqrt n$, the vertical digit interval alone still satisfies
\[
 \#\Av(s)\le
 \left\lceil\frac{2(\sqrt n+1)}{\sqrt n-1}\right\rceil.
\]
The four-child bound uses both tests; a program generating all vertical candidates before applying
the slanted test can examine more than four digits at a vertex. Each test defines an interval, so
their intersection can instead be computed before enumerating its consecutive digits.

\begin{figure}[p]
\centering
\includegraphics[page=1,width=.87\linewidth]{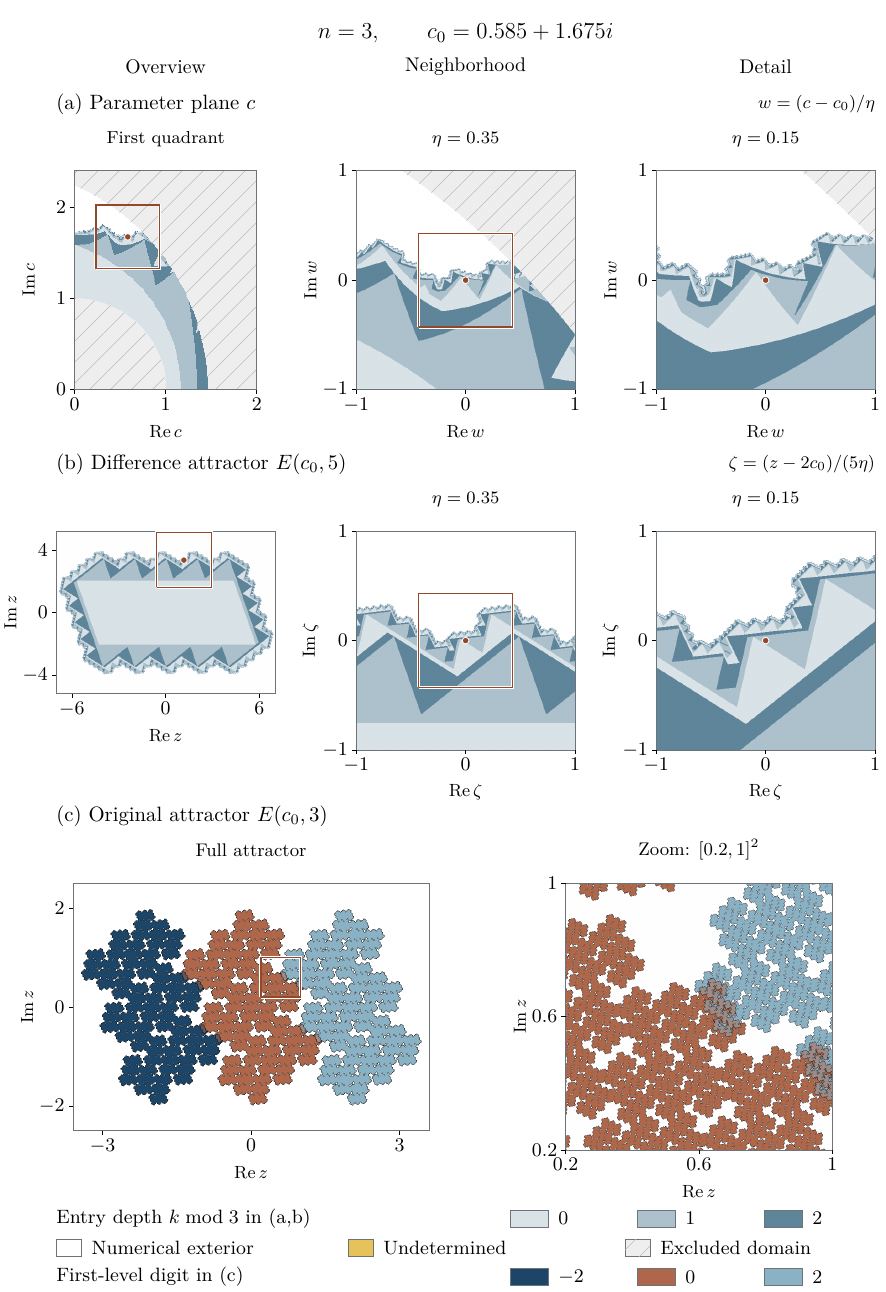}
\caption{Capture-depth plots and first-level pieces at $n=3$, $c_0=0.585+1.675i$.
(a) First-quadrant parameter overview, $0\le\Real c\le2$, and two close-ups at $c_0$.
(b) Full difference attractor $E(c_0,5)$ and two close-ups at $2c_0$.
Local coordinates are $w=(c-c_0)/\eta$ and $\zeta=(z-2c_0)/(5\eta)$,
with $\eta=0.35$ and $0.15$.
(c) Full original attractor $E(c_0,3)$ and one enlarged view of $[0.2,1]^2$,
colored by first digit. Dots and boxes locate the centers and windows.
In (a,b), blue tones record first detected entry depth modulo three; white denotes numerical
exterior exclusion, yellow undetermined outcomes, and hatching excluded samples.
Settings are in Appendix~\ref{supp:fields-pieces}.}
\label{fig:n3-parameter-attractors}
\end{figure}

\begin{proposition}[Sharpness of the uniform branching bounds]\label{prop:branch-sharpness}
For $n=24$ and $c=49/10+i/1000$, the enclosure tree rooted at $2c$ contains the word
\begin{equation}\label{eq:branch-sharp-word}
 u=(18,-32,-42,12,-18,18),
\end{equation}
whose endpoint has exactly four admissible children, with digits $38,40,42,44$.
For every $n\ge2$, at $c=i\sqrt n$ the root vertex $2c$ has exactly three admissible children, with digits $-2,0,2$.
\end{proposition}
\begin{proof}
For the first assertion, $\rho^2=24010001/10^6>24$ and $a=46$. Put $r_0=49/10$ and
\[
 L=\frac{46}{|y|}\sum_{j=1}^{16}|\Imag(c^{-j})|,
 \qquad U=L+\frac{46}{|y|}\frac{r_0^{-16}}{r_0-1}.
\]
The series and its geometric tail give $L\le H\le U$, and rational arithmetic gives
\begin{equation}\label{eq:branch-support-interval}
 3.024325<L\le H\le U<3.024327.
\end{equation}
The exact recurrence \eqref{eq:normalized-branching} puts the endpoints of every prefix and of the four listed children
in the coordinate intervals of Table~\ref{tab:branch-witness}. These intervals lie in
$|\beta|<3.023$, $|\alpha|<49.023$, hence strictly inside the enclosure by the lower bound in
\eqref{eq:branch-support-interval}. At the length-six endpoint, $40.9779<\alpha<40.9781$; every other
even digit in $A_{47}$ has $|\alpha-t|>3.024327$ and is excluded by the upper bound. Thus all
prefixes survive and exactly four children pass both tests.

For $c=i\sqrt n$, the support series sums to $H=2$ and $a+H=2n$. The root vertex has $(\beta,\alpha)=(2,0)$
and lies in the closed enclosure. Its children have $(\beta',\alpha')=(-t,-2n)$. Precisely $t=-2,0,2$
satisfy both inequalities. These examples prove the uniform sharpness of four and the least eventual
three-child threshold $25$; three cannot be reduced to two in its stated range.
\end{proof}

Along \eqref{eq:branch-sharp-word}, the orbit of $2c$ already enters the inner region at depth two, so the
four-child vertex proves sharpness for the full enclosure tree, not a worst-case running time for
\Cref{alg:geom}, which stops upon capture; nor does it establish sharpness separately for each
$3\le n\le24$. The bound remains valid for sufficiently accurate verified upper bounds for the widths at each
fixed parameter in the stated range (\Cref{cor:working-branching}).

\begin{figure}[!t]
\centering
\includegraphics[width=\linewidth]{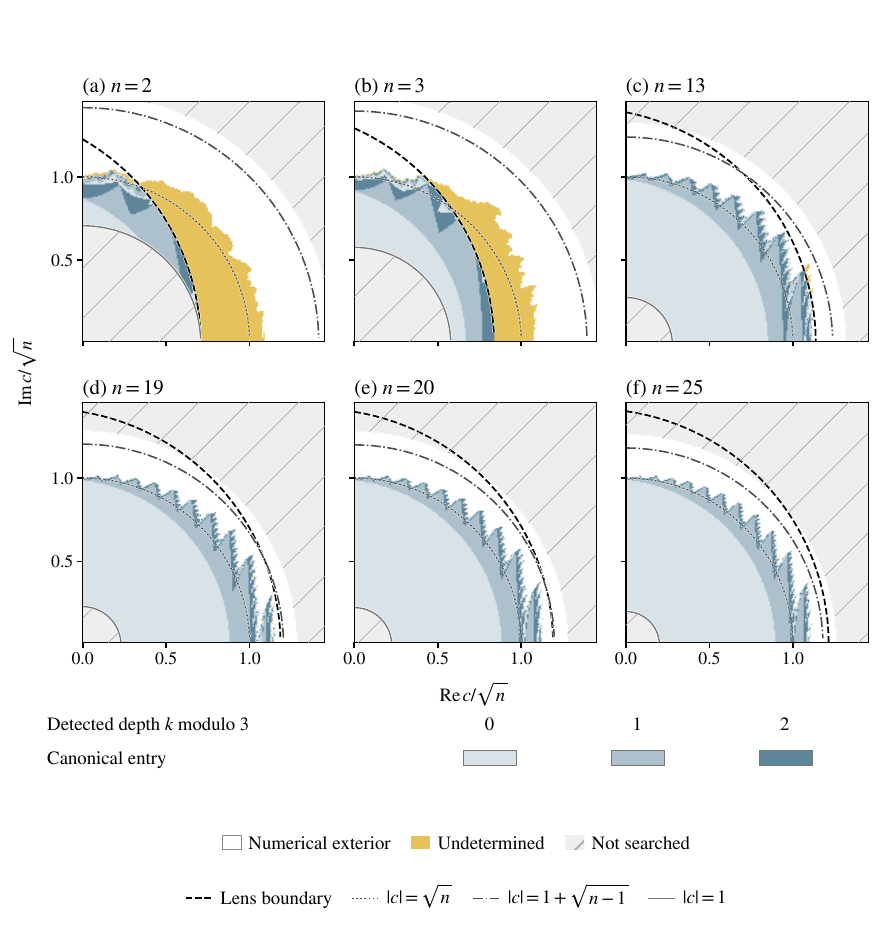}
\caption{Numerical search outcomes for $n=2,3,13,19,20,25$ at depth seven and per-level word limit $1200$, in coordinates $c/\sqrt n$. Blue tones record first detected entry depth modulo three. White denotes numerical exterior exclusion; yellow marks undetermined centers. Hatched cells were not searched. Curves show the lens boundary and radial bounds.}
\label{fig:multi-n-atlas}
\end{figure}

\section{Finite-depth computations}\label{sec:experiments}
We compare searches across numbers of maps, inner regions, and the two planes.
All searches use $g_t(z)=c(z-t)$, $t\in A_N$, with $N=2n-1$; parameter searches start at $2c$.
The plots record floating-point outcomes at sampled points: first detected entry, numerical
exterior exclusion, or undetermined. Appendix~\ref{app:computations} gives the numerical methods.

For $n=3$ and $c_0=0.585+1.675i$, \Cref{fig:n3-parameter-attractors} compares capture layers
near $c_0$ and $2c_0$ (see also \Cref{sec:local-fields}) and shows the overlapping first-level
pieces of $E(c_0,3)$.

\subsection{Common search limits for different numbers of maps}
For each $2\le n\le25$, we search a $180\times180$ grid in the same normalized
parameter window, with depth limit seven and per-level word limit $1200$.
Blue tones record first detected entry depth modulo three: colors repeat every three levels,
so equal colors need not indicate equal depths.
Figure~\ref{fig:multi-n-atlas} shows six choices of $n$.

Across all $482\,178$ searched centers, there are $286\,272$ canonical detections,
$177\,828$ exterior exclusions and $18\,078$ undetermined outcomes.
The accompanying archive provides the complete count table for all $2\le n\le25$. Outside the lens, canonical entry is
unavailable: a searched center is either excluded by the outer test or remains undetermined.

The undetermined counts need not decrease with $n$: they rise from $46$ to $86$
between $n=13$ and $n=14$. The grids are fixed in $c/\sqrt n$, whereas the
monotonicity in $n$ of \eqref{eq:arity-monotonicity} concerns exact capture at fixed $c$.
These counts under fixed search limits estimate neither area nor a convergence rate.

\begin{figure}[!hbp]
\centering
\includegraphics[width=.92\linewidth]{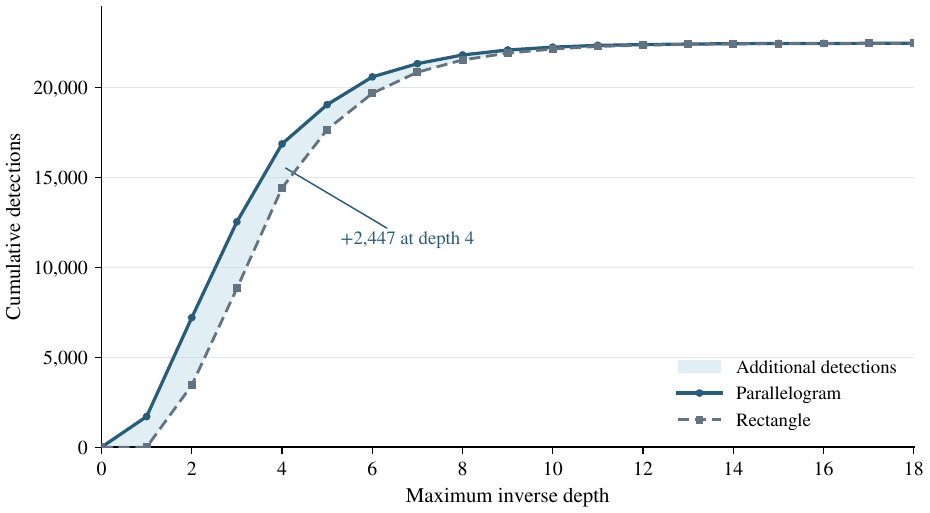}
\caption{Cumulative interior detections on the same $57\,347$-parameter sample for $n=3$.
Only the inner region changes; both tests use the same inverse tree and outer enclosure.
The blue curve represents canonical parallelogram entry and the dashed slate curve
rectangle entry; the pale band shows the additional detections by the former
at each depth. At depth four the counts are $16\,850$ and $14\,403$;
by depth eighteen they are $22\,443$ and $22\,441$.}
\label{fig:controlled-regions}
\end{figure}

\begin{figure}[p]
\centering
\includegraphics[width=\linewidth,height=.85\textheight,keepaspectratio]{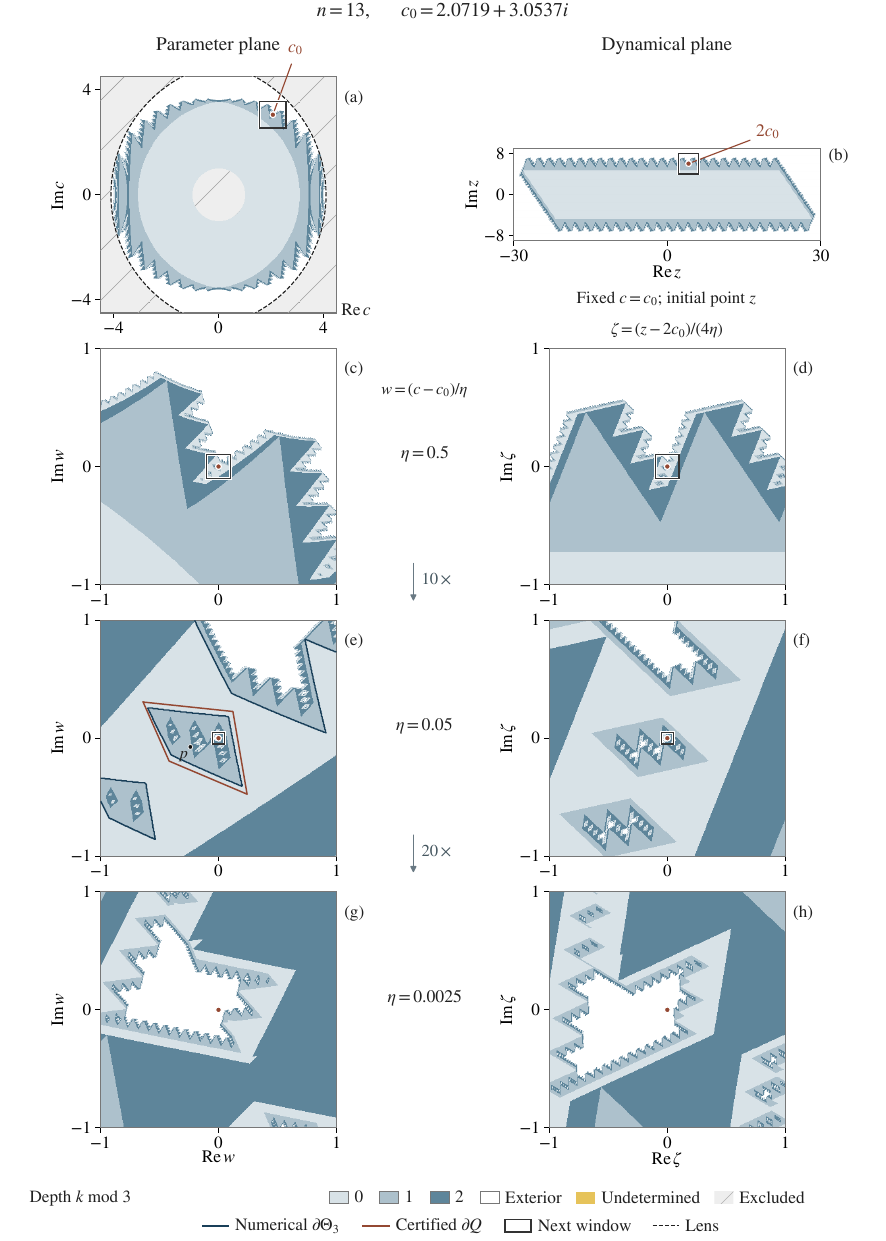}
\caption{First detected entry depths, modulo three, for $n=13$ at $c_0=2.0719+3.0537i$.
Local coordinates are $w=(c-c_0)/\eta$ and $\zeta=(z-2c_0)/(4\eta)$.
Overviews (a,b) lead to $\eta=0.5$ (c,d), $0.05$ (e,f), and $0.0025$ (g,h); boxes locate
the next windows. Panel (e) overlays a numerical contour for $\partial\Theta_3(13)$ and the certified polygon
$\partial Q\subset\Theta_3(13)$ around $p=2.06+3.05i\in\Theta_4(13)\setminus\Theta_3(13)$
(Appendix~\ref{supp:low-depth-topology}). The dynamical panels use the difference
attractor $E(c_0,25)$. White marks numerical exterior samples.}
\label{fig:M13context}
\end{figure}

\subsection{Earlier detection with the parallelogram}\label{sec:controlled-comparison}
By \Cref{prop:rectangle-in-trap}, replacing the rectangle by the parallelogram cannot delay
exact entry. We measure the gain by comparing
\[
 P_c=\PTrap(c,N),\qquad R_c=R(c,N).
\]
For $n=3$, retain the $57\,347$ centers of a $480\times270$ grid in
$[0.15,0.95]\times[1.5,1.95]$ satisfying $q>3$ and $q+2|x|<5$.
Here $c=x+iy$ and $q=|c|^2$. The sample lies beyond the connected annulus and inside the lens.
Both tests use the same inverse tree and outer enclosure; the search continues
after entry into $P_c$ until entry into $R_c$ or depth $18$. The per-level word limit of $16\,384$ is never
reached. \Cref{fig:controlled-regions} compares the cumulative detections directly;
Table~\ref{tab:controlled-depths} gives the corresponding counts.

Of the $22\,441$ parameters detected by both tests by depth $18$, the parallelogram detects
$14\,732$ ($65.65\%$) earlier: one step earlier at $14\,394$ and two at $338$.
Two further samples were detected with $P_c$ but not with $R_c$, leaving five and seven undetermined
samples, respectively.

On separate normalized grids for $n=2,3,5,10,20,25$, earlier detection ranges from $65.05\%$ to $78.74\%$ of the jointly
detected parameters, while the final numbers detected by the two inner tests nearly agree
(Table~\ref{tab:controlled-arities}). Appendix~\ref{supp:controlled-settings} gives
the grids, stopping limits, and numerical comparison tolerances.

\subsection{Parameter and dynamical layers}\label{sec:local-fields}
At $n=13$ and $c_0=2.0719+3.0537i$, \Cref{fig:M13context} compares parameter search
from $2c$ with membership search in $E(c_0,25)$ at fixed $c_0$.
Each panel uses $360\times360$ centers, depth limit $20$, and per-level word limit $4096$.
Every searched center in the local panels receives a numerical entry or exclusion outcome
within these limits.

Panel (e) overlays the certified polygon $\partial Q\subset\Theta_3(13)$ around
$p\in\Theta_4(13)\setminus\Theta_3(13)$.
Appendix~\ref{supp:low-depth-topology} proves that the complementary component containing $p$
is unbounded at depth two and bounded at depth three. This enclosed finite-level gap
therefore contains interior parameters of $\Mn$.
The finest window lies inside $Q$ and resolves later capture and numerical exterior samples.

\subsection{Further questions}
The uncaptured interior set
$(\interior\Mn\cap\Xn)\setminus(\R\cup\Theta_{\mathrm{fin}}(n))$
is relatively closed and nowhere dense in $(\interior\Mn\cap\Xn)\setminus\R$,
by \Cref{cor:density-trap-captured}.
For $n\ge3$, it contains isolated quadratic parameters with logarithmic capture-depth growth nearby
(\Cref{prop:quadratic-noncapture}).
Can it also contain non-isolated points? Which other arithmetic families exhibit comparable depth laws?

Can finite entry describe the locus beyond the canonical lens for $2\le n\le19$?
In the reciprocal parameter $\lambda=1/c$, Hung~\cite[Section~3.3]{Hung2007PolygonFractals}
extended the binary covering region using a hexagon and investigated coverings combining
first- and second-level images. Such constructions suggest a way to seek additional capture
levels.

Is there a quantitative local relation between parameter capture layers and dynamical entry layers?
For binary systems, Solomyak proved asymptotic similarity between the locus and the difference
attractor at certain boundary parameters~\cite{Solomyak2005Asymptotic}. Relating the derivatives of the polynomials $c\mapsto g_u(2c)$ to the dynamical
contractions would be a first step toward explaining the paired parameter and dynamical plots.

Can analogous certified interior layers be constructed for other connectedness loci, such as
the Fibonacci--Mandelbrot family studied by Sirvent and
Thuswaldner~\cite{SirventThuswaldner2015Fibonacci}, regular polygonal iterated function
systems~\cite{BandtHung2008Ngons}, or planar similarities with independent multipliers?

\FloatBarrier
\appendix
\numberwithin{figure}{section}
\numberwithin{table}{section}
\renewcommand{\theHfigure}{\thesection.\arabic{figure}}
\renewcommand{\theHtable}{\thesection.\arabic{table}}

\section{Preliminary geometric proofs}\label{app:geometry}

\subsection{Affine normalization}\label{supp:foundations}
\begin{proof}[Proof of \Cref{prop:affine-universality}]
Write $B=\mu+sA_m$, where $\mu=b+(m-1)\omega/2$ and $s=\omega/2\ne0$.
The affine homeomorphism $H_{B,c}(z)=sz+\mu c/(c-1)$ satisfies
\[
 H_{B,c}\left(a+\frac zc\right)=\mu+sa+\frac{H_{B,c}(z)}c,
 \qquad a\in A_m.
\]
Thus it conjugates the two IFSs, so uniqueness of their compact attractors gives
$K_B(c)=H_{B,c}(E(c,m))$. Taking differences cancels the translation, and
$A_m-A_m=A_{2m-1}$ gives $K_B(c)-K_B(c)=sE(c,2m-1)$.
Connectedness is preserved at each parameter.
\end{proof}

For the ordinary alphabet $\{0,\ldots,n-1\}$, put
\[
 K(c,n)=\left\{\sum_{j=0}^{\infty}b_jc^{-j}:b_j\in\{0,\ldots,n-1\}\right\}.
\]
Taking $b=0$, $\omega=1$ and $m=n$ gives
\[
 K(c,n)=\tfrac12E(c,n)+\frac{(n-1)c}{2(c-1)},\qquad
 K(c,n)-K(c,n)=\tfrac12E(c,2n-1).
\]
Hence $c\in K(c,n)-K(c,n)$ is exactly the marked-point condition $2c\in E(c,2n-1)$.
The series defining $K(c,n)$ is a complex-base digit representation; for real-base expansions, see
R\'enyi and Parry~\cite{Renyi1957,Parry1960}. Self-affine tile theory imposes additional
lattice and digit conditions~\cite{LagariasWang1996Digits}; connectedness and topology in
that setting are studied in~\cite{KiratLau2000,Akiyama2021CollinearTiles}.

\subsection{Horizontal chords}\label{app:horizontal-chords}
The neighboring-overlap argument in \Cref{prop:Mn_characterization} uses the additive
part of Hopf's horizontal-chord theorem~\cite{Hopf1936Chords}.

\begin{lemma}[Horizontal chords of a plane continuum]\label{lem:horizontal-chords}
Let $K\subset\C$ be nonempty, compact, and connected, and let $a,b>0$. If
\[
 K\cap\bigl(K+(a+b)\bigr)\ne\varnothing,
\]
then $K\cap(K+a)\ne\varnothing$ or $K\cap(K+b)\ne\varnothing$. Equivalently, the missing positive
horizontal chord lengths of $K$ form an additive set.
\end{lemma}
\begin{proof}
Let $x_-,x_+$ and $y_-,y_+$ be the coordinate extrema of $K$. If $y_-=y_+$, then $K$ is a
horizontal interval and the assertion is immediate. Otherwise, in the rectangle
\[
 R=[x_-,x_++a+b]\times[y_-,y_+],
\]
$C=K\cup(K+(a+b))$ is a continuum meeting both vertical sides, while $K+a$ meets both
horizontal sides. The rectangle-crossing lemma gives $(K+a)\cap C\ne\varnothing$
\cite[Lemma~7]{PascolettiZanolin2013}. The two possible intersections give the conclusions,
translating by $-a$ in the second case.
\end{proof}

\subsection{The connected annulus}\label{supp:annulus-proof}
The following area argument proves the inclusion $\{1<|c|<\sqrt n\}\subset\operatorname{int}\Mn$
from \cite[Proposition~5(i)]{EJSn24}. The same area
obstruction underlies the overlap bound in~\cite[Lemma~2.5(i)]{Solomyak1998Measure}.
\begin{proof}[Proof of \Cref{prop:inner-annulus}]
Suppose $1<\rho=|c|<\sqrt n$ and $E=E(c,n)$ is disconnected. By
\Cref{prop:Mn_characterization}, its first-level pieces are pairwise disjoint. For sufficiently
small $\varepsilon>0$, the images of $E_\varepsilon=\{z:\operatorname{dist}(z,E)<\varepsilon\}$
under $f_t(z)=t+z/c$, $t\in A_n$, remain pairwise disjoint and lie in $E_\varepsilon$. Therefore
\[
 \operatorname{area}(E_\varepsilon)
 \ge\sum_{t\in A_n}\operatorname{area}\bigl(f_t(E_\varepsilon)\bigr)
 =\frac n{\rho^2}\operatorname{area}(E_\varepsilon),
\]
contradicting its finite positive area. The open annulus thus lies in $\operatorname{int}\Mn$;
by relative closedness, $\Mn$ also contains the outer circle $|c|=\sqrt n$.
\end{proof}
In the variable $\lambda=1/c$, \Cref{prop:inner-annulus} gives the annulus $n^{-1/2}\le|\lambda|<1$ of
zeros of power series $1+\sum_{j\ge1}d_j\lambda^j$ with $d_j\in D_n$. For these series,
Kewarth~\cite{Kewarth2026Rigidity} proves the smaller annulus $(n-1)^{-1/2}\le|\lambda|<1$ and suggests
that the best possible annulus is $n^{-1/2}\le|\lambda|<1$.

\subsection{An enclosure proof of the radial bound}\label{supp:radial-bound}
Vertical support gives the known bound \eqref{eq:known-radial-bound} and the estimate used
for the twenty-map case.
\begin{proposition}\label{prop:Mn-non-real-disk}
For every $n\ge2$,
\[
 \Mn\setminus\R\subset\{c\in\C:|c|<1+\sqrt{n-1}\}.
\]
\end{proposition}
\begin{proof}
Let $c=\rho e^{i\theta}\notin\R$, with $\rho>1$, and put $y=\operatorname{Im}c$ and $N=2n-1$.
By \eqref{eq:VE-series-Y} and $|\sin(k\theta)|\le k|\sin\theta|$,
\[
 \VE(c,N)<(N-1)|\sin\theta|\sum_{k=1}^{\infty}k\rho^{-k}
 =\frac{2(n-1)|y|}{(\rho-1)^2}.
\]
The inequality is strict because $|\sin(2\theta)|<2|\sin\theta|$ for non-real $c$.
If $\rho\ge1+\sqrt{n-1}$, this gives $\VE(c,N)<2|y|=|\lv(2c)|$.
Thus $2c\notin\PEnc(c,N)\supset E(c,N)$ by \Cref{prop:enc-exact}, and the marked-point
criterion excludes $c$ from $\Mn$.
\end{proof}

This recovers \cite[Proposition~4]{EJSn24} and, with the lens definition, the following part of
\cite[Theorem~2]{EJSn24}.
\begin{corollary}\label{cor:Mn-non-real-lens}
For every $n\ge21$, $\Mn\setminus\R\subset\Xn$.
\end{corollary}
\begin{proof}
For $n\ge21$, we have $1+\sqrt{n-1}<\sqrt{2n}-1$, equivalently $n>11+4\sqrt6$.
Thus every $c\in\Mn\setminus\R$ satisfies
$|c\pm1|\le|c|+1<\sqrt{2n}$, the two lens inequalities.
\end{proof}

\section{The sharp global threshold}\label{app:twenty}\label{app:threshold}
The radial bound gives lens containment for $n\ge21$. We treat $n=20$ and certify off-lens roots
for $2\le n\le19$.
\subsection{The twenty-map case}
Outside the lens, the radial bound confines a possible non-real
counterexample to a narrow crescent. The vertical enclosure test leaves
only the first digits $20$ and $22$; the slanted test excludes both.
Figure~\ref{fig:n20-q1-geometry} locates the crescent and two nearby
boundary windows.

\begin{proposition}\label{prop:M20-lens}
Every $c\in\mathcal M_{20}\setminus\R$ satisfies $|c|^2+2|\operatorname{Re}c|<39$; that is,
$\mathcal M_{20}\setminus\R\subset\mathcal X_{20}$.
\end{proposition}
\begin{proof}
Put $N=39$ and suppose $c=x+iy\in\mathcal M_{20}\setminus\mathcal X_{20}$ is non-real.
By \Cref{rem:basic-symmetries}, assume $x\ge0$, $y>0$, and write $\rho=|c|$.
Then $\rho^2+2x\ge39$ with $x\le\rho$ gives $\rho\ge\sqrt{40}-1$, and the radial bound gives
$\rho<1+\sqrt{19}$. In particular $\rho>5$, so the estimate used in
\Cref{prop:Mn-non-real-disk} gives
\[
 H=\frac{\VE(c,N)}y<\frac{38}{(\rho-1)^2}<\frac{19}{8}.
\]
Some $t\in A_N$ sends $2c$ into $E(c,N)\subset\PEnc(c,N)$.
By \Cref{lem:dynamics}, the vertical and slanted coordinates of $g_t(2c)$ are $y(4x-t)$ and
$-(2y/\rho)(\rho^2+x(t-4x))$, respectively. Using \eqref{eq:SE-from-VE-def}, the enclosure tests become
\begin{equation}\label{eq:n20-clean-vertical}
 |4x-t|\le H,\qquad
 \bigl|\rho^2+x(t-4x)\bigr|\le\frac{38+H}{2}<\frac{323}{16}.
\end{equation}
Now $2x\ge39-\rho^2>19-2\sqrt{19}>41/4$ and $x\le\rho<43/8$;
hence $41/2<4x<43/2$. The first test forces the even digit $t$ to be $20$ or $22$.
But then
\[
 \rho^2+x(t-4x)\ge39+18x-4x^2>\frac{323}{16},
\]
since the quadratic decreases for $x>9/4$ and equals $323/16$ at $x=43/8$.
This contradicts the second test.
\end{proof}

\begin{figure}[p]
\centering
\includegraphics[width=\linewidth]{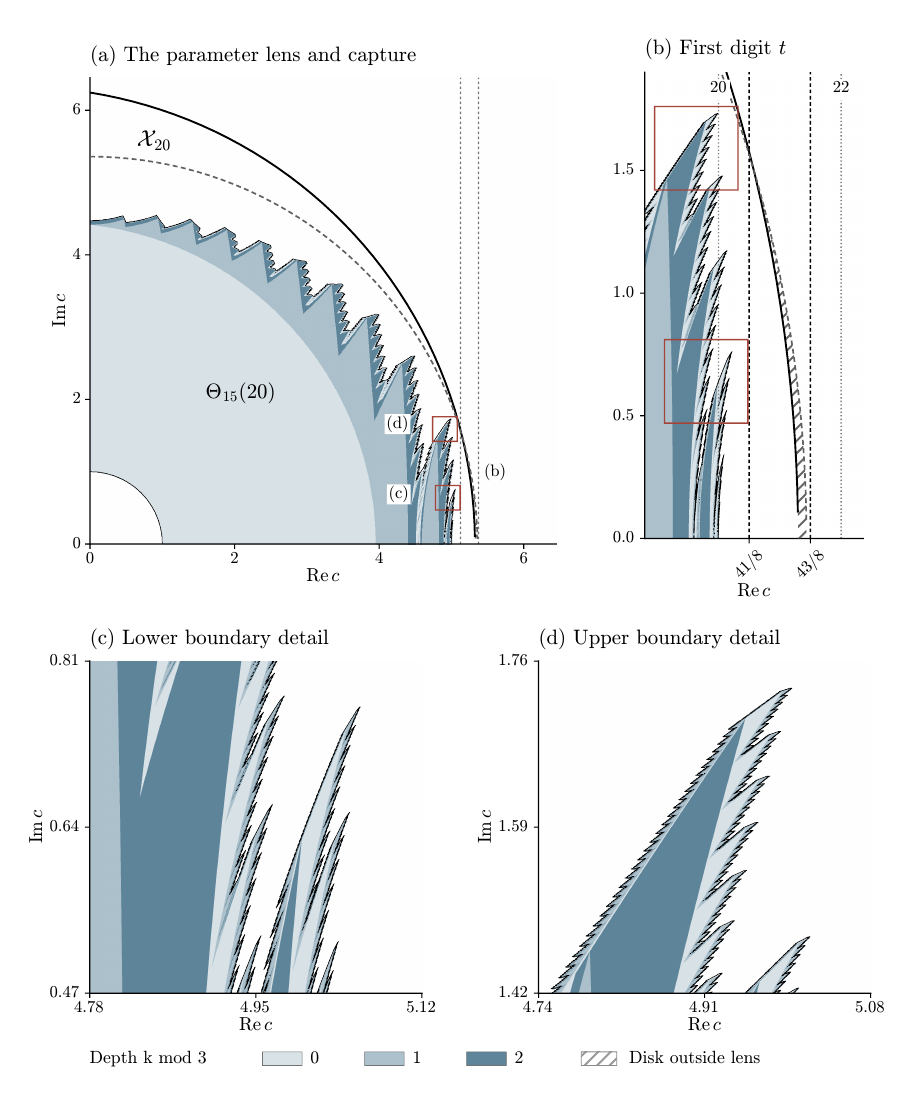}
\caption{The twenty-map case, $N=39$.
(a) The first-quadrant lens and numerical capture levels; the solid and
 dashed curves are $|c+1|=\sqrt{40}$ and $|c|=1+\sqrt{19}$.
(b) The hatched crescent is the part of the radial disk outside the lens.
Its points satisfy $41/8<x<43/8$. The vertical enclosure bound leaves
only $t=20,22$, and the slanted bound excludes both
(Proposition~\ref{prop:M20-lens}); the digit labels mark $x=t/4$.
(c,d) Enlargements of the marked boundary windows.
The three blue tones record first detected entry depth modulo three,
through depth $15$ in (a,c) and $16$ in (b,d).}
\label{fig:n20-q1-geometry}
\end{figure}

\subsection{Certified roots outside the lens}\label{supp:off-lens-certificates}
For each $2\le n\le19$, \Cref{tab:off-lens-cert} gives a monic $P_n\in\Z[c]$ with nonleading
coefficients in $\Dn$, a Gaussian rational center $\tilde c_n$, and a radius $r_n>0$.
For $P(z)=\sum_{k=0}^d a_kz^k$, put
\[
 \mathcal R_P(z_0,r)=\sum_{k=2}^d |a_k|
 \bigl((|z_0|+r)^k-|z_0|^k-k|z_0|^{k-1}r\bigr).
\]
This bounds the terms beyond first order in $P(z_0+h)$ for $|h|\le r$.
The table supplies lower bounds for the Rouch\'e and off-lens margins
\[
 \Delta_n=|P_n'(\tilde c_n)|\,r_n-|P_n(\tilde c_n)|-\mathcal R_{P_n}(\tilde c_n,r_n),\qquad
 \Lambda_n=(|\tilde c_n+1|-r_n)^2-2n.
\]
Polynomial values at the centers are exact Gaussian rationals;
rational bounds for their moduli give the displayed margins.

\begin{lemma}[Rouch\'e disk test]\label{lem:rouche-disk-test}
Let $P\in\C[z]$, $z_0\in\C$ and $r>0$. If $|P(z_0)|+\mathcal R_P(z_0,r)<|P'(z_0)|\,r$, then $P$ has
exactly one zero in $\overline{B(z_0,r)}$, counted with multiplicity.
\end{lemma}
\begin{proof}
Set $L(z)=P(z_0)+P'(z_0)(z-z_0)$. On $|z-z_0|=r$,
\[
 |P-L|\le\mathcal R_P(z_0,r)<|P'(z_0)|\,r-|P(z_0)|\le|L|.
\]
Rouch\'e's theorem gives the same zero count for $P$ and $L$ inside the circle, with neither
vanishing on it. The unique zero $z_0-P(z_0)/P'(z_0)$ of $L$ lies inside by the same inequality.
\end{proof}

\begin{table}[!htbp]
	\centering
	\scriptsize
	\setlength{\tabcolsep}{3pt}
	\renewcommand{\arraystretch}{1.12}
	\begin{tabular}{r l l r r r}
		\toprule
		$n$                                & $P_n(c)$        & $\tilde c_n$ & $r_n$                & $\Delta_n>$ & $\Lambda_n>$ \\
		\midrule
		2                                  & $c^{3}-c^{2}+1$                         &
			$0.877438833123+0.744861766620\,i$ & $10^{-3}$                               & $0.0026684$ & $0.075557$                                      \\
		3                                  & $c^{3}-2c^{2}+2$                        &
			$1.419643377607+0.606290729207\,i$ & $10^{-3}$                               & $0.0028294$ & $0.217274$                                      \\
		4                                  & $c^{3}-3c^{2}+c+3$                      &
			$1.884646177119+0.589742805022\,i$ & $10^{-3}$                               & $0.0031977$ & $0.663092$                                      \\
		5                                  & $c^{3}-4c^{2}+3c+3$                     &
			$2.273409138442+0.563821092829\,i$ & $10^{-3}$                               & $0.0032321$ & $1.02645$                                       \\
		6                                  & $c^{3}-5c^{2}+5c+4$                     &
			$2.755773570847+0.474476778007\,i$ & $10^{-3}$                               & $0.0031196$ & $2.32339$                                       \\
		7                                  & $c^{3}-5c^{2}+5c+4$                     &
			$2.755773570847+0.474476778007\,i$ & $10^{-3}$                               & $0.0031196$ & $0.323393$                                      \\
		8                                  & $c^{3}-6c^{2}+7c+7$                     &
			$3.313682542356+0.421052806988\,i$ & $10^{-3}$                               & $0.0033216$ & $2.77647$                                       \\
		9                                  & $c^{3}-6c^{2}+7c+7$                     &
			$3.313682542356+0.421052806988\,i$ & $10^{-3}$                               & $0.0033216$ & $0.776475$                                      \\
		10                                 & $c^{3}-6c^{2}+8c+9$                     &
			$3.353263977249+1.222280727312\,i$ & $10^{-3}$                               & $0.0103477$  & $0.435835$                                      \\
		11                                 & $c^{3}-7c^{2}+10c+9$                    &
			$3.806735133791+0.423562585034\,i$ & $10^{-3}$                               & $0.0037431$ & $1.27445$                                       \\
		12                                 & $c^{3}-7c^{2}+10c+11$                   &
			$3.855311981586+0.784808055439\,i$ & $10^{-3}$                               & $0.0072530$ & $0.180142$                                      \\
		13                                 & $c^{3}-7c^{2}+12c+12$                   &
			$3.846271905939+1.591733658655\,i$ & $10^{-3}$                               & $0.0152924$  & $0.00976647$                                    \\
		14                                 & $c^{3}-8c^{2}+13c+13$                   &
			$4.342889763045+0.309577572123\,i$ & $10^{-3}$                               & $0.0030983$ & $0.631606$                                      \\
		15                                 & $c^{4}-8c^{3}+14\sum_{j=0}^{2}c^{j}$    &
			$4.459402595334+1.287852995168\,i$ & $10^{-3}$                               & $0.0673229$   & $1.45242$                                       \\
		16                                 & $c^{4}-8c^{3}+15\sum_{j=0}^{2}c^{j}$    &
			$4.462920327187+1.637845297937\,i$ & $10^{-3}$                               & $0.0891868$  & $0.514630$                                       \\
		17                                 & $c^{8}-8c^{7}+16\sum_{j=0}^{6}c^{j}$    &
			$4.499982156746+1.936550004261\,i$ & $10^{-6}$                               & $0.0655749$   & $1.79\times 10^{-5}$                            \\
		18                                 & $c^{4}-9c^{3}+17\sum_{j=0}^{2}c^{j}$    &
			$4.962660230936+0.943348614313\,i$ & $10^{-3}$                               & $0.0577094$  & $0.431150$                                      \\
		19                                 & $c^{12}-9c^{11}+18\sum_{j=0}^{10}c^{j}$ &
			$4.999999990671+1.414213647216\,i$ & $10^{-9}$                               & $0.0496746$  & $1.15\times 10^{-7}$                            \\
		\bottomrule
	\end{tabular}
	\caption{Certified data for off-lens witnesses $c_n\in\Mn\setminus\Xn$, $2\le n\le19$.
		The decimal coordinates and radii are exact terminating rationals.
		The last two columns give verified lower bounds for $\Delta_n$ and
		$\Lambda_n$, rounded downward.
		The inequalities $\Delta_n>0$, $\Lambda_n>0$, and $\Imag(\tilde c_n)>r_n$ imply, respectively,
		that the disk $B_n=\overline{B(\tilde c_n,r_n)}$ contains a unique zero, lies outside the lens,
        and avoids the real axis.
		Since $\mathcal X_n\subset\mathcal X_{n+1}$ and $D_n\subset D_{n+1}$, a witness for $n+1$ with
		coefficients in $D_n$ also serves $n$; the rows $n=6,7$ and $n=8,9$ share witnesses in this way.}
	\label{tab:off-lens-cert}
\end{table}

\FloatBarrier

\begin{proposition}\label{prop:off-lens-nlt20}
For each $2\le n\le 19$ there exists $c_n\in\Mn\setminus\Xn$ with $c_n\notin\R$.
\end{proposition}
\begin{proof}
Use the data in \Cref{tab:off-lens-cert} and set $B_n=\overline{B(\tilde c_n,r_n)}$.
By \Cref{lem:rouche-disk-test} and $\Delta_n>0$, it contains one zero $c_n$ of $P_n$;
$\operatorname{Im}\tilde c_n>r_n$ makes it non-real.
Since $|\tilde c_n+1|>r_n$, the margin $\Lambda_n>0$ gives
$|c_n+1|\ge|\tilde c_n+1|-r_n>\sqrt{2n}$, so $c_n\notin\Xn$ and
$|c_n|>\sqrt{2n}-1\ge1$.
The coefficients of $P_n$ then give $c_n\in\Rn\setminus\Dbar\subset\Mn$ by
\Cref{prop:Mn-closure-Rn}.
\end{proof}

\section{Exact certificates for the worked examples}\label{app:exact-certificates}
The four parts of this appendix supply, in order, the finite data behind
\Cref{sec:worked-certificates}, \Cref{ex:n21-word-chart-strata},
\Cref{prop:branch-sharpness}, and the depth-three example of
\Cref{sec:finite-level-geometry}; Appendix~\ref{supp:branch-witness} also adapts
the branching bound to verified widths. Except for the interval centers in
\Cref{tab:branch-witness}, all decimals are exact, and every comparison is an
exact rational inequality.

Except in Appendix~\ref{supp:worked-certificates}, we use the rescaled coordinates of
\eqref{eq:rectangle-inverse}. For $c=x+iy$ with $y>0$ and $q=|c|^2$, put
\begin{equation}\label{eq:exact-coordinates}
 \alpha(z)=\frac{\Imag(cz)}y,\qquad \beta(z)=\frac{\Imag z}y;\qquad
 \beta'=\alpha-t,\qquad \alpha'=2x\beta'-q\beta,
\end{equation}
where $(\alpha',\beta')$ are the coordinates of $g_t(z)$. The marked point $2c$
has $(\alpha,\beta)=(4x,2)$, and the trap $\PTrap(c,N)$ is given by $|\alpha|<N$
and $q|\beta|<N-2|x|$.

\subsection{Interior and exterior certificates for \texorpdfstring{$n=3$}{n=3}}\label{supp:worked-certificates}
Here $N=5$, the digits are $A_5=\{-4,-2,0,2,4\}$, and, as in
\Cref{sec:worked-certificates}, we iterate $z_{j+1}=c(z_j-t_j)$ from $z_0=2c$.
The two parameters are $c_0=0.585+1.675i$ and $c_1=0.585+1.705i$.

\emph{Enclosures.} Fifty terms of the series \eqref{eq:VE-series-Y} and the
tail bound $4r^{-50}/(r-1)$, with $r=177/100<|c_0|$ and $r=9/5<|c_1|$, give
$\VE(c_0,5)<3.95$ and $\VE(c_1,5)<3.813$; here $|c_0|^2=62957/20000$ and
$|c_1|^2=3.24925$. Since $\rho\SE(c,5)=4y+\VE(c,5)$ by
\eqref{eq:SE-from-VE-def}, we obtain the verified enclosures
\begin{equation}\label{eq:c0-enclosures}
 \begin{aligned}
 E(c_0,5)&\subset\{z:|\Imag z|\le3.95,\ |\Imag(c_0z)|\le10.65\},\\
 E(c_1,5)&\subset\{z:|\Imag z|\le3.813,\ |\Imag(c_1z)|\le10.633\}.
 \end{aligned}
\end{equation}

\emph{Interior certificate.} The word $(2,-4,-2,4)$ moves $2c_0$ as follows:
\begin{center}
\begin{tabular}{@{}crr@{}}
\toprule
$j$ & $\operatorname{Re}z_j$ & $\operatorname{Im}z_j$\\
\midrule
0 & $1.17$ & $3.35$\\
1 & $-6.0968$ & $0.5695$\\
2 & $-2.1805405$ & $-3.1789825$\\
3 & $5.219179495$ & $-2.1621101$\\
4 & $4.334754422075$ & $0.777291245625$\\
\bottomrule
\end{tabular}
\end{center}
Write $c_0=x+iy$ and $q=|c_0|^2$. Both trap inequalities hold at $z_4$, and
exact comparison bounds the vertical and slanted margins:
\[
 \begin{aligned}
 \Delta_v&=(5-2x)y-q|\Imag z_4|=\frac{126990520078699}{32000000000000}
 >q^3\,\frac{3}{80},\\
 \Delta_s&=5y-|\Imag(c_0z_4)|=\frac{527656771467}{800000000000}
 >\frac{71}{40}\,q^2\,\frac{3}{80}.
 \end{aligned}
\]
As $|c_0|<71/40$, the radius \eqref{eq:certified-ball-radius} of this word of
length four, $q^{-2}\min\{\Delta_s/|c_0|,\Delta_v/q\}$, exceeds $3/80$. This
proves \eqref{eq:c0-certified-ball}.

\emph{Minimum depth.} Every prefix of a word capturing $c_0$ sends $2c_0$ into
$E(c_0,5)$ (proof of \Cref{prop:algorithm-correctness}), and hence into the
first parallelogram in \eqref{eq:c0-enclosures}. Testing all five digits at
each point that stays there leaves, through depth three, only the words $(0)$, $(2)$,
$(2,-4)$ and $(2,-4,-2)$. None of these, nor the empty word, captures $c_0$.
The trap has vertical half-width $(5-2x)y/q<2.04$, which excludes $2c_0$ and
the endpoints of $(0)$, $(2,-4)$ and $(2,-4,-2)$; the endpoint of $(2)$ fails
the slanted inequality, since $|\Imag(c_0z_1)|=9.8789825>5y=8.375$. Hence
$\kappa_3(c_0)=4$.

\emph{Exterior certificate.} At $c_1$ exactly one branch survives at each of
the depths one to three, and none at depth four. The table lists the surviving
parent word and the digits that pass the vertical test and both tests of
\eqref{eq:c0-enclosures}:
\begin{center}
\begin{tabular}{@{}clcc@{}}
\toprule
Depth & Parent word & Vertical test & Both tests\\
\midrule
1 & $\emptyword$ & $\{2,4\}$ & $\{2\}$\\
2 & $(2)$ & $\{-4\}$ & $\{-4\}$\\
3 & $(2,-4)$ & $\{-4,-2\}$ & $\{-2\}$\\
4 & $(2,-4,-2)$ & $\{4\}$ & $\varnothing$\\
\bottomrule
\end{tabular}
\end{center}
The last candidate, again along $(2,-4,-2,4)$, has
$|\Imag(c_1z_4)|=10.64642170337075>10.633$. Since the second parallelogram in
\eqref{eq:c0-enclosures} contains $E(c_1,5)$, this is an exterior verdict of \Cref{alg:geom}, and
\Cref{prop:algorithm-correctness} gives $c_1\notin\mathcal M_3$.

\subsection{Exact capture depths for two word-region examples}\label{supp:chart-verification}
Here $n=21$, $N=41$, $u=(10,-40,-40)$ and $P_u(z)=z^3-5z^2+20z+20$, as in
\Cref{ex:n21-word-chart-strata}. Both parameters have $x>0$; write
$b=(41-2x)/q$, so that the trap is $|\alpha|<41$, $|\beta|<b$.

\emph{The root $c_*$.} Since $P_u'(r)=3(r-5/3)^2+35/3>0$, the polynomial $P_u$
has a unique real root $\xi$, and $P_u(-5/6)=-155/216<0<36/125=P_u(-4/5)$ gives
$-5/6<\xi<-4/5$. By Vieta's formulas, $c_*=x+iy$ has $x=(5-\xi)/2$ and
$q=-20/\xi$, so
\[
 \frac{29}{10}<x<\frac{35}{12},\qquad 24<q<25.
\]
Hence $q+2x<185/6<41$, so $c_*$ lies in the lens, and $b<22/15<8/5$. Since
$P_u(c_*)=0$, \Cref{prop:dyn-poly} gives $g_u(2c_*)=0$ and
$g_{(10,-40)}(2c_*)=2c_*(c_*^2-5c_*+20)=-40$. Both points lie in the trap, so
$\kappa_{21}(c_*)\le2$. No shorter word captures $c_*$. The marked point has
$\beta=2>b$. A first digit $t\le10$ gives $\beta=4x-t>8/5>b$. For $t\ge12$,
note that $\beta(g_{-40}(w))=\alpha(w)+40$ and $\beta(-40)=0$ give
$\alpha(g_{10}(2c_*))=-40$; since $\alpha(g_t(2c))=2x(4x-t)-2q$, it follows that
\[
 \alpha(g_t(2c_*))=-40-2x(t-10)<-41.
\]
Hence $\kappa_{21}(c_*)=2$.

\emph{The parameter $c_\dagger=2.85+4.057i$.} Here
\[
 q=24.581749,\qquad q+2x=30.281749<41,\qquad
 b=\frac{35300000}{24581749}\in(1.436,1.437),
\]
and the coordinates along $u$ are terminating decimals:
\begin{center}
\begin{tabular}{@{}crr@{}}
\toprule
$j$ & $\alpha(z_j)$ & $\beta(z_j)$\\
\midrule
0 & $11.4$ & $2$\\
1 & $-41.183498$ & $1.4$\\
2 & $-41.1603872$ & $-1.183498$\\
3 & $22.478243738002$ & $-1.1603872$\\
\bottomrule
\end{tabular}
\end{center}
The third iterate $z_3$ lies in the trap. No shorter word captures $c_\dagger$. The
marked point has $\beta=2>b$. A first digit satisfies the vertical trap inequality only if
$|4x-t|<b$, that is, $t\in\{10,12\}$, and the resulting values $\alpha=-41.183498$
and $\alpha=-52.583498$ violate the slanted one. A two-digit word capturing
$c_\dagger$ would satisfy
\[
 q|4x-t|=q|\beta_1|\le2x|\beta_2|+|\alpha_2|<2xb+41<2.001\,q,
\]
so again $t\in\{10,12\}$. After $t=12$, every second digit gives
$\beta_2\le-12.583498<-b$. After $t=10$, only the second digit $-40$ gives
$|\beta_2|<b$ (the value $-42$ is not a digit of $A_{41}$), and its slanted
coordinate is $-41.1603872<-41$. Hence $\kappa_{21}(c_\dagger)=3$.

\subsection{The four-child vertex and verified widths}\label{supp:branch-witness}
\Cref{tab:branch-witness} lists the iterates used in the proof of
\Cref{prop:branch-sharpness}: the prefixes of the word
\eqref{eq:branch-sharp-word}, starting from $2c$, and the four children of its
endpoint.
\begin{table}[!htbp]
\centering
\caption{A four-child vertex at $n=24$, $c=49/10+i/1000$. Each entry denotes the closed interval of radius $10^{-4}$ about the displayed decimal. Prefix rows follow \eqref{eq:branch-sharp-word}; the final four rows give its children. The coordinates are defined in \eqref{eq:normalized-branching}.}
\label{tab:branch-witness}
\begin{tabular}{@{}lrr@{}}
\toprule
Prefix length or child digit & $\beta$ & $\alpha$\\
\midrule
0 & $2.0000$ & $19.6000$\\
1 & $1.6000$ & $-32.3400$\\
2 & $-0.3400$ & $-41.7480$\\
3 & $0.2520$ & $10.6328$\\
4 & $-1.3672$ & $-19.4482$\\
5 & $-1.4482$ & $18.6334$\\
6 & $0.6334$ & $40.9780$\\
\midrule
38 & $2.9780$ & $13.9761$\\
40 & $0.9780$ & $-5.6239$\\
42 & $-1.0220$ & $-25.2239$\\
44 & $-3.0220$ & $-44.8239$\\
\bottomrule
\end{tabular}
\end{table}

The digit-spacing argument of \Cref{prop:vertical-pruning} remains valid when
the enclosure widths are replaced by verified upper bounds.
\begin{corollary}[Branching with verified upper bounds]\label{cor:working-branching}
Let $n\ge2$, $N=2n-1$, and $c=x+iy$ with $\rho=|c|>1$ and $y\ne0$.
Let $\widehat{\VE}\ge\VE(c,N)$ and $\widehat{\SE}\ge\SE(c,N)$ define a verified
outer enclosure. For a positive integer $m$, if
\begin{equation}\label{eq:working-branching-margin}
 \min\left\{\frac{\widehat{\VE}}{|y|},
       \frac{\rho\widehat{\SE}}{2|x||y|}\right\}<m,
\end{equation}
then at most $m$ digits send any given parent point into that enclosure.
The second term is $+\infty$ if $x=0$.
\end{corollary}
\begin{proof}
For $x\ne0$, the digit intervals have lengths $2\widehat{\VE}/|y|$ and
$\rho\widehat{\SE}/(|x||y|)$, so their intersection has length less than $2m$.
If $x=0$, the first interval alone has length less than $2m$. In either case,
the alphabet spacing of two allows at most $m$ digits.
\end{proof}

For non-real $c$ with $|c|\ge\sqrt n$, the proof of
\Cref{thm:uniform-branching-cap} shows that, with the exact widths, the minimum
in \eqref{eq:working-branching-margin} is less than $4$ for $n\ge3$ and less
than $3$ for $n\ge25$; for $n=2$ there are only three digits.
As these inequalities are strict at each fixed parameter, sufficiently accurate
verified widths retain the bounds. The required accuracy depends on $c$ and is not uniform as $y\to0$,
and unresolved comparisons may retain extra candidates.
\Cref{thm:uniform-branching-cap} uses no algebraic property of $c$, in contrast
with finite-type results under complex Pisot hypotheses~\cite{Bandt2023FiniteType}.
A bound on the number of children alone, however, gives neither a finite address
automaton nor the neighborhood graphs used to study boundaries and
interiors~\cite{Bandt2025Interior}.

\subsection{An enclosed component appearing at depth three}\label{supp:low-depth-topology}
We prove the statements of \Cref{sec:finite-level-geometry} about
$p=103/50+(61/20)i=2.06+3.05i$ for $n=13$, so that $N=25$. Let $S$ be the
closed square with center $p$ and half-side $10^{-5}$. The certificates use
exact rational interval arithmetic over segments and squares of parameters, for
the lens inequalities and for the coordinates $\alpha,\beta$ of
\eqref{eq:exact-coordinates} along fixed words.

The convex quadrilateral $Q$ with counterclockwise vertices
\[
 2.040+3.069i,\quad 2.051+3.044i,\quad
 2.084+3.030i,\quad 2.078+3.065i
\]
contains $S$. Its boundary splits into $191$ consecutive segments; on each,
the lens inequalities and both trap inequalities hold after one of the words
\[
 (6,-20,-24),\quad(6,-20,-22),\quad(8,-24,-10),\quad(8,-24,-12).
\]
Hence $\partial Q\subset\Theta_3(13)$; \Cref{fig:M13context}(e) shows $Q$.
Throughout $S$, each of the $1+25+25^2+25^3=16\,276$ words of length at most
three violates a trap inequality, whereas $(8,-24,-10,22)$ satisfies both; hence
$S\subset\Theta_4(13)\setminus\Theta_3(13)$. By the Jordan curve theorem, the
component of $\C\setminus\Theta_3(13)$ containing $p$ lies inside $Q$, and it
contains $S$.

Nine interval segments cover the polygonal path
$p\to2.08+3.05i\to2.08+5i$, and on each of them all $651$ words of length at
most two fail a lens or trap inequality. The ray $\{2.08+iy:y\ge5\}$ lies
outside the lens, since $q+2x\ge20929/625>25$ on it. Hence the component of
$\C\setminus\Theta_2(13)$ containing $p$ is unbounded.

Finally, $\partial\Theta_3(13)$ misses $\partial Q\subset\Theta_3(13)$ but has
points inside $Q$, on the boundary of the bounded component above, and outside
$Q$, on the rightward ray from $1+i\in\Theta_0(13)$. Hence
$\partial\Theta_3(13)$ is disconnected.

A simpler disconnection occurs at depth one: the part of $\Theta_1(13)$ in the
upper half-plane is disconnected. On the line $x=17/5$, the one-letter
conditions of \Cref{sec:finite-level-geometry} fail. For $t\ne14$, we have
$q|4x-t|\ge2312/125>91/5=N-2x$, and for $t=14$, we have
$|3x^2-y^2-tx|=323/25+y^2>25/2$. So the line misses
$\Theta_1(13)=\bigcup_{|u|=1}\mathcal U_u$, yet $1+i\in\Theta_0(13)$ and
$7/2+i/4\in\mathcal U_{14}$ lie on opposite sides of it.

\section{Interior noncapture at quadratic parameters}\label{app:quadratic-depth}
We prove \Cref{prop:quadratic-noncapture}: no level captures the parameter
$\gamma_n=(1+i\sqrt{4n-1})/2$, whereas every nearby parameter $c\ne\gamma_n$ is first captured at depth
$\log(1/|c-\gamma_n|)/\log\sqrt n+O(1)$. We also prove the root-free radius
\eqref{eq:explicit-degree-radius}. Following the outline given after the theorem,
\Cref{lem:quadratic-lattice} yields noncapture at $\gamma_n$ and the lower bound, and
\Cref{lem:quadratic-monic} supplies the finitely many words that complete the upper bound.

Fix $n\ge3$ and write $\gamma=\gamma_n$, $\rho=|\gamma|=\sqrt n$ and
$P(c)=c^2-c+n$, the minimal polynomial of $\gamma$. Since
$|\gamma|^2+2|\Real\gamma|=n+1<2n-1$, the parameter $\gamma$ lies in the
non-real lens. We use the normalization from the proof of
\Cref{prop:dyn-poly}: $w=z/(2c)$ and digits $d=t/2\in D_n$. The inverse
branches become $T_{d,c}(w)=cw-d$, the marked point becomes $w=1$, and word
lengths are unchanged. For a word $u$, let $w_j(c)$ be the image of $1$ under
the first $j$ digits of $u$, so that $w_{|u|}(c)=T_{u,c}(1)=P_u(c)$. The trap
becomes $Q(c)=\PTrap(c,2n-1)/(2c)$; for $c=x+iy$ with $y>0$,
\begin{equation}\label{eq:quadratic-normalized-trap}
 Q(c)=\left\{w:
 \frac{|\Imag(c^2w)|}{y}<n-\tfrac12,\quad
 \frac{|\Imag(cw)|}{y}<\frac{n-\tfrac12-|x|}{|c|^2}\right\}.
\end{equation}
Write $-u$ for the word with negated digits. Since
$T_{-u,c}(-w)=-T_{u,c}(w)$ and $Q(c)=-Q(c)$, the word $-u$ sends $-w$ into
$Q(c)$ whenever $u$ sends $w$ into $Q(c)$.

\begin{lemma}[Lattice separation]\label{lem:quadratic-lattice}
Let $\Lambda=\Z[\gamma]$.
\begin{enumerate}[label=\textup{(\alph*)},leftmargin=*]
\item No nonzero integer polynomial of height at most $n-1$ vanishes at
$\gamma$. In particular, $T_{u,\gamma}(1)\in\Lambda\setminus\{0\}$ for every
word $u$.
\item $\overline{Q(\gamma)}\cap\Lambda=\{0\}$. Hence
$\delta_n:=\dist\bigl(\Lambda\setminus\{0\},\overline{Q(\gamma)}\bigr)>0$.
\item The closed parallelogram
$\Pi=\gamma^{-1}\{A\gamma+B:A,B\in\R,\ |A|\le\tfrac12,\ |B|\le\tfrac n2\}$
lies in $Q(\gamma)$, and its translates by $\Lambda$ cover $\C$.
\end{enumerate}
\end{lemma}

\begin{proof}
(a) Divide the polynomial by a power of $X$ so that its constant term is
nonzero. If it vanished at $\gamma$, it would be divisible in $\Z[X]$ by the
monic polynomial $P$, so its constant term would be a nonzero multiple of $n$;
this is impossible for height at most $n-1$. The second assertion follows
because $P_u$ is monic of height at most $n-1$ and $P_u(\gamma)\in\Lambda$.

(b) Write $\gamma w=A\gamma+B$ with $A,B$ real. As $\gamma^2=\gamma-n$, we have
$\Imag(\gamma w)=A\Imag\gamma$ and $\Imag(\gamma^2w)=(A+B)\Imag\gamma$, so
\eqref{eq:quadratic-normalized-trap} becomes
\begin{equation}\label{eq:quadratic-lattice-coordinates}
 w\in Q(\gamma)\quad\Longleftrightarrow\quad
 |A|<1-\tfrac1n\quad\text{and}\quad|A+B|<n-\tfrac12.
\end{equation}
If $w\in\Lambda$, then $\gamma w\in\gamma\Lambda=\Z\gamma+n\Z$, so $A\in\Z$
and $B\in n\Z$. The same inequalities with non-strict signs then force $A=0$
and $B=0$. The distance is positive because $\Lambda$ is discrete and
$\overline{Q(\gamma)}$ is compact.

(c) By the criterion \eqref{eq:quadratic-lattice-coordinates},
$\Pi\subset Q(\gamma)$, since $|A|\le\frac12<1-\frac1n$ and
$|A+B|\le\frac{n+1}2<n-\frac12$ for $n\ge3$. Moreover, $\gamma\Pi$ is a
fundamental parallelogram of $\gamma\Lambda=\Z\gamma+n\Z$, so $\Pi$ is one of
$\Lambda$.
\end{proof}

\begin{lemma}[Monic representations]\label{lem:quadratic-monic}
For every $L\in\Lambda\setminus\{0\}$ there is $K_L$ such that, for every
$K\ge K_L$, $L=F(\gamma)$ for some monic $F\in\Z[X]$ of degree $K$ and height
at most $n-1$.
\end{lemma}

\begin{proof}
Since $\gamma$ has minimal polynomial $X^2-X+n$ with $n\ge3$, Gilbert's criterion
\cite[Theorem~1]{Gilbert1981} shows that $\gamma$ is the base of a radix representation of
$\Z[\gamma]$ with digits $0,\ldots,n-1$; see also~\cite{KataiKovacs1981}. Thus
$L=\sum_{j\le m}b_j\gamma^j$ with $b_j\in\{0,\ldots,n-1\}$ and $b_m\ge1$. Since $(\gamma^2-1)P(\gamma)=0$, replacing $b_m\gamma^m$ by
\[
 \gamma^m\bigl(\gamma^4-\gamma^3+(n-1)\gamma^2+\gamma+(b_m-n)\bigr)
\]
gives a monic representation of degree $m+4$ with coefficients in $D_n$.

Next, the polynomials
\[
 \begin{aligned}
 1+(X^2-1)P(X)&=X^4-X^3+(n-1)X^2+X+(1-n),\\
 1+(X^3-X-1)P(X)&=X^5-X^4+(n-1)X^3+(1-n)X+(1-n)
 \end{aligned}
\]
have height $n-1$ and take the value $1$ at $\gamma$. If a monic
representation has degree $k$, multiplying its leading term $X^k$ by either
polynomial keeps the value, raises the degree by four or five, and keeps the
height bound, because the new coefficient of $X^k$ is $1-n$. Every integer at least
twelve is a sum of fours and fives, so every degree $K\ge k+12$ occurs.
\end{proof}

\begin{proof}[Proof of \Cref{prop:quadratic-noncapture}]
\emph{Step 1: a magnifying word.} Let $u_r$ consist of the first $2r$ digits
of $(1,1-n,-1,n-1)^\infty$. Its $\ell$th pair of digits,
$\ell=0,1,\ldots$, acts by $w\mapsto c^2w-(-1)^\ell(c-n+1)$, so induction on
$r$ gives
\begin{equation}\label{eq:quadratic-amplification}
 T_{u_r,c}(1)=(-1)^r\bigl(1-G_r(c)P(c)\bigr),\qquad
 G_r(c)=\sum_{\ell=0}^{r-1}(-c^2)^\ell=\frac{1-(-c^2)^r}{1+c^2}.
\end{equation}
At $c=\gamma$, the orbit of $1$ under these digits is the cycle
$1,\gamma-1,-1,1-\gamma,1$. It is bounded, so the remainder in
\eqref{eq:finite-expansion-remainder} tends to zero; thus
$2\gamma\in E(\gamma,N)$, and $\gamma\in\Mn$ by
\Cref{prop:Mn_characterization}. By \Cref{lem:quadratic-lattice}(a,b), no
word captures $\gamma$, so $\gamma\notin\Theta_{\mathrm{fin}}(n)$.

Near $\gamma$, we have $|c|^2>2$, and $1+c^2$ stays away from zero because
$1+\gamma^2=\gamma+1-n$. Hence there are a neighborhood $W\subset\Xn\setminus\R$
of $\gamma$ and a constant $\Gamma_0>1$ such that
\begin{equation}\label{eq:quadratic-G-bounds}
 \Gamma_0^{-1}|c|^{2r}\le|G_r(c)|\le\Gamma_0|c|^{2r}\qquad(c\in W,\ r\ge1).
\end{equation}
Put $\Gamma=\Gamma_0^2\sup_{c\in W}|c|^2$, so that
$|G_{r+1}(c)|\le\Gamma|G_r(c)|$ on $W$.

\emph{Step 2: entry from an annulus.} Fix $R>\max_{w\in\Pi}|w|$. By
\Cref{lem:quadratic-lattice}(c), each $w$ with $R\le|w|\le\Gamma R$ lies in
$L+\Pi\subset L+Q(\gamma)$ for some $L\in\Lambda$, and $L\ne0$ because
$|w|\ge R$. By compactness, finitely many such points $L_j$ suffice, and by
\Cref{lem:quadratic-monic} they are the values at $\gamma$ of monic polynomials
\[
 F_j(X)=X^K+\epsilon_{j,1}X^{K-1}+\cdots+\epsilon_{j,K},\qquad
 \epsilon_{j,i}\in D_n,
\]
of one common degree $K$. For the word $v_j=(\epsilon_{j,1},\ldots,\epsilon_{j,K})$
and every $\nu\in\C$, we have $T_{v_j,c}(\nu-1)=c^K\nu-F_j(c)$. Put
$R_-=R\rho^{-K}$ and $R_+=\Gamma R_-$. If $R_-\le|\nu|\le R_+$, then
$R\le|\gamma^K\nu|\le\Gamma R$, so $\gamma^K\nu-F_j(\gamma)\in Q(\gamma)$ for
some $j$. The inequalities defining $Q(c)$ are strict and continuous in $c$
and $w$. By compactness of the annulus, there is a neighborhood $V\subset W$ of
$\gamma$ such that
\begin{equation}\label{eq:quadratic-suffix-cover}
 c\in V,\quad R_-\le|\nu|\le R_+
 \quad\Longrightarrow\quad T_{v_j,c}(\nu-1)\in Q(c)
 \quad\text{for some }j,
\end{equation}
and then also $T_{-v_j,c}(1-\nu)\in Q(c)$.

\emph{Step 3: upper bound and interiority.} Let $c\in V$ with
$0<|P(c)|<R_-$. By \eqref{eq:quadratic-G-bounds}, there is a least $r$ with
$|G_r(c)P(c)|\ge R_-$, and $r\ge2$ because $G_1=1$. Then
\[
 R_-\le|G_r(c)P(c)|\le\Gamma|G_{r-1}(c)P(c)|<R_+.
\]
By \eqref{eq:quadratic-amplification}, $T_{u_r,c}(1)=\pm(1-\nu)$ with
$\nu=G_r(c)P(c)$ in this annulus, so $u_r$ followed by $v_j$ or $-v_j$
captures $c$, by \eqref{eq:quadratic-suffix-cover}. Hence
$c\in\Theta_{2r+K}(n)$. Minimality of $r$ gives
$\Gamma_0^{-1}|c|^{2r-2}|P(c)|<R_-$, so
\[
 \kappa_n(c)\le2r+K
 \le\frac{\log(1/|P(c)|)}{\log|c|}+O(1)
 =\frac{\log(1/|c-\gamma|)}{\log\rho}+O(1),
\]
uniformly in $c$. For the last step, note that
$P(c)=(c-\gamma)(c-\overline\gamma)$, where $|c-\overline\gamma|$ is bounded
above and below on $V$, and that replacing $\log|c|$ by $\log\rho$ changes the
quotient by $O\bigl(|c-\gamma|\log(1/|c-\gamma|)\bigr)$. Every $c\ne\gamma$
sufficiently close to $\gamma$ satisfies $0<|P(c)|<R_-$, so it is captured and
lies in $\interior\Mn$ by \Cref{prop:finite-capture-interior}. Since
$\gamma\in\Mn$, a whole neighborhood of $\gamma$ lies in $\Mn$, and
$\gamma\in\interior\Mn$.

\emph{Step 4: lower bound.} Shrink $V$ so that, for $c\in V$,
$|c|\ge\rho_*>1$, the regions $Q(c)$ lie in a fixed disk $|w|\le B_0$, and
every point of $Q(c)$ lies within $C_Q|c-\gamma|$ of $\overline{Q(\gamma)}$; the
last property holds because the four vertices of $Q(c)$ depend smoothly on $c$. Suppose that a word of length $k$ captures $c\in V$.
Running its orbit backward from $w_k(c)\in Q(c)$ by the maps
$w\mapsto(w+d)/c$, $d\in D_n$, gives
$|w_j(c)|\le B_*:=\max\{B_0,(n-1)/(\rho_*-1)\}$ for every $j$. For any word,
the differences $e_j=w_j(\gamma)-w_j(c)$ satisfy $e_0=0$ and
$e_{j+1}=\gamma e_j+(\gamma-c)w_j(c)$, so that, for every $c$,
\begin{equation}\label{eq:quadratic-perturbation}
 |w_k(\gamma)-w_k(c)|\le|c-\gamma|\,\max_{j<k}|w_j(c)|\,\frac{\rho^k-1}{\rho-1}.
\end{equation}
Since $w_k(\gamma)\in\Lambda\setminus\{0\}$ and $w_k(c)\in Q(c)$,
\Cref{lem:quadratic-lattice} gives
\[
 \delta_n\le|w_k(\gamma)-w_k(c)|+C_Q|c-\gamma|
 \le\Bigl(\frac{B_*}{\rho-1}+C_Q\Bigr)|c-\gamma|\rho^k.
\]
Thus every capture word has length
$k\ge\log(1/|c-\gamma|)/\log\rho-O(1)$. With Step 3, this proves
\eqref{eq:quadratic-depth-law} on a sufficiently small disk $U_n\subset V$
about $\gamma$, and Steps 1 and 3 give the remaining assertions.
\end{proof}

\begin{proof}[Proof of \eqref{eq:explicit-degree-radius}]
Let $F$ be monic of degree $d\ge1$ and height at most $n-1$, let $F(c)=0$, and
put $h=|c-\gamma|$. Since $R_{n,d}\le\rho-1$, the bound holds when
$h\ge\rho-1$, in particular for every root with $|c|\le1$. We may therefore
assume that $h<\rho-1$, so that $|c|>1$. Take the word $u$ with $P_u=F$
(\Cref{prop:dyn-poly}); then $w_d(c)=F(c)=0$ and $w_d(\gamma)=F(\gamma)$. By
\Cref{lem:quadratic-lattice}(a), $F(\gamma)$ is a nonzero element of
$\Z[\gamma]$, so $|F(\gamma)|^2=F(\gamma)F(\overline\gamma)$, its norm, is a
positive integer; thus $|F(\gamma)|\ge1$. Running the orbit backward from
$w_d(c)=0$ gives $|w_j(c)|\le(n-1)/(|c|-1)\le(n-1)/(\rho-1-h)$ for $j<d$. As
$n-1=(\rho-1)(\rho+1)$, \eqref{eq:quadratic-perturbation} yields
\[
 1\le\frac{h(\rho+1)(\rho^d-1)}{\rho-1-h},
 \qquad\text{that is,}\qquad
 h\bigl((\rho+1)\rho^d-\rho\bigr)\ge\rho-1.
\]
This is the first inequality in \eqref{eq:explicit-degree-radius}; the second
follows by enlarging the denominator to $(\rho+1)\rho^d$. The same proof
gives \eqref{eq:explicit-degree-radius} with $\gamma_n$ replaced by any non-real root of
$X^2-aX+n$, where $a\in\Z$ and $a^2<4n$.
\end{proof}

\section{Numerical methods and comparison data}\label{supp:experiment-settings}\label{app:computations}
This appendix gives the settings behind the numerical figures and the comparisons of
\Cref{sec:experiments}. The grid searches use floating-point arithmetic, except those for
\Cref{fig:quadratic-noncapture}(b). Series-tail estimates bound the omitted terms of the
enclosure but not every rounding error, so the depths and verdicts of the floating-point
searches are numerical observations. The grid-search and comparison programs stop as soon as
the number of retained words at one depth reaches the per-level word limit, whereas
\Cref{alg:geom}, which the exact reference implementation follows, tests the remaining children
at that depth for entry before returning \Undetermined. None of the controlled comparisons or
local searches reaches its word limit.

The exact checks are separate. The accompanying archive verifies in rational arithmetic the
eighteen root disks and thirty-six lower bounds of Table~\ref{tab:off-lens-cert}, the worked
examples of Appendix~\ref{app:exact-certificates} and their minimum capture depths, the
branching witness, and the interval certificate of Appendix~\ref{supp:low-depth-topology}.
Minimum-depth checks exclude all shorter words, and the branching check tests every candidate
child.

\subsection{Grids and enclosures for different numbers of maps}\label{supp:atlas-settings}
For each $2\le n\le25$, including the six values shown in \Cref{fig:multi-n-atlas}, the
$180\times180$ equally spaced grid, endpoints included, covers
\[
 0\le\Real(c)/\sqrt n\le1.45,\qquad
 0.015\le\Imag(c)/\sqrt n\le1.45.
\]
Samples with $|c|\le1$ or $|c|>1+\sqrt{n-1}+0.08\sqrt n$ are not searched; the latter lie
beyond the radial bound of \Cref{prop:Mn-non-real-disk}. Each search starts at $2c$ and runs to
depth seven with per-level word limit $1200$, using canonical entry inside the lens and
exterior exclusion throughout. Writing $c=\rho e^{i\theta}$ and $\varepsilon=2\times10^{-13}$,
the enclosure keeps
\[
 M=\min\left\{2000,\max\left\{24,
 \left\lceil\frac{-\log\bigl(\varepsilon(\rho-1)\bigr)}{\log\rho}\right\rceil\right\}\right\}
\]
terms of the support series and uses the half-widths
\[
 \widehat V=(N-1)\left(\sum_{j=1}^{M}\rho^{-j}|\sin(j\theta)|
             +\frac{\rho^{-M}}{\rho-1}\right),\qquad
 \widehat S=\frac{(N-1)|\operatorname{Im}c|+\widehat V}{\rho}.
\]
Near the unit circle, the cap of $2000$ terms may leave a tail term larger than $\varepsilon$.

\subsection{Controlled comparison of the two inner regions}
\label{supp:controlled-settings}
This comparison changes only the inner region: the canonical parallelogram
$P_c=\PTrap(c,2n-1)$ or the closed rectangle $R_c$ of \Cref{prop:rectangle-in-trap}.
The three-map sample and the stopping rule are those of Section~\ref{sec:controlled-comparison};
the common inverse tree retains distinct words at coincident endpoints.
The numerical enclosure uses the half-widths of Appendix~\ref{supp:atlas-settings} with $M=80$.
The implementation uses $w=z/2$ and $d=t/2\in D_n$, so that $w_0=c$, $w\mapsto c(w-d)$, the
half-widths are halved, and inverse depths are unchanged. Grid inputs are binary64 numbers;
subsequent calculations use extended precision. Outer comparisons are enlarged and inner
comparisons tightened by the numerical tolerance
\[
 \epsilon_{\mathrm{half}}=10^{-12}
       \bigl(n+\widehat V_{\mathrm{half}}/|y|\bigr),
 \qquad \widehat V_{\mathrm{half}}=\widehat V/2.
\]
This tolerance is not a rigorous bound for all rounding errors. At most twelve words are
retained at one depth in the three-map comparison, so the per-level word limit $16\,384$ does
not affect the reported counts.

\begin{table}[!htbp]
\centering\small
\setlength{\tabcolsep}{7pt}
\caption{Cumulative numerical outcomes on the same $57\,347$ three-map samples.
Each inner count, together with its undetermined count and the common
exterior count, sums to the sample size.}
\label{tab:controlled-depths}
\begin{tabular}{@{}rrrrrr@{}}
\toprule
&\multicolumn{2}{c}{Inner test succeeds}&&\multicolumn{2}{c}{Undetermined}\\
\cmidrule(lr){2-3}\cmidrule(lr){5-6}
Depth &$P_c$&$R_c$&Exterior&$P_c$&$R_c$\\
\midrule
2 & 7\,199 & 3\,441 & 32\,065 & 18\,083 & 21\,841\\
4 &16\,850 &14\,403 & 33\,862 & 6\,635 & 9\,082\\
8 &21\,793 &21\,517 & 34\,756 & 798 & 1\,074\\
12&22\,369 &22\,337 & 34\,882 & 96 & 128\\
18&22\,443 &22\,441 & 34\,899 & 5 & 7\\
\bottomrule
\end{tabular}
\end{table}

For $n=2,3,5,10,20,25$, the second comparison keeps the centers of a
$160\times160$ grid in $c/\sqrt n\in[0,1.1]\times[0,1.1]$ that satisfy
$|c|^2>n$ and $|c|^2+2|\Real c|<2n-1$. The depth limit, support truncation,
tolerance and word limit are unchanged, and at most ten words are retained at
any depth. Table~\ref{tab:controlled-arities} gives the results. Two samples
remain undetermined for $n=2$ with either inner region, and for $n=3$ one
remains undetermined with $R_c$ and none with $P_c$; every other sample
receives a numerical interior or exterior outcome.

\begin{table}[!htbp]
\centering\small
\setlength{\tabcolsep}{7pt}
\caption{Numerical inner detections on six normalized grids. ``Earlier'' is
the percentage detected sooner by $P_c$ among samples detected by both
inner tests by depth $18$.}
\label{tab:controlled-arities}
\begin{tabular}{@{}rrrrrrr@{}}
\toprule
&&\multicolumn{2}{c}{Depth $4$}&\multicolumn{2}{c}{Depth $18$}&Earlier\\
\cmidrule(lr){3-4}\cmidrule(lr){5-6}
$n$&Samples&$P_c$&$R_c$&$P_c$&$R_c$&(\%)\\
\midrule
2&626&50&27&196&196&67.86\\
3&1\,216&355&301&476&475&65.05\\
5&2\,389&1\,071&990&1\,186&1\,186&68.30\\
10&5\,232&2\,658&2\,554&2\,780&2\,780&71.80\\
20&6\,748&2\,326&2\,266&2\,384&2\,384&77.22\\
25&7\,029&2\,160&2\,112&2\,201&2\,201&78.74\\
\bottomrule
\end{tabular}
\end{table}

\FloatBarrier

\subsection{Capture-depth plots and first-level pieces}\label{supp:fields-pieces}
The capture-depth plots of \Cref{fig:n3-parameter-attractors} use $480\times480$ grids, depth
$20$ and per-level word limit $4096$; the parameter overview is restricted to
$[0,2]\times[0,2.4]$. The first-level pieces in \Cref{fig:roots-n4,fig:n3-parameter-attractors}
show finite survival, not certified membership: each piece is tested separately over the
rectangular pixel cells through $24$ inverse steps, and coverage outlines are kept inside
overlaps. The piece views in \Cref{fig:n3-parameter-attractors} are $720$ pixels across, with
at most $20\,000$ digit evaluations per pixel and piece.

\subsection{The twenty-map illustration}
\label{supp:n20-illustration}
Figure~\ref{fig:n20-q1-geometry} samples the window $[0,6.45]^2$ on a $1300\times1300$ grid in
(a), $[4.7,5.59]\times[0,1.9]$ on a $950\times2028$ grid in (b), and
$[4.78,5.12]\times[0.47,0.81]$ and $[4.74,5.08]\times[1.42,1.76]$ on $1600\times1600$ grids in
(c) and (d). Searches start at $2c$, use $A_{39}$ with the canonical inner and outer tests, and
stop at depth $15$ in (a,c) and $16$ in (b,d), with per-level word limit $4096$. No searched
center remains undetermined within these limits. These sampled outcomes certify neither whole
pixel rectangles nor the topology of the displayed boundary. The circle arcs and the crescent
come from the explicit bounds, independently of the samples.

\subsection{The quadratic illustration}\label{supp:quadratic-illustration}
In \Cref{fig:quadratic-noncapture}, $\gamma_3=(1+i\sqrt{11})/2$; following the column titles of
the figure, the neighborhood and the detail are the wider and the narrower close-up in rows
(a) and (b).

\emph{Panel (a).} The overview uses $200\times240$ grid centers on $[0,2]\times[0,2.4]$.
Binary64 searches exclude $|c|\le1$ and use canonical entry, disk and support-strip pruning,
depth $26$ and word limit $2048$. The interval $(1,2]$ is added analytically, and unresolved
samples remain yellow. The neighborhood ($501\times501$ centers) and the detail
($401\times401$ centers) have coordinate half-widths $0.1$ and $0.002$ about $\gamma_3$. Both
use binary64 arithmetic, depth $38$ and word limit $32\,768$; the neighborhood uses disk and
support-strip pruning, and the detail uses disk pruning in the normalized coordinate $z/(2c)$,
with initial state $1$. The analytically classified centers (\Cref{prop:quadratic-noncapture})
are not searched. All $160\,800$ off-center detail samples are captured, at depths $8$--$18$,
and neither local search reaches its word limit.

\emph{Panel (b).} Here the computation is exact. The states $w=z/(2\gamma_3)=(A+B\gamma_3)/D$,
with integers $A$, $B$ and $D>0$, obey $T_d(w)=\gamma_3w-d=(-3B-dD+(A+B)\gamma_3)/D$, and entry
into the trap means $2|A-2B|<5D$ and $3|A+B|<2D$. The full view uses $721\times465$ centers in
$[-7.2,7.2]\times[-7\sqrt{11}/5,7\sqrt{11}/5]$ through depth $28$. Both close-ups use
$401\times401$ centers $z=2\gamma_3+s\zeta$ with $|\Real\zeta|\le1$ and
$|\Imag\zeta|\le\sqrt{11}/3$, for $s=1/2$ and $1/100$, through depth $60$. All off-center detail
samples enter, at depths $8$--$19$; in the neighborhood, one sample besides the center remains
undetermined. Branches with $|w|>3$ are pruned, since $\sum_{j\ge1}2|\gamma_3|^{-j}=1+\sqrt3<3$,
and coincident states are merged by exact equality. The searches have no storage cutoff, and
they certify grid centers, not whole pixels.

\emph{Panel (c).} The original-attractor view uses an $840\times660$ grid (horizontal span
$7.4$, center zero, depth $26$). First-level-piece coverage is tested independently in binary64
arithmetic, with at most $50\,000$ digit evaluations per pixel and piece. Colors identify digits
$-2,0,2$; black outlines are numerical coverage contours. The adjacent orbit diagram uses
$\gamma_3^2=\gamma_3-3$ to give the exact cycle $2\gamma_3\to-6\to-2\gamma_3\to6\to2\gamma_3$.
\FloatBarrier

\section*{Acknowledgments}
We thank Joan Salda\~na and the anonymous referees for their comments and suggestions.
The first author thanks Wolfram Research for its support during the early exploratory phase of this work.
He also thanks the organizers of the 2026 IHP workshop \emph{Integrating Research and Illustration in Number Theory},
\emph{Numeration 2026}, and \emph{Fractals and Related Fields V} for opportunities to present and discuss parts of this work.

\section*{Declaration of generative AI use}
ChatGPT and Codex (OpenAI), Claude (Anthropic), and Gemini (Google) were used to assist exposition, reference checking, programming, and mathematical exploration. Model versions included GPT-6 Astra, Claude Opus 5.5, and Gemini 4 Argon. The authors take responsibility for the mathematical arguments, computations, references and final text.

\section*{Disclosure statement}
No potential conflict of interest was reported by the authors.

\section*{Funding}
This work was supported by the Ministerio de Ciencia, Innovaci\'on y Universidades, Agencia Estatal
de Investigaci\'on, and FEDER, UE, under Grant PID2023-146424NB-I00; and by the Generalitat de
Catalunya under Grant 2021 SGR 00113. The first author was also supported by the Universitat de
Girona and Banco Santander Grant Programme for Researchers in Training (IFUdG 2022--2024).

\section*{Data availability statement}
The accompanying archive, \emph{Interior layers: computational materials}
(version~1.0.0; \doi{10.5281/zenodo.23441158}), contains the numerical data,
figure-generation code, exact-rational certificate verifiers, and reproduction instructions.
Computational methods are described in Appendix~\ref{app:computations}.

\smallskip
\noindent\emph{Interactive companion.}
The explorer \cite{Espigule2026CollinearSoftware} lets readers vary $n$ and $c$,
compare the parameter and dynamical planes, and inspect capture layers and
inverse orbits. It runs in a web browser at
\mbox{\url{https://complextrees.com/collinear/}}.

\end{document}